\documentclass[10pt,reqno]{amsart}
\usepackage[T1]{fontenc}
\usepackage[utf8]{inputenc}
\usepackage[english]{babel}
\usepackage{amsmath}
\usepackage{amsfonts}
\usepackage{amssymb}
\usepackage{amsthm}
\usepackage{bm}
\usepackage{braket}
\usepackage{graphicx}
\usepackage{marginnote}
\usepackage[margin=1.5in]{geometry}
\usepackage[colorlinks=true, pdfstartview=FitV, linkcolor=darkblue, citecolor=darkred, urlcolor=cyan]{hyperref}
\usepackage{dsfont} 
\usepackage{pxfonts}
\usepackage{microtype}

\usepackage{color}
\usepackage{csquotes}
\definecolor{darkred}{RGB}{203,65,84}
\definecolor{darkblue}{RGB}{70,130,180}
\definecolor{brown}{RGB}{139,69,19}

\theoremstyle{plain}
\newtheorem{proposition}{Proposition}[section]
\newtheorem{theorem}[proposition]{Theorem}
\newtheorem{lemma}[proposition]{Lemma}
\newtheorem{corollary}[proposition]{Corollary}
\newcounter{foo}
\newtheorem{theo}[foo]{Theorem}

\theoremstyle{definition}
\newtheorem{definition}[proposition]{Definition}
\theoremstyle{remark}
\newtheorem{rmk}[proposition]{Remark}

\newcommand{\hHn}{\mathbf{H}_+^{2,n}}

\newcommand{\C}{\mathbb{C}}
\newcommand{\D}{\mathbf{D}}
\newcommand{\E}{\mathbf{E}}

\renewcommand{\H}{\mathbf{H}}

\newcommand{\K}{\mathrm{K}}

\newcommand{\N}{\mathbb{N}}
\renewcommand{\P}{\mathbf{P}}
\newcommand{\Pp}{\mathsf{P}}

\newcommand{\R}{\mathbb{R}}
\renewcommand{\S}{\mathbf{S}}

\newcommand{\GG}{\mathcal{G}}

\newcommand{\II}{\mathrm{II}}
\newcommand{\tr}{\operatorname{tr}}
\newcommand{\q}{\textbf{q}}
\newcommand{\g}{\textbf{g}}
\newcommand{\G}{\mathsf{G}}
\newcommand{\A}{\mathsf{A}}

\newcommand{\pd}{\eth}
\newcommand{\CH}{\mathcal{CH}}
\newcommand{\Hh}{\mathsf{H}}

\newcommand{\bR}{{\mathbf E}^{2,n}}

\renewcommand{\epsilon}{\varepsilon}

\newcommand{\de}{{\mathrm d}}

\newcommand{\SO}{\mathsf{SO}}
\newcommand{\PSO}{\mathsf{PSO}}

\newcommand{\GR}[1]{\mathcal{G}(#1)}
\newcommand{\Gr}[1]{\operatorname{Gr}_{2,0}\left(#1\right)}
\newcommand{\Ein}{\mathbf{Ein}}

\newcommand{\Isom}{\operatorname{Isom}}

\newcommand{\Hom}{\operatorname{Hom}}

\newcommand{\Id}{\operatorname{Id}}
\newcommand{\I}{\mathrm{I}}

\newcommand{\dvol}{{\mathrm d}{\operatorname{vol}}}
\newcommand{\area}{\operatorname{area}}

\newcommand{\Stab}{\operatorname{Stab}}
\newcommand{\Fix}{\operatorname{Fix}}

\renewcommand{\span}{\operatorname{span}}

\newcommand{\Hess}{\operatorname{Hess}}

\renewcommand{\leq}{\leqslant}
\renewcommand{\geq}{\geqslant}

\newcommand{\Hn}{\mathbf H^{2,n}}

\newcommand{\bHn}{\partial_\infty \Hn}
\newcommand{\T}{\mathsf T}
\newcommand{\No}{\mathsf N}
\renewcommand{\leq}{\leqslant}
\renewcommand{\geq}{\geqslant}
\newcommand{\Grn}{\GR{\Hn}}

\newcommand{\ms}{\mathsf}
\newcommand{\seqk}[1]{\{#1_k\}_{k\in\mathbb N}}
\newcommand{\seqj}[1]{\{#1_j\}_{j\in\mathbb N}}

\newcommand{\diam}{\operatorname{diam}}
\newcommand{\defeq}{\coloneqq}
\newcommand{\eqdef}{\eqqcolon}
\newcommand{\Fr}{\operatorname{Fr}}
\newcommand{\PE}{\mathbf P(E)}
\newcommand{\Proj}{\mathbf{P}}
\newcommand{\cHn}{{\bar{\mathbf H}}^{2,n}}

\newcommand{\chHn}{{\bar{\mathbf H}}_+^{2,n}}

\newcommand{\sbt}{\,\begin{picture}(-1,1)(-1,-1)\circle*{2}\end{picture}\ }
\renewcommand{\dot}[1]{\overset{\sbt}{#1}}
\renewcommand{\ddot}[1]{\overset{\sbt\sbt}{#1}}

\title[Maximal surfaces in $\Hn$]{Plateau problems for maximal surfaces\\ in pseudo-hyperbolic space\\ ------
\\ Problèmes de Plateau pour les surfaces maximales\\ dans l'espace pseudo-hyperbolique}

\author[F. Labourie]{Fran\c cois Labourie}
\address{Universit\'e C\^ote d'Azur, CNRS,  LJAD,  France}
\email{francois.labourie@univ-cotedazur.fr}

\author[J. Toulisse]{J\'{e}r\'{e}my Toulisse}
\address{
Universit\'e C\^ote d'Azur, CNRS, LJAD,  France}
\email{jtoulisse@univ-cotedazur.fr}

\author[M. Wolf]{Michael Wolf}
\address{School of Mathematics \\
Georgia Institute of Technology}
\email{mwolf40@gatech.edu}

\thanks{F.L., J.T. and M.W. acknowledge support from U.S. National Science Foundation (NSF) grants DMS 1107452, 1107263, 1107367 \enquote{RNMS: Geometric structures And Representation varieties} (the GEAR Network). 
F.L.  and J.T. were supported by the ANR grant DynGeo ANR-11-BS01-013.  F.L. was supported by the Institut Universitaire de France. 
F.L. and M.W. appreciate the support of the Mathematical Sciences Research Institute during the fall of 2019 (NSF DMS-1440140). M.W. acknowledges support from NSF grants DMS-1564374 and DMS-2005551 as well as the Simons Foundation. 
}

\date{\today}

\begin{document}
\maketitle

\begin{abstract}
We define and prove the existence of unique solutions of an asymptotic Plateau problem for spacelike maximal surfaces in the pseudo-hyperbolic space of signature $(2,n)$: the boundary data is given by   
loops on the boundary at infinity of the pseudo-hyperbolic space
which are limits of positive curves. We also discuss a compact Plateau problem.  The required compactness arguments rely on an analysis of the pseudo-holomorphic curves defined by the Gauß lifts of the maximal surfaces.
\end{abstract}

\begin{abstract} Nous définissons un problème de Plateau asymptotique pour les surfaces maximales de type espace dans l'espace pseudo-hyperbolique de signature $(2,n)$ dont le bord à l'infini est donné par des courbes, dites {\em semi--positives}, et qui sont limites de courbes positives. Nous montrons l'existence et l'unicité des solutions correspondantes et discutons le problème de Plateau compact correspondant.  Les arguments de compacité utilisés requièrent l'analyse de courbes pseudo-holomorphes définies par le relevé de  Gauß de surfaces maximales.
\end{abstract}

\setcounter{tocdepth}{1}
\tableofcontents




\section{Introduction}

Our goal is to study Plateau problems in the pseudo-hyperbolic space $\Hn$, which can be quickly described as the  space of negative definite lines in a vector space of signature $(2,n+1)$. As such $\Hn$ is a complete homogeneous pseudo-Riemannian manifold of signature $(2,n)$ and curvature $-1$.

Quite naturally, $\Hn$ bears many resemblances to the hyperbolic plane, which corresponds to the case $n=0$. In particular, generalising the Klein model, $\Hn$ may be described as one of the connected components of the complement to a quadric in the projective space of dimension $n+2$. 

This quadric is classically called the {\em Einstein universe} and we shall denote it by $\bHn$ \cite{charette}. Analogously to the hyperbolic case, the space $\bHn$ carries a conformal metric of signature $(1,n)$ and we will consider it as a boundary at infinity of $\Hn$. Topologically, $\bHn$ is the quotient of $\S^1\times \S^n$ by an involution.

From the Lie group perspective, the space $\Hn$ has $\PSO(2,n+1)$ as a group of isometries and the Einstein space $\bHn$ is  the {\em Shilov boundary} of this rank two Hermitian group, that is the unique closed $\PSO(2,n+1)$-orbit in the boundary of the symmetric domain.

Positive triples and  positivity in the Shilov boundary \cite{Clerc:2003aa} play an important role in the theory of Hermitian symmetric spaces; of notable importance are the {\em positive loops}. Important examples of these are spacelike curves homotopic to $\S^1$ and specifically the {\em positive circles} which are boundaries at infinity in our compactification to totally geodesic embeddings of hyperbolic planes.  Then {\em semi-positive loops} are limits of positive loops in some natural sense (see 
paragraph~\ref{sss:PositiveSet} for precise definitions).

Surfaces in a pseudo-Riemannian space may have induced metrics of variable signatures. We are interested in this paper in {\em spacelike surfaces} in which the induced metric is positive everywhere. Among these are the {\em maximal surfaces} which are critical points of the area functional, for variations with compact support, see paragraph \ref{ss:MaximalSurfaces} for details. These maximal surfaces are the analogues of minimal surfaces in the Riemannian setting. An important case of those maximal surfaces in $\Hn$ are, again,  the totally geodesic surfaces which are isometric to hyperbolic planes.

We refer to the first two sections of this paper for precise definitions of what we have above described only roughly.

Our main Theorem is the following.

\begin{theo}[\sc Asymptotic Plateau problem]\label{t:MainTheo1}
Any semi-positive loop in $\bHn$ bounds a unique complete maximal surface in $\Hn$.
\end{theo}
In this paper, a semi-positive loop is not necessarily smooth. Also note that a properly embedded surface might not be complete and so the completeness condition is not vacuous.

On the other hand, we will show in section~\ref{sec:Submanifolds} that complete spacelike surfaces limit on semi-positive loops in $\bHn$, and so Theorem~\ref{t:MainTheo1} may be understood as identifying semi-positivity as the condition on curves in $\bHn$ that corresponds to complete maximality for surfaces in $\Hn$.  

The uniqueness part of the  theorem is strikingly different from the corresponding setting in hyperbolic space where the uniqueness of solutions of the asymptotic Plateau problem fails in general for some quasi-symmetric curves  as shown by Anderson,  Wang and Huang \cite{Anderson:1983aa,Wang:2012aa,Huang:2015aa}.

As a tool in the theorem above, we also prove the following result, of independent interest, on the Plateau problem with boundary in $\Hn$. The relevant notion for curves is that of {\em strongly positive curves}, and among those the connected set of {\em deformable} ones which are defined in paragraph~\ref{ss:typeofcurves}.

\begin{theo}[\sc Plateau problem]\label{t:MainTheo2}
Any deformable strongly positive  closed curve in $\Hn$ bounds a unique compact complete maximal surface in $\Hn$.
\end{theo}
One of the original motivations for this paper comes from the \enquote{equivariant situation}.  Recall that $\G\defeq \PSO(2,n+1)$  is the isometry group of a Hermitian symmetric space $M$: the maximal compact subgroup of $\G$ has an $\mathsf{SO}(2)$ factor which is  associated to a line bundle $L$ over $M$. Thus a representation $\rho$ of the fundamental group of a closed orientable surface $S$ in $\G$ carries a {\em Toledo invariant}: the Chern class of the pull back of $L$ by any map equivariant under $\rho$ from the universal cover of $S$ to $M$ \cite{Toledo:1979}.  The {\em maximal representations} are those for which the integral of the  Toledo invariant achieves its maximal value. These maximal representations have been extensively studied, from the point of view of Higgs bundles, by Bradlow, Garc{\'i}a-Prada and Gothen \cite{Bradlow:2006} and from the perspective of bounded cohomology, by Burger, Iozzi and Wienhard  \cite{BurgerIozziWienhard}. In particular, a representation is maximal if and only if it preserves a positive continuous curve  \cite{Burger:2005,BurgerIozziWienhard}. Then Collier, Tholozan and Toulisse have shown that there exists a unique equivariant maximal surface with respect to a maximal representation in $\mathsf{PSO}(2,n+1)$ \cite{CTT}. This last result, an inspiration for our work, is now a consequence of Theorem \ref{t:MainTheo1}. 

We note that maximal surfaces in $\Hn$ were also considered in a work by Ishihara \cite{Ishihara}  -- see also Mealy \cite{Mealy}-- and that Yang Li has obtained results for the finite Plateau problems in the Lorentzian case \cite{YangLi}, while the codimension one Lorentzian case was studied by Bartnik and Simon in \cite{Bartnik}. Yang Li's paper contains many references pertinent to the flat case. Neither paper restricts to two spacelike dimensions.

Another motivation comes from the contemplation of  two other rank two groups: $\mathsf{SL}(3,\mathbb R)$ and $\mathsf{SL}(2,\mathbb R)\times \mathsf{SL}(2,\mathbb R)$, where we notice the latter group is isogenic to $\mathsf{PSO}(2,2)$.
\vskip 0.2truecm
\noindent{\em Affine spheres and $\mathsf{SL}(3,\mathbb R)$:} While maximal surfaces are the natural conformal variational problem for $\mathsf{SO}(2,n)$, the analogous problem in the setting of $\mathsf{SL}(3,\mathbb R)$ is that of affine spheres. Cheng and Yau \cite{Cheng:1986},  confirming a conjecture due to Calabi, proved that given any properly convex curve in the real projective space, there exists a unique affine sphere in $\mathbb R^3$ asymptotic to it. That result has consequences for the equivariant situation as well, due independently to Loftin and Labourie \cite{Loftin:2001,Labourie:2006b}. Our main Theorem \ref{t:MainTheo1} may be regarded as an analogue of the Cheng--Yau Theorem: both affine spheres and maximal surfaces (for $\SO(2,3)$)  are lifted as holomorphic curves -- known as cyclic surfaces in \cite{Labourie:2014wh} --  in $\mathsf G/\mathsf K_1$, where $\mathsf G$ is $\mathsf{SL}(3,\mathbb R)$ in the first case and $\mathsf{SO}(2,3)$ in the second,  and $\mathsf K_1$ is a compact torus. Moreover these holomorphic curves finally  project as minimal surfaces in the symmetric space of $\mathsf G$.
\vskip 0.2truecm
\noindent{\em The case $n=1$:} The case of $\mathsf{SL}(2,\mathbb R)\times \mathsf{SL}(2,\mathbb R)$ and maximal surfaces in ${\mathbf H}^{2,1}$ has been extensively studied by Bonsante and Schlenker \cite{BonsanteSchlenker} and written only in the specific case of quasi-symmetric boundaries values -- see also  Tamburelli \cite{Tamburelli2019, Tamburelli2020} for further extensions. Our main Theorem~\ref{t:MainTheo1} is thus a generalization in higher dimension of one of their main results.  {Also in the case of ${\mathbf H}^{2,1}$, we note that Bonsante and Seppi \cite{BonsanteSeppi} have shown the existence, for any $K<-1$, of a unique $K$-surface  extending a semi-positive loop in $\partial_\infty {\mathbf H}^{2,1}$.
\vskip 0.2truecm
\noindent{\em Lorentzian asymptotic Plateau problem:} There is also an interesting analogy with the work of Bonsante, Seppi and Smillie \cite{BonsanteSeppiSmillie1} in which they prove that, for every $H>0$, any regular domain $D$ in the $(n+1)$-dimensional Minkowski space contains a unique entire spacelike surface of constant mean curvature $H$ whose domain of dependence is $D$. Their work corresponds to the non-semisimple Lie group $\SO(n,1) \ltimes \R^{n,1}$. The similarities between the asymptotic behavior of their regular domains and our notion of semi-positive loops in $\bHn$ are striking, in that both only require a non-degeneracy over 2 or 3 points. 

\vskip 0.2truecm
In a subsequent paper \cite{QS}, the first two authors study the analogue of the Benoist--Hulin  result \cite{Benoist:2013vx} for convex geometry and study quasisymmetric positive curves and the relation with the associated maximal surface . In contrast, Tamburelli and Wolf study the case of \enquote{polygonal curves} in the $\mathbf H^{2,2}$ case, whose group of isometries is $\mathsf{SO}(2,3)$ which is isogenic to $\mathsf{Sp}(4,\mathbb R)$ \cite{TamburelliWolf}; there they prove results analogous to Dumas--Wolf \cite{Dumas:2014wz}. One goal in that work is to identify local limiting behavior of degenerating cocompact families of representations.

The proof of Theorem~\ref{t:MainTheo1} follows a natural outline.  We prove the uniqueness portion by relying on a version of the Omori maximum principle; the bulk of the proof is on the existence question. To that end, we approximate a semi-positive loop on $\bHn$ by semi-positive graphs in $\Hn$; as maximal surfaces in our setting are stable, we solve the Plateau problem for these with a continuity method, proving compactness theorems relevant to that situation.  We then need to show that these finite approximations converge, limiting on a maximal surface with the required boundary values. Thus, much of our argument comes down to obtaining compactness theorems with control on the boundary values.  Some careful analysis of this setting allows us to restrict the scope of our study to disks and semi-disks. Then, the main new idea here is to use the {\it Gauß lift} of the surfaces, to an appropriate Grassmannian, which are shown to be pseudo-holomorphic curves. We can then use Schwarz lemmas to obtain 
\begin{enumerate}
	\item first a compactness theorem under a bound on the second fundamental form,
	\item then after a rescaling argument using a Bernstein-type theorem in the rescaled limit $\bR$ of $\Hn$, a uniform bound on the second fundamental form.
\end{enumerate}

We would like to thank specifically Andrea Tamburelli for pointing out the use of Omori Theorem in this setting, Alex Moriani, Enrico Trebeschi, Fanny Kassel and the referee  for pointing out various mistakes in an earlier version, as well as useful comments by Dominique Hulin, Fanny Kassel, Qiongling Li, John Loftin,  Raffe Mazzeo, Anna Wienhard and Tengren Zhang. Helmut Hofer provided crucial references for the pseudo-holomorphic appendix and we would like to especially thank him here.

\subsection{Structure of this article} 
\begin{enumerate}
\item In section \ref{sec:PseudoHyperbolic Geometry}, we describe the geometry of the pseudo-hyperbolic space $\Hn$, and its boundary at infinity, the Einstein universe $\bHn$. There we carefully define positive and semi-positive curves in $\bHn$.
	\item In section \ref{sec:Submanifolds}, we discuss curves and surfaces in $\Hn$. In particular we introduce maximal surfaces and show that they may be interpreted as holomorphic curves. We also discuss spacelike curves and various notions related to them.
	\item In section \ref{sec:Uniqueness}, we prove the uniqueness part of our two main theorems.
	\item In section \ref{sec:MainCompactness}, we prove, using the holomorphic curve interpretation, a crucial compactness theorem for maximal surfaces. We feel this is of some independent interest. 
	\item In section \ref{sec:Specific}, we describe different consequences of our main compactness Theorem, whose formulations we will use in the proof of Theorem~\ref{t:MainTheo1}.
	\item In section \ref{sec:Plateau}, we prove the Plateau Theorem \ref{t:MainTheo2} by the continuity method, relying on the both the stability of the maximal surface and a compactness consequence from section~\ref{sec:Specific}.
	\item In section \ref{s:ProofsMainTheorems} we prove the Asymptotic Plateau Theorem \ref{t:MainTheo1} using the Plateau Theorem \ref{t:MainTheo2}, an exhaustion procedure, and the results in section~\ref{sec:Specific}.
	\item In the Appendices \ref{app:bg} and \ref{app:phol}, we describe the notion of {\em bounded geometry} and prove the relevant results needed for the holomorphic curve interpretation. We expect that last appendix has some independent interest.
\end{enumerate}

\section{Pseudo-hyperbolic geometry}\label{sec:PseudoHyperbolic Geometry}

In this section, we describe the basic geometry of the pseudo-hyperbolic space and its boundary, the Einstein universe. Part of the material covered here can be found in \cite{charette,CTT,DKG}. This section consists mainly of definitions.

\subsection{The pseudo-hyperbolic space}

In this paper, we will denote by $E$ a vector space equipped with a non-degenerate quadratic form $\q$ of signature $(2,n+1)$. The group $\mathsf{O}(E)$ of linear transformations of $E$ preserving $\q$ has four connected components, and we will denote by $\G:=\SO_0(E)$ the connected component of the identity. The group $\G$ is isomorphic to $\SO_0(2,n+1)$.

\begin{definition}
The \emph{pseudo-hyperbolic space} $\Hn$ is the space of negative definite lines in $E$, namely
\[\Hn\defeq\mathbf{P}\left(\{ x\in E\mid~\q(x)<0\}\right)\subset \PE\   \ \ .\]  
\end{definition}

The pseudo-hyperbolic space $\Hn$ is naturally equipped with a signature $(2,n)$ pseudo-Riemannian metric $\g$ of curvature $-1$.
The group $\G$ acts by isometries on $\Hn$ and the stabilizer of a point contains a group isomorphic to   $\SO_0(2,n)$ as an index two subgroup. In particular, $\Hn$ is a (pseudo-Riemannian) symmetric space of $\G$.

\subsubsection{Geodesics and acausal sets}\label{sss:Geodesics}

Complete geodesics are intersections of projective lines with $\Hn$. Any two distinct points $(x,y)$ lie on a unique complete geodesic. We parametrize a geodesic by parallel tangent vectors.

A geodesic $\gamma$, which is the intersection of the projective line ${\mathbf P}(F)$ with $\Hn$, can be  of three types:
\begin{enumerate}
\item \emph{Spacelike geodesics}, when $F$ has signature $(1,1)$, or  equivalently $\q(\dot\gamma)$ is positive.
\item \emph{Timelike geodesics}, when $F$ has  signature $(0,2)$, or equivalently $\q(\dot\gamma)$ is negative.
\item \emph{Lightlike geodesics}, when $F$ is degenerate, or equivalently $\q(\dot\gamma)=0$. \end{enumerate}

A geodesic segment is the restriction of a parametrized complete geodesic to the segment $[0,1]$. Two distinct points $(x,y)$ are extremities of a geodesic segment, which is unique unless the corresponding complete geodesic is timelike (in which case there are exactly two such geodesic segment). 

We say the pair of points $(x,y)$ is {\em acausal} if they are the extremities of a spacelike geodesic segment $\gamma$. We then define its {\em spatial distance} as 
$$
\pd (x,y)\defeq \int_0^1 \sqrt {\q\left(\dot{\gamma}\right)}\  {\rm d} t\ .
$$
A subset $U$ of $\Hn$ is {\em acausal} if every pair of distinct points in $U$ is acausal. 

\subsubsection{Hyperbolic planes} \label{sss:HyperbolicPlanes} A {\em hyperbolic plane $H$} in $\Hn$ is the intersection of $\Hn$ with a projective plane $\mathbf{P}(F)$ where $F$ is a three-dimensional linear subspace of signature $(2,1)$. The spatial distance $\pd$ restricts to the hyperbolic distance on any hyperbolic plane.

A {\em pointed hyperbolic plane $\Pp$} is a pair $(q,H)$ where $H$ is a hyperbolic plane and $q\in H$. A pointed hyperbolic plane is equivalent to the datum of  an orthogonal decomposition $E=L\oplus U \oplus V$ where $L$ is a negative definite line, $U$ a positive definite $2$-plane and $V=(L\oplus U)^\bot$. 

\subsubsection{The double cover}
In the sequel, we will often work with the space
\[\hHn = \{x \in E,~\q(x)=-1 \}  \ \ .\]  
The natural projection $\mathbf{P}: E\setminus\{0\} \to \PE$ restricts to a double cover $\hHn \to \Hn$.

The tangent space $\T_x\hHn$ is canonically identified with $x^\bot$. The restriction of $\q$ to $\T_x\hHn$ equips $\hHn$ with the signature $(2,n)$ pseudo-Riemannian metric such that the cover $\hHn \to \Hn$ is a local isometry. We still denote this metric by $\g$.

Complete geodesics in $\hHn$ are connected components of lifts of complete geodesics in $\Hn$. As in $\Hn$, we parametrize complete geodesics with parallel tangent vectors. A geodesic segment in $\hHn$ is the restriction of a (parametrized) complete geodesic to the segment $[0,1]$. A pair of distinct points $(x,y)$ in $\hHn$ is \emph{acausal} if $x$ and $y$ are joined by a spacelike geodesic segment, and a subset $U$ of $\hHn$ is acausal if any pairwise distinct points of $U$ are acausal.

The incidence geometry of $\hHn$ is more subtle than that of $\Hn$. To describe it, first observe that the preimage of a complete geodesic $[\gamma]$ in $\Hn$ has one connected component if $\gamma$ is timelike and two connected components otherwise. Given two distinct points $x$ and $y$ in $\hHn$, we denote by $[x]$ and $[y]$ their respective image image in $\Hn$. 

We distinguish the following cases:
\begin{enumerate}
\item[\emph{Case 1:}] when $[x]=[y]$, that is if $x=-y$. Then any complete timelike geodesic in $\Hn$ passing through $[x]$ lifts to two geodesic segments between $x$ and $y$. In particular, there are infinitely many geodesic segments between $x$ and $y$. 
\item[\emph{Case 2:}] when $[x]\neq [y]$ and the complete geodesic passing through them is timelike. Then there is a unique geodesic segment between $[x]$ and $[y]$ having a lift with extremities are $x$ and $y$. In particular, there is a unique geodesic segment between $x$ and $y$.
\item[\emph{Case 3:}] If $[x]\neq [y]$ and the geodesic $\gamma$ passing through them is spacelike. Then either $x$ and $y$ lie on the same connected component of the preimage of $\gamma$, in which case there is a unique geodesic segment between $x$ and $y$, or they lie in different connected components in which case there is no geodesic segment between $x$ and $y$.
\item[\emph{Case 4:}] If $[x]\neq [y]$ and the geodesic $\gamma$ passing through them is lightlike. Then either $x$ and $y$ lie on the same connected component of the preimage of $\gamma$, in which case there is a unique geodesic segment between $x$ and $y$, or they lie in different connected components in which case there is no geodesic segment between $x$ and $y$.
\end{enumerate}

The different situations are easily described using the scalar product $\langle x,y\rangle$ of the points $x$ and $y$ associated to the quadratic form $\q$.

\begin{lemma}\label{l:GeodesicTypeAndScalarProduct}
Consider two distinct points $x$ and $y$ in $\hHn$.
\begin{enumerate}
	\item There is a spacelike geodesic segment between $x$ and $y$ (that is, the pair $(x,y)$ is acausal) if and only if $\langle x,y\rangle<-1$ .
	\item There is a unique timelike geodesic segment between $x$ and $y$ if and only if $ \vert \langle x,y\rangle\vert < 1$ .
	\item There is a lightlike geodesic segment between $x$ and $y$ if and only if $ \langle x,y\rangle  = -1$ .
\end{enumerate}
Three points $(x_1,x_2,x_3)$ lies in a hyperbolic plane if and only if for any $i\neq j$ we have $\langle x_i,x_j\rangle<-1$ and \begin{equation}2\langle x_1,x_2\rangle \langle x_1,x_3\rangle \langle x_2,x_3\rangle + \langle x_1,x_2\rangle^2+\langle x_1,x_3\rangle^2+\langle x_2,x_3\rangle^2 < 1 \label{eq:3inHyp}~ .\end{equation}
\end{lemma}

\begin{proof} Items (i), (ii) and (iii) correspond to cases different from \emph{Case 1} described above. In particular, the set  $\{x,y\}$ spans a 2-plane $F$ in which the matrix of the quadratic form is 
given by $$\left(\begin{array}{ll} -1 & \langle x,y\rangle \\ \langle x,y\rangle & -1 \end{array}\right)\ ,$$
whose determinant is equal to $\delta\defeq 1-\vert \langle x,y\rangle\vert^2$.

Item (ii) corresponds to \emph{Case 2} described above. In particular, this happens if and only if the plane $F$ has signature $(0,2)$, that is if and only if $\delta>0$ (the case of signature $(2,0)$ is impossible since $F$ contains negative definite vectors).

Item (i) is a particular situation in \emph{Case 3}, thus a necessary condition is to have $\delta<0$, meaning that $\vert \langle x,y\rangle\vert>1$. The two different connected components of the preimage of the geodesic between $[x]$ and $[y]$ are distinguished by the sign of the function $\langle x,.\rangle$. This sign is negative on the connected component containing $x$.

Item (iii) is a particular situation in \emph{Case 4}, thus a necessary condition is to have $\delta=0$, meaning that $\vert \langle x,y\rangle \vert =1$. Similarly to item (i), the connected component of the preimage of the geodesic between $[x]$ and $[y]$ are distinguished by the sign of the function $\langle x,.\rangle$. 

For the last statement, observe that $(x_1,x_2,x_3)$ lie in a hyperbolic plane if and only if for any $i\neq j$ the points $x_i$ and $x_j$ are joined by a spacelike geodesic segment and the 3-plane $P$ spanned by $x_1,x_2$ and $x_3$ has signature $(2,1)$. Since the subspace of $P$ spanned by $x_1$ and $x_2$ has signature $(1,1)$, then $P$ has signature $(2,1)$ if and only if $\det \left((\langle x_i,x_j\rangle_{1\leq i,j\leq 3} \right) < 0$ which is equivalent to the condition \eqref{eq:3inHyp}.
\end{proof}

Similarly, a (pointed) hyperbolic plane in $\hHn$ is a connected component of a lift of a (pointed) hyperbolic plane in $\Hn$. A pointed hyperbolic plane in $\hHn$ thus corresponds to an orthogonal decomposition $E=L\oplus U \oplus V$ where $L$ is an  oriented negative definite line, $U$ is a positive definite plane and $V=(L\oplus U)^\bot$.

\subsection{Pseudo-spheres and horospheres}\label{ss:pseudosphere}
We describe here the geometry of  pseudo-spheres, and (pseudo)-horospheres  which are  counterparts in pseudo-hyperbolic space of the corresponding hyperbolic notions. 

\subsubsection{Pseudo-sphere} Let $F$ be an $(n+2)$-dimensional real vector space equipped with a quadratic form $\q_F$ of signature $(2,n)$.  The \emph{pseudo-sphere} is 
\[\S^{1,n} \defeq \{ x\in F,~\q_F(x)=1 \}  \ \ .\]  
The pseudo-sphere $\S^{1,n}$ is equipped with a pseudo-Riemannian metric $\g_{\S^{1,n}}$ of curvature $+1$ and  signature $(1,n)$. This metric is invariant under the action of the group $\SO_0(F)$ which is isomorphic to $\SO_0(2,n)$.

\subsubsection{Horosphere} Let us return to the basic case where $E$ is equipped with a signature $(2,n+1)$ quadratic form $\q$. The \emph{null-cone} of $E$ is
\[\mathcal{N}(E) = \big\{ v\in E\setminus \{0\},~\q(v)=0\big\}  \ \ .\]  
Given a point $v\in \mathcal{N}(E)$, the set $P_v = \{x\in E,~\langle x,v\rangle = -1\}$ is a (degenerate) affine hyperplane whose direction is $v^\bot$. The corresponding \emph{horosphere} is
\[\mathcal{H}(v) = P_v\cap \hHn  \ \ .\]  
We also refer to the projection of $\mathcal{H}(v)$ in $\Hn$ as a  horosphere (and denote it the same way).

Given a point $u\in \mathcal{H}(v)$, denote by $F$ the linear subspace of $E$ orthogonal to $u$ and $v$. The restriction $\q_{\vert F}$ of the quadratic form $\q$ to $F$ has signature $(1,n)$. Consider the map
$$f: \left\{\begin{array}{llll}
& F & \longmapsto & E \ , \\
& x & \longmapsto & u+x-\frac{\q_{\vert F}(x)}{2}\cdot v\ . \end{array}\right.$$
One easily checks that $f$ is a diffeomorphism between $F$ and $\mathcal H(v)$.

Moreover, since ${\rm D}_xf(h)=h-\langle x,h\rangle \cdot v$ and $v$ is isotropic, the pull-back by $f$ of the induced metric on $\mathcal H(v)$ (which is induced by $\q$) is equal to $\q_{\vert F}$. Thus $\mathcal{H}(v)$ is isometric to the pseudo-Euclidean space $\E^{1,n}$ of signature $(1,n)$. 

\subsubsection{Horospheres as limits of pseudo-spheres} Let $x$ be a point in $\hHn$ (the picture is similar in $\Hn$). Let 
\[\T^1_x\hHn \defeq \left\{ v\in \T_x\hHn,~\q(v)=1\right\}  \ \ .\]  

Since the restriction of $\q$ to $\T_x\hHn$ has signature $(2,n)$, the space $\T^1_x\hHn$ is isometric to $\S^{1,n}$ and its metric is $\Stab_\G(x)$ invariant. We will thus denote it by $\g_{\S^{1,n}}$.

For $\rho$ positive, the exponential map $\exp_x(\rho.)$ restricts to a diffeomorphism between $\T^1_x\hHn$ and the hypersurface 
\[\beta(x,\rho):=\big\{y\in\hHn,~\pd(x,y)=\rho \big\} \ .\]   
Because the restriction of $\g$ to $\beta(x,\rho)$ is also $\Stab_\G(x)$-invariant, there exists a positive number $\lambda(\rho)$ such that $\exp(\rho.)^* \g = \lambda(\rho) \g_{\S^{1,n}}$. Using the same calculation as in classical hyperbolic geometry, one sees that $\lambda(\rho)=\sinh(\rho)$. As a result, $\beta(x,\rho)$ is an umbilical hypersurface of signature $(1,n)$ whose induced metric has sectional curvature $\sinh^{-2}(\rho)$.

Let $\seqk{\rho}$ be a sequence  of positive numbers tending to infinity, and for any $k$, let $x_k$ be a point in $H_k\defeq\beta(x,\rho_k)$. Observe that $H_k$, being a 
non-degenerate hypersurface,  has a canonical normal framing, as in definition \ref{def:framing}. Let $g_k$ in $\G$ map $\T_{x_k}H_k$ to a fixed vector space $V_0$  in $\T_{x_0}\Hn$.  Then  $g_k (H_k)$ converges to the horosphere  passing through $x_0$ and tangent  to $V_0$.  We will need in Proposition \ref{l:ExhaustionCurves} the fact that this convergence is in  the sense of Appendix \ref{app:ConvPseudoRiem}).

\subsection{Grassmannians}\label{ss:Grassmannians}

\subsubsection{Riemannian symmetric space}\label{ss:RiemannianSymmetricSpace} We summarize some of the properties of the Riemannian symmetric space of $\G$.

\begin{proposition}
The Riemannian symmetric space  of $\G$ is isometric to the Grassmannian $\Gr{E}$ of oriented $2$-planes in $E$ of signature $(2,0)$.
\end{proposition}

\begin{proof} The group $\G$ acts transitively on $\Gr{E}$ and the stabilizer of a point  is isomorphic to $\SO(2)\times \SO(n+1)$ which is a maximal compact subgroup of $\G$. This realizes $\Gr{E}$ as the Riemannian symmetric space of $\G$. \end{proof}

Since the maximal compact subgroup of $\G$ contains $\SO(2)$ as a factor, $\Gr{E}$ is a Hermitian symmetric space.  

The corresponding Kähler structure may be described this way. Let $P$ be a point in $\Gr{E}$:
\begin{itemize}
	\item  the tangent space $\T_P\Gr{E}$ at $P$ is identified with $\Hom(P,P^\bot)$. The {\em Riemannian metric} $h_P(\cdot,\cdot)$ at $P$ is defined for $\varphi\in\Hom(P,P^\bot)$ by 
  $$
 h_P(\varphi,\varphi)\defeq  -\tr(\varphi^*\varphi)\ ,
 $$
where $\varphi^*: P^\bot \to P$ is the adjoint of $\varphi$ using $\q$. Note that since  $\q$ is negative definite on $P^\bot$, we have $\tr(\varphi^*\varphi)\leq 0$. 
\item Since the plane $P$ is oriented, it carries a canonical complex structure $J$: the rotation by angle $\pi/2$.  Precomposition by $J$ defines a {\em complex structure} on $\Hom(P,P^\bot)=\T_P\Gr{E}$, hence a $\G$-invariant almost complex structure on $\Gr{E}$. This almost complex structure is the complex structure associated to the Hermitian symmetric space $\Gr{E}$.  
\end{itemize} 
 
By a theorem of Harish-Chandra (see for instance \cite{Clerc:2003aa}),  $\Gr{E}$ is biholomorphic to a bounded symmetric domain in $\C^{n+1}$.

Note that a point $P$ in $\Gr{E}$ gives rise to an  orthogonal splitting $E=P\oplus P^\bot$. We denote by $\pi_P$ the orthogonal projection from $E$ to $P$. The following lemma is straightforward.

\begin{lemma}\label{l:UniformlyBilipschitz}
Given a compact set $K$ in $\Gr{E}$, there exists a constant $C$, with $C\geq 1$ such that for any $Q$ and $P$ in $K$ and $v\in Q$,
\[C^{-1}\Vert \pi_P(v)\Vert \leq \Vert v\Vert \leq \Vert \pi_P(v)\Vert  \ \ .\]  
\end{lemma}

\begin{proof}
The inequality on the right comes from the fact that $P^\bot$ is negative definite, so $\pi_P$ is length non decreasing.  The inequality on the left comes from the compactness of $K$.
\end{proof}

\subsubsection{Grassmannian of a pseudo-Riemannian space} In this paragraph  $(M,g)$ will be a pseudo-Riemannian manifold of signature $(2,n)$.

The {\em Grassmannian $\GR{M}$ of positive definite 2-planes} in $M$ is the fiber bundle $\pi: \GR{M} \to M$ whose fiber over a point $x\in M$ is the Riemannian symmetric space $\Gr{\T_xM}$.

Observe that, $\pi: \GR{M} \to M$ has a {\em horizontal distribution} given by the parallel transport, giving a splitting
\[\T_{(x,P)}\GR{M} = \T_x M \oplus \Hom(P,P^\bot)=P\oplus P^\bot\oplus \Hom(P,P^\bot)  \ \ .\]  

This splitting allows us to define the {\em canonical Riemannian metric} $g$ on $\GR{M}$ given at a point $(x,P)\in\GR{M}$ by
\[g\defeq\left(g^0\mid_P,-g^0\mid_{P^\bot},h_P\right) \ ,\] 

where  $h_P$ is the Riemannian metric on the fiber described above, and $g^0$ is the metric on $\T_x M$. Let us also define for all positive $\lambda$, the {\em  renormalized metric}
\[g_\lambda \defeq\left(\frac{1}{\lambda} g^0\mid_P,-\frac{1}{\lambda} g^0\mid_{P^\bot},h_P\right) \ .\]

\subsubsection{The Grassmannian of ${\mathbf H}^{2,n}$} When $M=\Hn$, we have already remarked in paragraph~\ref{sss:HyperbolicPlanes} that a point in $\GR{\Hn}$ is identified with an orthogonal splitting $E=L\oplus U \oplus V$ where $L$ is a negative definite line, $U$ a positive definite plane and $V=(L\oplus U)^\bot$. The exponential map thus naturally identifies $\GR{\Hn}$ with the space of pointed hyperbolic planes in $\Hn$. We will later on freely use this identification.

Up to an index two subgroup, the stabilizer of a point $(q,P)$ in $\GR{\Hn}$ is isomorphic to $\SO(2)\times\SO(n)$. The projection
\[\left\{\begin{array}{llll} 
\GR{\Hn} & \longrightarrow & \Gr{E} \ ,\\
L\oplus U\oplus V & \longmapsto & U\ . \end{array}\right.\]
is a $\G$-equivariant proper Riemannian submersion when $\GR{\Hn}$ is equipped with the canonical Riemannian metric.

Similarly, a point in $\GR{\hHn}$ corresponds to a pointed hyperbolic plane in $\hHn$.

\subsubsection{A geometric transition}\label{sec:renogr}

Let $\lambda$ be a positive number. We denote by $\Hn_\lambda$ the space $\Hn$ equipped with the metric $\g_\lambda=\frac{1}{\lambda} \g$ where $\g$ is the metric on $\Hn$. Then we have

\begin{proposition}\label{pro:GeomTrans}
\begin{enumerate}
	\item The Riemaniann manifold $\GR{\Hn_\lambda}$ is isometric to $\GR{\Hn}$ equipped with the normalized metric $g_{\lambda}$.
	\item When $\lambda$ tends to $0$ the Riemaniann manifold $\GR{\Hn_\lambda}$ converges in the sense of Appendix \ref{app:ConvPseudoRiem} to $\GR{\E^{2,n}}$ where $\E^{2,n}$ is pseudo-Euclidean space of signature $(2,n)$. 
\end{enumerate}
\end{proposition}
Observe that, even if the notion of convergence of Riemannian manifolds described in Appendix \ref{app:ConvPseudoRiem} requires the choice of a point, since the manifolds $\GR{\Hn_\lambda}$ and $\GR{\E^{2,n}}$ are homogeneous, this choice is not needed here.
We might write the first item in terms of our notation as stating that the two metric spaces $(\GR{\Hn, \g_{\lambda}}, g)$ and $(\GR{\Hn, \g}, g_{\lambda})$ are isometric.

\begin{proof} The first statement comes from the fact that the metric on $\Hom(P,P^\bot)$ is a conformal invariant. The second statement is standard.
\end{proof}
	
We will call $\GR{\Hn_\lambda}$ the {\em renormalized Grassmannian}.

\subsection{Einstein universe}\label{ss:EinsteinUniverse}

The {\em Einstein universe} is the  boundary of $\Hn$ in $\PE$: 
\[\bHn \defeq \big\{ x\in\PE~,~\q(x)=0\big\}\ .\]
Associated is a compactification:
$$
\cHn\defeq\Hn\cup\bHn.
$$
The group $\G$ acts transitively on $\bHn $ and the stabilizer of a point in $\bHn $ is a maximal parabolic subgroup. As for $\Hn$, we will often discuss the double cover of the boundary at infinity as well as the associated compactification
\begin{eqnarray*}
	\partial_\infty \hHn &\defeq& \big\{ x \in \P_+(E)~,~\q(x)=0\big\}\ , \\
	\chHn&\defeq& \hHn\cup\partial_\infty \hHn\ ,
\end{eqnarray*}
where $\P_+(E)=(E\setminus\{0\})/\R_+$ is the set of rays in $E$. We will consider $\partial_\infty\hHn$ as the boundary of $\hHn$.

\subsubsection{Photons, circles and lightcone}\label{sss:PhotonsCircles}

Let us first define some subsets of $\bHn$.
\begin{enumerate}
\item A \emph{photon} or {\em lightlike line} in $\bHn$ is the projectivization of an isotropic $2$-plane in $E$.

\item A \emph{spacelike circle} (respectively {\em timelike}) is the intersection of $\bHn $ with  the projectivisation of a subspace of signature $(2,1)$ (respectively $(1,2)$). Equivalently, a spacelike circle is the boundary of a hyperbolic plane in $\Hn$.

\end{enumerate}
Observe that two distinct points in $\bHn$ lie either on a photon or span a non-degenerate $2$-plane in $E$. In the second case, we say that $x$ and $y$ are \emph{transverse}.

\subsubsection{Conformal structure}
The tangent space $\T_x\bHn $ is identified with the space $\Hom\left(x,x^\bot/x\right)$. The vector space $x^\bot/x$ inherits a signature $(1,n)$ quadratic form from $\q$ and $E$, providing $\T_x\bHn $ with a conformal class of quadratic form. As a result, $\bHn $ is naturally equipped with a conformal structure $[g_\Ein]$ of signature $(1,n)$.

The conformal structure then allows for the definition of timelike and lightlike vectors and curves in $\bHn $. For instance, photons are lightlike curves, while the spacelike and timelike circles are respectively spacelike and timelike curves in $\bHn$ in terms of the conformal structure. 
\subsubsection{Product structure} \label{sss:ProductStructure}

Let $\Pp$ be a pointed hyperbolic plane in $\hHn$, which as usual corresponds to an orthogonal splitting $E=L\oplus U \oplus V$ where $U$ is a positive definite $2$-plane, $V$ is definite negative and $L$ an oriented negative definite line. Let $W=L\oplus V$ and denote by  $\langle.,.\rangle_U$ and $\langle.,.\rangle_W$ the positive definite scalar product induced by $\pm\q$ on $U$ and $W$ respectively.  Then any isotropic ray $x\in \partial_\infty \hHn$ contains a unique point $(u,w)\in U\oplus W$ with $\langle u,u\rangle_U=\langle w,w\rangle_W=1$. This gives a diffeomorphism 
\[\partial_\infty \hHn\cong \S^1\times \S^n\ ,\]
where $\S^1\subset U$ and $\S^n\subset W$ are the unit spheres. In this coordinate system, the conformal
metric of $\partial_\infty \hHn$ is given by
\[[g_\Ein]=[g_{\S^1} \oplus -g_{\S^n}]~, \]
where $g_{\S^i}$ is the canonical metric on $\S^i$ of curvature 1 (see \cite[Section 2.1]{frances}).

\subsection{Positivity} We now discuss the important notion of positivity in the pseudo-hyperbolic setting.

\subsubsection{Positive triples}\label{sss:PositiveTriples}  Let $\tau$ be a triple of pairwise distinct points in the compactification  $\cHn$ (or in $\chHn$). We call $\tau$ a {\em positive triple} if  it spans a space of signature $(2,1)$. It will be called a {\em{negative triple}} if it spans a space of signature $(1,2)$.  
The positive triple is {\em at infinity} if all three points belong to $\bHn$ (or in $\partial_\infty \hHn$).

Positive triples are (possibly ideal) vertices of hyperbolic triangles in $\Hn$. Given a positive triple $\tau$, we will denote by $b_\tau$ the barycenter of the hyperbolic triangle spanned by $\tau$.

We warn the reader that the terminology {\em positive triples}, though standard, may be confusing:  being a positive triple is invariant under all permutations of the elements.

\subsubsection{Semi--positive loops}\label{sss:PositiveSet} We now define the notion of (semi-)positive loops in the compactification  $\cHn$. The definition for $\chHn$ is similar.

\begin{definition}
Let  $\Lambda$ be a subset  of $\cHn$ homeomorphic to a circle. 
\begin{enumerate}
	\item $\Lambda$ is a {\em positive loop} if any triple of points in $\Lambda$ is positive.
	\item $\Lambda$ is a {\em semi-positive loop} if it does not contain any negative triple, and if $\Lambda$ contains at least one positive triple.
\end{enumerate}
\end{definition}

The next lemma concerns the special case of semi-positive loops in $\bHn$. Recall that photons and transverse points are defined in Paragraph \ref{sss:PhotonsCircles}. 

\begin{lemma}\label{l:semipositivenotphotons}
Let $\Lambda$ be a topological circle in $\bHn$ that does not contain any negative triple. Then $\Lambda$ is a semi-positive loop if and only if it is different from a photon.
\end{lemma}

\begin{proof}
If $\Lambda$ is a photon, it does not contain any positive triple and so is not a semi-positive loop.

Conversely, let us assume that $\Lambda$ does not contain any positive triple. We want to show that $\Lambda$ is a photon. If $\Lambda$ is not a photon, then we can find two transverse points $x,y$ in $\Lambda$. Denote by $U_x$ and $U_y$ the open set of points in $\Lambda$ that are transverse to $x$ and to $y$ respectively. Observe that $U_x$ is contained in a photon, and the same is true for $y$. In fact, if not, we could find 2 points $z,t\in U_x$ such that $x,z,t$ are pairwise transverse. In particular the triple $(x,z,t)$ is positive.

We now claim that $\Lambda\setminus (U_x\cup U_y)$ contains at most 2 points. In fact, the complement of $U_x\cup U_y$ is contained in $(x\oplus y)^\bot$ which has signature $(1,n)$. So any triple of pairwise distinct points in $(x\oplus y)^\bot$ must be negative ($\R^{1,n}$ does not contain any isotropic 2-plane), so $\Lambda\setminus (U_x\cup U_y)$ cannot contain more than two points.

This implies that $\Lambda$ is contained in the union of two non disjoint photons $\phi_1\cup \phi_2$. Since two photons intersect at most in one point, $\phi_1\cup \phi_2$ is homeomorphic to the wedge sum of two circles. The only topological circle embedded in the wedge sum of two circles is  one of the circles. This implies that $\Lambda$ is equal to $\phi_1$ or $\phi_2$ contradicting the existence of a pair of transverse points.
\end{proof}

We have the following

\begin{lemma}\label{l:DisjointBarycenter}
Let $\Lambda$ be a semi-positive loop  in $\cHn$, $\tau$ a positive triple in $\Lambda$ and $b$ a point in the interior of the hyperbolic triangle with vertices $\tau$. Then $b^\bot$ is disjoint from $\Lambda$. In particular, the pre-image of $\Lambda$ in $\chHn$ has two connected components.
\end{lemma}

\begin{proof}
Consider $\tau=(z_1,z_2,z_3)$ and  $b$ as in the proposition. Choose a lift of $b$ in $\hHn$, and lift $z_1,z_2$ and $z_3$ to vectors  in the affine hyperplane $\{x\in E~,~\langle x,b\rangle=-1\}$ (we denote the lift with the same letters). Since $b$ is in the interior of the triangle with vertices $z_1,z_2$ and $z_3$, there exists $t_1,t_2,t_3>0$ such that $b=\sum_{i=1}^3 t_iz_i$.

First observe that we have $\langle z_i,z_j\rangle<0$ for any $i\neq j$. In fact, the 3-plane $P=\span\{z_1,z_2,z_3\}$  splits as $P=\R\cdot b\oplus b^\bot$ with $b^\bot$ positive definite. Since $z_i\in F$ we can write $z_i=b+x_i$ with $x_i\in b^\bot$ and the condition $\langle z_i,z_i\rangle \leq 0$ gives $\langle x_i,x_i\rangle\leq 1$. The fact that $\langle z_i,z_j\rangle<0$ then follows from Cauchy-Schwarz inequality together with the fact that $z_i\neq z_j$.

Consider now a vector $x\in E$ which lifts a point in $\Lambda$. We first claim that  there exists at least one $z_i$ such that $\langle x,z_i\rangle\neq 0$. In fact, if not, $x$ would belong to the space $H$ orthogonal to $\span\{z_1,z_2,z_3\}$. Since $H$ has signature $(0,n)$, $x$  would be negative definite and so the space spanned by $z_1,z_2$ and $x$ would have signature $(1,2)$. This is impossible by semi-positivity.

Then we claim that there is no pair $(i,j)$ such that $\langle x,z_i\rangle<0$ and $\langle x,z_j\rangle >0$. In fact, if this were the case, the matrix of $\q$ in the basis $(z_i,z_j,x)$ would have the form
\[ \left(\begin{array}{ccc}
-\epsilon_1 & -\alpha & -\beta \\
-\alpha & -\epsilon_2 & \gamma \\
-\beta & \gamma & -\epsilon_3	
\end{array}
 \right)~,\]
where $\epsilon_i \geq 0$ and $\alpha,\beta,\gamma>0$. The determinant of such matrix is
\[\Delta = -\epsilon_1\epsilon_2\epsilon_3 +2\alpha\beta\gamma +\epsilon_3\alpha^2+\epsilon_2\beta^2+\epsilon_1\gamma^2~.\]
Since $\span\{z_1,z_2\}$ has signature $(1,1)$, we have $\epsilon_1\epsilon_2-\alpha^2<0$. In particular $\Delta>0$ and $\span\{z_1,z_2,x\}$ has signature $(1,2)$, contradicting semi-positivity. 

We thus find $\langle x,b\rangle =\sum t_i\langle x, z_i\rangle \neq 0$ and $\Lambda$ is disjoint from $b^\bot$. As a result, $\Lambda$ is contained in the affine chart $\P(E)\setminus \P(b^\bot)$ and its preimage in $\P_+(E)$ has two connected components, determined by the sign of the linear form $\langle b,.\rangle$.
\end{proof}

\begin{lemma}
Let $\Lambda$ be a semi-positive loop in $\cHn$. Then
\begin{enumerate}
	\item If $\Lambda$ is contained in $\Hn$ and $x$ is a point in $\Lambda$, then $\Lambda$ is disjoint from $x^\bot$.
	\item If $\Lambda$ is contained in $\bHn$, then any point in $\Lambda$ is contained in a positive triple of $\Lambda$.
\end{enumerate}
\end{lemma}

\begin{proof}
The first item is obvious: if $y\in \Lambda$ is orthogonal to $x$, then since in this case we restrict to $y \in \Hn$, we see that $x\oplus y$ has signature $(0,2)$ contradicting semi-positivity.

From the second item, observe that a triple $(x,y,z)$ in $\Lambda \subset \bHn$ is not positive if and only if $\langle x,y\rangle\langle x,z\rangle\langle y,z\rangle=0$. Denote by $U_x$ the set of points $y$ in $\Lambda$ transverse to $x$, that is so that $\langle x,y\rangle \neq 0$. The set $U_x$ is open and non-empty since from the previous lemma, $\langle z_i,x\rangle \neq 0$ for a positive triple $(z_1,z_2,z_3)$. We just have to find a pair of points $y$ and $z$ in $U_x$ which are transverse to each other (as well as to $x$). This can always be done unless $U_x$ is contained in a photon $\phi$. 

We claim that this is not possible. In fact, if $U_x$ is different from $\Lambda\setminus \{x\}$, then its boundary in $\Lambda$ would contain at least two points, and these points would be in $\phi \cap x^\bot$ which is a single point. If $U_x=\Lambda\setminus \{x\}$, then  $\{x\}= \phi\cap x^\bot$ and $\Lambda = \phi$, which does not contain any positive triple.
\end{proof}

This lemma has the following corollary.

\begin{corollary}\label{c:NegativeScalarProduct}
Let $\Lambda$ be a semi-positive loop  contained either in $\bHn$ or in $\Hn$ and $\Lambda_+$ be a connected component of its preimage in $\chHn$. For any two points $x$ and $y$ in $\Lambda_+$, we have $\langle x,y\rangle\leq 0$.
\end{corollary}

\begin{proof}
If $\Lambda$ is contained in $\Hn$, the first item of the previous lemma implies that if $x\in \Lambda_+$, the linear function $\langle x,.\rangle$ never vanishes. By connectedness of $\Lambda$, the sign of $\langle x,.\rangle$ is constant and must be negative because $\langle x,x\rangle =-1$.

Now assume that $\Lambda$ is contained in $\bHn$ and let $z_1, z_2$ and $z_3$ be vectors in $E$ lifting a positive triple $\tau$ in $\Lambda_+$ whose barycenter lifts to $b=z_1+z_2+z_3$. As remarked in the proof of Lemma \ref{l:DisjointBarycenter}, $\langle z_i,z_j\rangle<0$ for $i\neq j$ and for any $x\in \Lambda_+$ the sign of $\langle x,z_i\rangle$ is independent of $i$ among those $z_i$ with $\langle x,z_i\rangle \neq 0$. Because $\langle b,x\rangle= \sum_i \langle z_i,x\rangle<0$, this sign must be negative. This prove the result when $x$ is contained in a positive triple, and then for every $x$ by the second item of the previous lemma.
\end{proof}

We now consider the special case of semi-positive loops in $\partial_\infty \hHn$. Observe that a loop in $\bHn_+$ is semi-positive if and only if its projection to $\bHn$ is semi-positive.

The two-to-one cover $\bHn_+ \to \bHn$ is nontrivial on each photon in $\bHn$. In particular, any photon in $\bHn$ lifts to a photon segment between two antipodal points in $\bHn_+$. We call a \emph{biphoton} in $\bHn_+$ a topological circle consisting of two photon segments between antipodal points and whose projection to $\bHn$ consists of different photons (in particular, a biphoton is not semi-positive).

We call a map $f$ between metric spaces {\em strictly contracting} whenever $d(f(x),f(y))<d(x,y)$ for $x$, $y$ distinct.

\begin{proposition}\label{p:PositiveCirclesAreGraph}
Let $\Lambda$ be a loop in $\partial_\infty \hHn$ and consider a splitting $\partial_\infty \hHn \cong \S^1 \times \S^n$ associated to a pointed hyperbolic plane.
\begin{enumerate}
	\item The loop $\Lambda$ is semi-positive if and only if it is the graph of a $1$-Lipschitz  map from $\S^1$ to $\S^n$ and not a biphoton, nor a photon.
	\item The loop $\Lambda$ is positive if and only if it is the graph of a strictly contracting map from $\S^1$ to $\S^n$.
\end{enumerate}
\end{proposition}

The proposition will follow from the following

\begin{lemma}\label{l:PositiveTriplesAreGraph}
Let $\tau=(z_1,z_2,z_3)$ be a triple in $\partial_\infty \hHn$ and consider a splitting $\partial_\infty \hHn \cong \S^1 \times \S^n$ associated to a pointed hyperbolic plane. Write $z_i=(u_i,w_i)$ in this splitting.
\begin{enumerate}
	\item\label{it:PositivityCondition1} We have $\langle z_i, z_j\rangle \leq 0$ for all pairs $(i,j)$ if and only if $\tau$ is not negative and $d_{\S^1}(u_i,u_j)\geq d_{\S^n}(w_i,w_j)$ for every pair $(i,j)$.
	\item\label{it:PositivityCondition2} We have $\langle z_i,z_j\rangle <0$ for every $i\neq j$ if and only if $\tau$ is positive and $d_{\S^1}(u_i,u_j)>d_{\S^n}(w_i,w_j)$ for any $i\neq j$.
\end{enumerate}
\end{lemma}

\begin{proof}
The determinant of the matrix with coefficients $\langle z_i,z_j\rangle$ is given by $2\langle z_1,z_2\rangle\langle z_1,z_3\rangle\langle z_2,z_3\rangle$, so the condition on the sign of $\langle z_i,z_j\rangle$ implies the positivity or the the non-negativity of $\tau$.

For the condition on the distances, we use the same notation as in Subsection \ref{sss:ProductStructure}. In particular 
\[ \langle z_i,z_j\rangle = \langle u_i,u_j \rangle_U - \langle w_i,w_j\rangle_W\ . \]
For item \ref{it:PositivityCondition1}, the condition $\langle z_i,z_j\rangle\leq 0$ is thus equivalent to
\[\langle u_i,u_j \rangle_U \leq \langle w_i,w_j\rangle_W\leq 1 \ .\]
Using the formula $\langle x,y\rangle = \cos \left(d_{\S^k}(x,y)\right)$, the previous equation holds if and only if
\[ d_{\S^1}(u_1,u_2) \geq d_{\S^{n}}(w_1,w_2) \ ,\]
and item \ref{it:PositivityCondition1} follows. For item \ref{it:PositivityCondition2}, we replace the non-strict inequalities with strict inequalities.
\end{proof}

\begin{proof}[Proof of Proposition \ref{p:PositiveCirclesAreGraph}]
By Corollary \ref{c:NegativeScalarProduct}, if $\Lambda$ is semi-positive, then $\langle x,y\rangle\leq 0$ for any pair of points in $\Lambda$. In fact, $\Lambda$ is a component of the pre-image of its projection to $\P(E)$.

By the previous Lemma, the projection of $\Lambda$ to the first factor in $\partial_\infty \hHn\cong \S^1 \times \S^n$ must be injective and $\Lambda$ is the graph of a $1$-Lipschitz map from $\S^1\to \S^n$. Conversely, if $f$ is a $1$-Lipschitz map, by lemma \ref{l:PositiveTriplesAreGraph} then the image of $f$ does not contain negative triple: if we had a negative triple $(z_1,z_2,z_3)$ then at least one of the product $\braket{z_i,z_j}$ is positive. The result then follows from Lemma~\ref{l:semipositivenotphotons}.
The second item follows from Lemma \ref{l:PositiveTriplesAreGraph} \ref{it:PositivityCondition2}. 
\end{proof}

We now give three important corollaries.

\begin{corollary}\label{c:PhotonSegment}
Let $\Lambda$ be a semi-positive loop in $\bHn_+$. If $\Lambda$ contains two points $x$ and $y$ on a photon, then it contains the segment of a photon between $x$ and $y$.	
\end{corollary}

\begin{proof}
The semi-positive loop $\Lambda$ is the graph of a $1$-Lipschitz map $f: \S^1 \to \S^n$ by Proposition \ref{p:PositiveCirclesAreGraph}. The points $x$ and $y$ correspond to points $(u,f(u))$ and $(v,f(v))$ with $d_\S^1(u,v)=d_{\S^n}(f(u),f(v))$. 

We first observe that $x$ and $y$ cannot be antipodal: in fact, the pair $(x,y)$ is antipodal if and only if the pairs $(u,v)$ and $(f(u),f(v))$ are. Since $f$ is 1-Lipschitz, it must map the two arcs between $u$ and $v$ to two geodesic arcs between $f(u)$ and $f(v)$. The graph of such a map is either a photon or a biphoton and is thus not semi-positive.

In particular, there is a unique shortest arc segment $[u,v]$ between $u$ and $v$ in $\S^1$ and $f$ maps $[u,v]$ isometrically to an arc of geodesic in $\S^n$. The graph of $f\vert_{[u,v]}$ is the segment of a photon between $x$ and $y$.
\end{proof}

Proposition \ref{p:PositiveCirclesAreGraph} also provides for a nice topology on the set of semi-positive loops in $\bHn$ : we say that a sequence $\seqk{\Lambda}$ converges to $\Lambda_0$ if for any splitting $\partial_\infty \hHn\cong \S^1 \times \S^n$, the sequence $\seqk{f}$ converges $\mathcal{C}^0$ to $f_0$ where $\Lambda_k=\text{graph}(f_k)$ and $\Lambda_0=\text{graph}(f_0)$.

We have the following

\begin{corollary}\label{c:semi-approx}
Every semi-positive loop is a limit of smooth spacelike  positive loops.
\end{corollary}

\begin{proof}
Fix a splitting $\partial_\infty \hHn\cong \S^1 \times \S^n$. By Proposition	 \ref{p:PositiveCirclesAreGraph}, the loop $\Lambda$ is the graph of a $1$-Lipschitz map $f:\S^1\to \S^n$, and so its image is contained in a closed hemisphere $H$ of $\S^n$.

For $t\in[0,2]$, consider the geodesic isotopy $\phi_t: H \to H$ with the property that for any $x$, the path $(\phi_t(x))_{t\in [0,2]}$ is the (constant speed) geodesic starting at $x$ and ending at the center of the hemisphere. Such an isotopy is contracting for $t>0$ and $d_{\S^n}(x,\phi_t(x))\leq t$ (because $H$ has radius $\frac{\pi}{2}<2$). 

Thus for any  positive $\epsilon$, there is a positive  $\delta$, such that the map $f_\epsilon\defeq\phi_{\epsilon}\circ f$ is $(1-2\delta)$-Lipschitz and is at a distance at most $\epsilon$ from $f$. Thus, by density, there is a $(1-\delta)$-Lipschitz smooth map $g$ at a distance at most $2\epsilon$ from $f_\epsilon$. Hence $g$ is at distance at most $\epsilon$ from $f$ and its graph is a smooth positive loop. 
\end{proof}

\subsubsection{Convex hulls}\label{sss:ConvexHull} We want to define the convex hull of a semi-positive loop in $\cHn$. Note that the convex hull of a subset $\Lambda$ of $\P(E)$ is in general not well-defined: one first needs to lift $\Lambda$ to $\P_+(E)$, define the convex hull of the lifted cone as the intersection of all the closed half-spaces containing it, and then project down. The drawback of this construction is that it will in general depend on the lifted cone.

In our case, Lemma \ref{l:DisjointBarycenter} implies that the convex hull of a semi-positive loop $\Lambda$ in $\cHn$ is well-defined and will be denoted by $\CH(\Lambda)$. It has the following properties:

\begin{proposition}\label{p:PropertiesConvexHull}
Let $\Lambda$ be a semi-positive loop contained either in $\bHn$ or in $\Hn$, and let $\Lambda_+$ a connected component of its pre-image in $\P_+(E)$. Then \begin{enumerate} 
	\item\label{it:PropCH1} The convex hull $\CH(\Lambda)$ is contained in $\cHn$.
	\item\label{it:PropCH2} Let $p$ be in the interior of $\CH(\Lambda_+)$ and $q$  in $\CH(\Lambda_+)$, then  $\langle p,q\rangle<0$.
	\item\label{it:PropCH2bis} Assume $\Lambda$ is a positive loop in $\bHn$. If $p$ is a point of $\Hn_+$ lying in $\CH(\Lambda_+)$ and $q$ is a point in $\Lambda_+$, then $\langle p,q\rangle<0$.
	\item\label{it:PropCH3} Let $p$ be  in the interior of $\CH(\Lambda)$, then the set $\Lambda$ is disjoint from $p^\bot$.
	\item\label{it:PropCH4} If $\Lambda$ is contained in $\bHn$ and $p$ is in the interior of $\CH(\Lambda)$,  then any geodesic ray from $p$ to a point in $\Lambda$ is spacelike.
	\item\label{it:PropCH5} If $\Lambda$ is contained in $\bHn$, then the intersection of $\CH(\Lambda)$ with $\bHn$ is equal to $\Lambda$.
\end{enumerate}
\end{proposition}

\vskip 0,2truecm
\noindent{\em Proof of \ref{it:PropCH1}} Any $p\in\CH(\Lambda)$ can be lifted to a vector in $E$ of the form $p_0=\sum_{i=1}^k t_ix_i$ where $t_i>0$ and the $x_i$ are lifts of points in $\Lambda_+$ (actually, from a classical result of Carath\'eodory \cite{caratheodory1}, one can take $k=\dim(E)+2$). From Corollary \ref{c:NegativeScalarProduct} we get that $\q(p_0)\leq 0$.

\vskip 0,2truecm
\noindent{\em Proof of \ref{it:PropCH2}} For any vector $x$ lifting a point in $\Lambda_+$, the linear form $\langle x,.\rangle$ is non-positive on $\Lambda_+$ by Corollary \ref{c:NegativeScalarProduct}. Since $p$ is in the interior of $\CH(\Lambda)$, we have $\langle x,p\rangle<0$. Finally, any point $q\in\CH(\Lambda_+)$ lifts to a vector of the form $\sum_{i=1}^kt_ix_i$ with $t_i>0$ and $x_i\in \Lambda_+$.

\vskip 0,2truecm

\noindent{\em Proof of \ref{it:PropCH2bis}} The vector $p$ has the form $p=\sum_{i=1}^m t_ix_i$ with $x_i$ (lift of rays) in $\Lambda_+$ and $t_i>0$. For any vector $q$ lifting a point in $\Lambda_+$, we have $\langle p,q\rangle = \sum_{i=1}^m t_i\langle x_i,q\rangle \leq 0$ with equality if and only if $\langle x_i,q\rangle=0$ for each $i$. But since $\Lambda$ is positive, the only vector in $\Lambda$ whose scalar product with $q$ is $0$ is $q$ itself. Since $p$ lies in $\Hn_+$, we get $\langle p,q\rangle<0$.

\vskip 0,2truecm

\noindent{\em Proof of \ref{it:PropCH3}} As $\Lambda_+ \subset \CH(\Lambda_+)$, this follows from item \ref{it:PropCH2}.

\vskip 0,2truecm
\noindent{\em Proof of \ref{it:PropCH4}} Let $x$ be a vector in $E$ lifting a point in $\Lambda_+$. By item \ref{it:PropCH2}, the linear form $\langle x,.\rangle$ is negative on $\Lambda_+$ and so strictly negative on the interior of $\CH(\Lambda_+)$. The result follows.

\vskip 0,2truecm
\noindent{\em Proof of \ref{it:PropCH5}} Let $p$ be a vector in $E$ lifting a point in $\CH(\Lambda)\cap \bHn$. Then $p$ can be written $\sum_{i=1}^k t_ix_i$ with $t_i>0$ and $x_i\in\Lambda_+$. The condition $\langle p,p\rangle=0$ thus implies that, either $k=1$ and $p=x_1$, or that all the $x_i$ lie on a common photon. In this case $p$ lies on a segment of a photon which must be contained in $\Lambda$ by Corollary \ref{c:PhotonSegment}.

\section{Graphs, curves and surfaces}\label{sec:Submanifolds}

In this section, we study the differential geometric aspects of curves and surfaces in $\Hn$. We define the notion of \emph{maximal surface} and prove some important properties.

\subsection{Spacelike submanifolds in pseudo-hyperbolic spaces}

Recall that $\g$ denotes the pseudo-Riemannian metric of $\Hn$.

\label{def:bas-curv}
\begin{definition}[\sc Spacelike and acausal]
\noindent
\begin{enumerate}
   \item A submanifold $M$ of $\Hn$ is {\em spacelike} if the restriction of $\g$ to $M$ is Riemannian. Such a submanifold is either a curve or a surface.
	\item A spacelike submanifold $M$ of $\Hn$ is {\em acausal},  if every pair of distinct points in $M$ is acausal.\end{enumerate}		
\end{definition}
\subsubsection{Warped-product structure}  \label{sss:WarpedProduct} 

Let $\Pp=(q,H)$ in $\hHn$  be a pointed hyperbolic plane, associated to  the  orthogonal decomposition $E=q\oplus U \oplus V$ where 
\begin{enumerate}
	\item $q$ is an oriented negative definite line,
	\item $U$ is a positive definite plane, with induced norm $\Vert .\Vert$,  so that $q\oplus U$ defines $H$.
\end{enumerate}

Let  $\D^2\subset U$  be the unit (open) disk and $\S^n\subset W\defeq q+V$ be the unit sphere. The following is proved in \cite[Proposition 3.5]{CTT}.
 \begin{proposition}\label{pro:warped}
 	The map
\begin{equation} \label{e:defnOfPsi}
\Psi: \left\{\begin{array}{rcl}
 \D\times\S^n & \longrightarrow & \hHn\ , \\
  (u,w) & \longmapsto & \left(\frac{2}{1-\Vert u\Vert^2}u,\frac{1+\Vert u\Vert^2}{1-\Vert u \Vert^2}w \right)\ ,\end{array}\right.
\end{equation}
 is a diffeomorphism. Moreover, if $\g$ is the metric on $\hHn$, then
 \begin{equation}\label{e:WarpedProduct}
\Psi^*\g =  \frac{4}{\left(1-\Vert u \Vert^2 \right)^2}g_{\D} - \left(\frac{1+\Vert u\Vert^2}{1-\Vert u\Vert^2} \right)^2g_{\S^n}\ , 
\end{equation}
where $g_{\D}$ and $g_{\S^n}$ are respectively the flat Euclidean metric on the disk and the round metric on the sphere.
 \end{proposition}
 Observe that the parametrization $\hHn\cong \D \times \S^n$ extends smoothly to a parametrization of $\hHn\cup\partial_\infty\hHn$ by $\overline\D\times\S^n$.
 
 The  diffeomorphism $\Psi$  is  called the {\em warped diffeomorphism} and said to define the  {\em warped product structure on $\hHn$}. 
 
If the  preimage of  $q$ under $\Psi$ is $(0,v)$, the preimage of $H$ is  $\D \times \{v\}$.

For any $w$ in $\S^n$, the image of $\big( (0,w),\D\times \{w\}\big)$ is a pointed hyperbolic plane that we call {\em parallel} to $\Pp$. These pointed disks correspond exactly to the set of pointed hyperbolic planes whose projection to $\Gr{E}$ is $U$.

The following nice fact was pointed out to us by the referee:

\begin{lemma}\label{lem:ConfSphere}
Let $\Psi$ be as in Proposition \ref{pro:warped}. Identifying $\D$ with an hemisphere of $\S^2$ using the stereographic projection, the metric $\Psi^*\g$ is conformal to the metric $g_{\S^2}-g_{\S^n}$.
\end{lemma}

\begin{proof}
Consider the function $f$ from $\D\times \S^n$ to $\R$ sending $(u,v)$ to $\left(\frac{1-\Vert u\Vert^2}{1+\Vert u\Vert^2}\right)^2$. We obtain that the metric $f\cdot \Psi^*\g$ is equal to
\[\frac{4}{(1+\Vert u\Vert^2)^2}g_\D - g_{\S^n}~.\]
The result then follows from the fact that $\frac{4}{(1+\Vert u\Vert^2)^2}g_\D$ is the expression of the spherical metric on an hemisphere of $\S^2$ in a chart given by the stereographic projection. 
\end{proof}

\begin{definition}[\sc Warped projection]
The {\em warped projection} is the map  \[\pi_{\Pp}: \hHn \longrightarrow H\ \ ,\]
	corresponding (via $\Psi$) to the projection from $\D\times \S^n$ to $\D\times \{v\}$ and mapping $(u,w)$ to $(u,v)$.

A {\em timelike sphere} is the fiber of $\pi_\Pp$ above $q$, for some pointed hyperbolic plane $\Pp=(q,H)$. It  is the intersection of $\hHn$ with the subspace $W$ of $E$ of signature $(0,n+1)$.
\end{definition}

Note that, given a pointed hyperbolic plane $\Pp=(q,H)$ with warped projection $\pi_{\Pp}$, the preimage by $\pi_{\Pp}$ of a point different from $q$ is not totally gedesic, since its induced metric does not have curvature $-1$.

We then have a fundamental property of $\hHn$:

\begin{lemma}[\sc Projection increases length]\label{l:WarpedProjLengthIncreasing}
The warped projection increases the length of spacelike curves. Moreover if $x_1$ and $x_2$ are two distinct points in the same fiber, then $\langle x_1,x_2\rangle>-1$.
\end{lemma}

\begin{proof}
The fact that the warped projection is length-increasing is a direct consequence of equation (\ref{e:WarpedProduct}). 

If $x_1$ and $x_2$ project onto the same point, then $\Psi^{-1}(x_i)=(u,w_i) \in \D \times \S^n$ for $i=1,2$. Using the expression of $\Psi$, we see that 
\begin{equation}\label{e:ScalarProductInWarpedProduct}
\langle x_1,x_2\rangle = \frac{4\Vert u\Vert^2}{\left( 1-\Vert u\Vert^2\right)^2} - \left(\frac{1+\Vert u\Vert^2}{1-\Vert u \Vert^2} \right)^2\langle w_1,w_2\rangle_W ~,
\end{equation}
where $\langle.,.\rangle_W$ is the  positive definite scalar product induced by $-\q$ on $W=q\oplus V$. Since $w_i\in \S^n$, we have $\langle w_1, w_2\rangle_W<1$, thus
\begin{equation}
\langle x_1,x_2\rangle 
>  \frac{4\Vert u\Vert^2}{\left( 1-\Vert u\Vert^2\right)^2} - \left(\frac{1+\Vert u\Vert^2}{1-\Vert u \Vert^2} \right)^2= -1\ .
\end{equation}
This concludes the proof. 
\end{proof}

\subsubsection{Spacelike graphs}

From now on, all our surfaces are assumed to be connected and smooth up to their boundaries. 

\begin{definition}\label{def:spacelike-graph}
We define
\begin{enumerate}
	\item A spacelike submanifold $M$ of $\hHn$ is a {\em graph} if for any pointed hyperbolic plane, the restriction of the corresponding warped projection is a diffeomorphism onto its image. 
	\item If moreover this diffeomorphism is surjective, $M$ is an {\em entire graph}.
\end{enumerate}
\end{definition}
Observe that a  spacelike graph is always embedded. We now use the definitions of 
paragraph \ref{sss:Geodesics}. As in Lemma \ref{lem:ConfSphere}, we identify $\D$ with an hemisphere $B$ of $\S^2$ using the stereographical projection.
\begin{proposition}\label{p:acausalImpliesGraph}
Let $M$ be a connected spacelike submanifold of $\hHn$.
\begin{enumerate}
	\item If $M$ is acausal  then it is a graph. \label{it:ag1}
	\item If $M$ is a graph, then it is the graph of a $1$-Lipschitz map from a subset $U$ of $B$ to $\S^n$, in any warped product. \label{it:ag2}
\end{enumerate}
\end{proposition}

\vskip 0,2truecm
\noindent{\em Proof of \ref{it:ag1}} Given a pointed hyperbolic plane $\Pp$, the corresponding warped projection $\pi_\Pp$ restricts to a local diffeomorphism on $M$. It follows from Lemma \ref{l:GeodesicTypeAndScalarProduct} that, since $M$ is acausal, we have $\langle x,y\rangle  \geq -1$ for any pair $x,y\in M$. Lemma \ref{l:WarpedProjLengthIncreasing} then implies that the restriction of ${\pi_\Pp}$ to $M$ is injective and thus a diffeomorphism on its image.

\vskip 0,2truecm
\noindent{\em Proof of \ref{it:ag2}} Tangent vectors to the graph of $f$ at $(x,f(x))$ have the form $(u,{\rm D}f_x(u))$, where $u\in \T_x\D$. Using Lemma \ref{lem:ConfSphere}, one sees that $(u,{\rm D}f_x(u))$ is spacelike if and only if
\[\Vert u\Vert^2 - \Vert {\rm D}f_x(u)\Vert^2>0\ , \]
where the norms are computed using the sperical metrics on $B$ and $\S^n$. It implies 
\begin{equation}
\Vert {\rm D}f_x\Vert <1 \ .
\end{equation}
\hfill\qedsymbol

\begin{lemma}\label{l:WarpedProjectionIncreaseDistances}
Let $S$ be a connected spacelike acausal surface and $\Pp$ a pointed hyperbolic plane with associated warped projection $\pi_\Pp$. The restriction of $\pi_\Pp$ from $S$ to $\pi_\Pp(S)$ increases the induced path distances.
\end{lemma}

\begin{proof}
Let $\alpha=\pi_\Pp(a)$ and $\beta=\pi_\Pp(b)$ be points in $\pi_\Pp(S)$ with $a,b\in S$. For any path $\gamma$ between $\alpha$ and $\beta$ in $\pi_\Pp(S)$, its preimage by $\pi_\Pp$ in $S$ is a curve between $a$ and $b$ whose length is less than that of the one of $\gamma$ by Lemma  \ref{l:WarpedProjLengthIncreasing}. Taking the infimum over all path between $\alpha$ and $\beta$ yields the result.
\end{proof}

We have several different notions of boundary:

\begin{definition}[\sc Boundaries of acausal surfaces]\label{d:AsymptoticBoundary}
Let $S$ be an acausal surface in $\Hn$.
\begin{enumerate}
	\item The \emph{total boundary} $\Lambda$ of $S$ is $\overline S\setminus \text{int}(S)$, where $\overline S$ is the closure of $S$ in $\cHn$ and $\text{int}(S)$ is its interior.
 	\item The \emph{finite boundary} of $S$, denoted by $\partial S$, is the intersection of $\Lambda$ with $S$.
	\item The \emph{asymptotic boundary} of S, denoted by $\partial_\infty S$, is the intersection of $\Lambda$ with $\bHn$.
	\item The \emph{free boundary} of $S$ (or \emph{frontier}), denoted by $\Fr(S)$ is the complement of $\partial S\cup \partial_\infty S$ in $\Lambda$.
\end{enumerate}
\end{definition}

We will use the same notation for the corresponding objects in $\hHn$. Note that if the induced metric on $S$ is metrically complete, then $\Fr(S)=\emptyset$. If moreover $S$ is a manifold without boundary, then $\Lambda= \partial_\infty S$.

Given a acausal surface $S$ with induced metric $d_\I$, for any point $x$ in $S$, define $d_I(x,\Fr(S))$ as the supremum over all $R$ so that the closed ball of radius $R$ and center $x$ is complete. We also define the pseudo-distance to the frontier as

\[ \eth(x,\Fr(U))\defeq \inf\{\eth(x,z)\mid z\in \Fr(U)\}\ . \]

\begin{proposition}[\sc Boundary of acausal surfaces]\label{p:BoundaryacausalSurface}
Let $S$ be a closed  spacelike surface with boundary in $\hHn$.  Assume that  $\partial S$ is connected and is a a graph, then $S$ is a graph. 
\end{proposition}

\begin{proof}
Let $\pi$ be a warped projection on a hyperbolic plane $H$.
By assumption $\pi(\partial S)$ is a circle $\gamma$ embedded in $H$. By compactness of $S$, $f(x)\defeq \sharp(\pi^{-1}(x))$ is locally constant on each of the connected component of $H\setminus\gamma$. It follows (by compactness) that $f=0$ on the unbounded component of $H\setminus\gamma$. This implies that $\pi^{-1}\gamma=\partial S$. Hence that  $f=1$ in the (interior) neighborhood of $\gamma$. Thus $f=1$ in the bounded connected component of $H\setminus\gamma$.
\end{proof}
Finally, we describe entire spacelike graphs.

\begin{proposition}[\sc Entire spacelike graph]\label{p:EntireGraph}
Let $S$ be a simply connected spacelike surface without boundary.
\begin{enumerate}
	\item If $S$ is properly immersed or if its induced metric is complete, then $S$ is an entire graph. \label{it:eg1}
	\item If $S$ is an entire graph, then it is acausal. \label{it:eg2}
	\item If $S$ is an entire graph, then it intersects any timelike sphere exactly once. \label{it:eg3}
	\item  If $S$ is an entire graph, then its asymptotic boundary $\partial_\infty S$ is a semi-positive loop. \label{it:eg4}
\end{enumerate}
\end{proposition}

\vskip 0,2truecm
\noindent{\em Proof of \ref{it:eg1}} When the induced metric $g_\I$ on $S$ is geodesically complete, the argument comes from \cite[Proposition 3.15]{CTT}. In that case, the warped projection $\pi_\Pp$ on a pointed hyperbolic plane $\Pp$ is length-increasing. In particular we have
\[ \pi_\Pp^*g_{\H^2}\geq g_\I\ \ .\]
It follows that $\pi_\Pp^*g_{\H^2}$ is also complete. As a result, the restriction of $\pi_\Pp$ to $S$ is a proper immersion, hence a covering, and so a diffeomorphism since $S$ is simply connected.

If $S$ is properly immersed, the result follows from the fact that the warped projection is proper, so its restriction to $S$ is a covering.

\vskip 0,2truecm
\noindent{\em Proof of \ref{it:eg2}} This was proved in \cite[Lemma 3.7]{CTT}.

\vskip 0,2truecm
\noindent{\em Proof of \ref{it:eg3}} Given a timelike sphere $\Sigma$ in $\hHn$ and a point $q\in\Sigma$, the orthogonal of $\Sigma$ at $q$ defines a pointed hyperbolic plane $\Pp$ such that $\Sigma=\pi^{-1}_\Pp(q)$. $S$ is then the graph of a map $f$ and so $\Sigma\cap S=(0,f(0))$.

\vskip 0,2truecm
\noindent{\em Proof of \ref{it:eg4}} From Proposition \ref{p:acausalImpliesGraph}, $S$ is the graph of a 1-Lipschitz map from an hemisphere $B$ in $\S^2$ to $\S^n$. Such a map extends to a 1-Lipschitz map from the equator $\S^1$ to $\S^n$. By Proposition \ref{p:PositiveCirclesAreGraph}, its graph $\partial_\infty S$ is semi-positive unless it is a photon or a biphoton.

To prove that $\partial_\infty S$ is not a photon nor a biphoton, observe that from item \ref{it:eg2} the geodesic from any point $x$ in $S$ to any point $y$ in $\partial_\infty S$ is spacelike. In particular, $\partial_\infty S$ is disjoint from the hyperplane $x^\bot$. Given $\phi$ a photon or a biphoton, $\phi$ contains  a pair of antipodal points $(a,b)$. Either $a$ and $b$ are contained in $x^\bot$ or they lie in different connected components of $E\setminus x^\bot$. In both cases, $\phi$ intersect $x^\bot$. Thus $\partial_\infty S$ is semi-positive.
\hfill\qedsymbol

\begin{rmk}
We return briefly to Definition~\ref{def:spacelike-graph}.	 If for a spacelike submanifold of $\hHn$, there is at least one pointed hyperbolic plane for which the warped projection is a diffeomorphism onto its image, then the proof of \ref{it:eg2} above shows that the submanifold is acausal.  Then Proposition~\ref{p:acausalImpliesGraph} implies that the submanifold is a graph over every pointed hyperbolic plane and is thus, by Definition~\ref{def:spacelike-graph}, a spacelike graph.
\end{rmk}

Recall that   $\T_xf$ denotes the tangent map of $f$ at $x$.

\begin{proposition}\label{pro:diamwarp}\label{pro:diamwarp2} 

Let  $\mathsf{Q}=(q,Q)$ and $\Pp=(p,P)$ be pointed hyperbolic planes. Let $\varphi_\Pp$ be the restriction of the warped projection $\pi_\Pp$ to $Q$. Assume that  
$$
d_\GG(\T_pP,\T_qQ)<R\ .
$$ Then 
\begin{enumerate}
\item For each such positive constant $R$, there exists a positive constant $c$ so that $$
\Vert \T_q\varphi_\Pp\Vert\leq c\ ,
$$
\item For any $b$ larger than $1$, there exists such a positive constant $R$ so that 
$$
b^{-1}\leq \Vert \T_q\varphi_\Pp\Vert\leq b\ .
$$
\end{enumerate}
\end{proposition}

\begin{proof}
Since hyperbolic planes have complete induced metrics, for any $\Pp$ and $\mathsf{Q}$ as in the proposition, $\varphi_\Pp$ is a global diffeomorphism and so $\Vert \T_q\varphi_\Pp\Vert>0$. The result then follows from the fact that $\GR{\Hn}$ is locally compact and that $\varphi_\Pp=\Id$ when $\Pp=\mathsf{Q}$.
\end{proof}

\subsection{Strongly positive curves}\label{ss:typeofcurves}
Recall from \ref{sec:renogr} that $\Hn_\lambda$ is the space $\Hn$ equipped with the metric $\g_\lambda=\frac{1}{\lambda} \g$.

For a curve $\gamma$ and $x\in\gamma$, the {\em osculating plane}, denoted $\T^{(2)}_x\gamma$ is (given a parametrization of $\gamma$ so that $\gamma(t_0)=x$) the vector space generated by $\dot\gamma$ and $\nabla_{\dot\gamma}\dot\gamma$. The osculating plane is independent of the parametrization.
 We introduce the following properties of curves  in $\Hn$ which are further refinements of being spacelike and positive.

\begin{definition}[\sc Strongly positive curves]\label{d:TypeOfCurves} A smooth curve $\gamma$ in $\Hn_\lambda$ is \emph{strongly positive} if
\begin{enumerate}
	\item the curve $\gamma$ positive,
 	\item 	 for every point $x$ in $\gamma$, the osculating plane $\T^{(2)}_x\gamma$ has dimension $2$ and is spacelike.
	\item for any pair of disjoint points $x$ and $y$ in $\gamma$, the totally geodesic space containing $x, y$ and the tangent vectors in $\T_y\gamma$ is spacelike.
\end{enumerate}
\end{definition}

It is important to remark that since the lift of a connected positive curve to $\hHn$ has two connected components, we can use any of these to make sense of warped projection, graphs and so on. This fact will be used in the sequel.

\subsubsection{Unpinched curves}\label{sss:unpinched}
Given an acausal curve $\gamma$ in $\Hn$, we can define two distances on $\gamma$: the (extrinsic) spatial distance $\eth$ (see paragraph \ref{sss:Geodesics}) and the distance $d_\gamma$ along $\gamma$. The following notion is a comparison between those two.

\begin{definition}[\sc Unpinched curves]\label{d:Unpinched}
An acausal curve $\gamma$ is called \emph{unpinched} (or \emph{$\delta$-unpinched}) if there exists $\delta>0$ such that for all $x$, $y$ in $\gamma$,
$$
\frac{\eth(x,y)}{d_\gamma(x,y)}\leq \frac{1}{2} \implies \eth(x,y)>\delta\ .
$$
A sequence $\seqk{\gamma}$ is \emph{uniformly unpinched} if there is a $\delta>0$ such that any $\gamma_k$ is $\delta$-unpinched.
\end{definition}

Observe also that if $\gamma$ is  $\delta$-unpinched in $\Hn_\lambda$, it is also $\delta$ unpinched for $\Hn_{\mu}$ for $\mu\leq\lambda$.

\subsubsection{Angular width}\label{sec:angwidth}
Given a curve $\gamma$, we  denote by  $\gamma^{(3)}$  the set of pairwise distinct triples of points in $\gamma$ and by $\text{Gr}_{2,1}(E)$ the set of hyperbolic planes in $\Hn$. Assume $\gamma$ is strongly positive and consider the map
\[\Phi_\gamma:\left\{\begin{array}{llll}
 & \gamma^{(3)} & \longrightarrow & \text{Gr}_{2,1}(E) \\
& (x,y,z) & \longmapsto & x\oplus y\oplus z
\end{array}\right.~,
 \]
Since $\gamma$ is smooth, we have
\[\lim_{\substack{x_1,x_2\to x \\ x_1\neq x_2}} (x_1\oplus x_2)  = x\oplus \T_x\gamma~\ ,\ ~\lim_{\substack{x_1,x_2,x_3\to x \\ x_1\neq x_2\neq x_3}} (x_1\oplus x_2 \oplus x_3) = x\oplus \T_x^{(2)}\gamma~.\]
In particular, $\Phi_\gamma$ extends to a continuous map $\overline\Phi_\gamma: \gamma^3 \to \text{Gr}_{2,1}(E)$ where for $x\neq y$ we have
\[\overline \Phi_\gamma(x,y,y) = x\oplus y\oplus \T_y\gamma~\ ,\ ~\overline \Phi_\gamma(x,x,x)=x\oplus \T^{(2)}_x\gamma~. \]

Let $\gamma$ be a strongly positive curve in $\Hn$. Let $x_0$ be a point in $\gamma$. Consider the codimension 2 subspace $F\defeq  (x_0\oplus \T_{x_0}\gamma)^\bot$ of $E$.  Observe that $F$ has signature $(1,n)$. The projectivization of the orthogonal projection from $E$ to $F$ defines a map
\[\pi_0 : \P(E)\setminus \P(F^\bot) \to \P(F)~.\]
We have
\begin{proposition} Let $\gamma$, $x_0$, and $\pi_0$ be as above, and let $y$ be a point in $\cHn$.  
	If $V\defeq x_0\oplus y\oplus \T_{x_0}\gamma$ has signature $(2,1)$, then $\pi_0(y)$ is a spacelike line in $F$.
\end{proposition}

\begin{proof}
	The vector $\pi_0(y)$ is in the orthogonal in $V$ of the space $W$ generated by $x_0$ and $\T_{x_0}\gamma$; but $W$ has signature $(1,1)$, thus $\pi_0(y)$ is spacelike.
\end{proof}

According to this proposition, $\pi_0$ maps points in $\gamma\setminus\{x_0\}$ to positive definite lines in $F$. If we identify the set of positive definite lines in $F$ with the $n$-dimensional hyperbolic space $\H^n$, we obtain a curve $\pi_{x_0}(\gamma)\subset \H^n$ that we call {\em the angular projection} of $\gamma$ at $x_0$.

\begin{definition}[\sc Angular width]\label{d:AngularWidth}
The {\em angular width} of a compact strongly positive curve $\gamma$ is 
\[w(\gamma)\defeq\sup \left\{ \diam(\pi_{x_0}(\gamma))\mid ~x_0\in \gamma\right\}\ ,\]
where the diameter is computed in $\H^n$.
\end{definition}

\begin{proposition}[\sc Angular width and spacelike surfaces]\label{pro:bd-ang}
Let $\gamma$ be strongly positive curve in $\hHn$. Let $S$ be a spacelike surface with $\partial S\subset\gamma$. Assume that $S$ is included in the convex hull of $\gamma$. Then
for all points $y$ in $\partial S$, we have
$$
d(\T_yS,\T^{(2)}_y\gamma)\leq w(\gamma)\ .
$$
\end{proposition}

\begin{proof}  Let $\gamma$, $x_0$, and $\pi_0$ be as above. Since $\pi_0$ is a linear map and thus preserves convex hulls, $\pi_0(S)$ is included in the convex hull of $\pi_0(\gamma)$. Let $c$ be a curve in $S$ starting from $x_0$ that is orthogonal to $\partial S$. We parametrize $c$ by arc length so that $c(0)=x_0$. Observe first that
$$
\lim_{s\to 0}\pi_0(c(s))=\pi_0(\dot c(0))\ .
$$ 
It follows that $\dot c(0)$ belongs to the convex hull (in $\Hn$) of $\pi_0(\gamma)$. Similarly
$$
\lim_{s\to 0}\pi_0(\gamma(s))=\pi_0(n_0)\ .
$$
where $n_0$ and $T_{x_0}\gamma$ generates $\T^{(2)}\gamma$. It follows that 
$$
d_\GG(\T_{x_0}S,\T^{(2)}_{x_0}\gamma)=d_{{\bf H}^n}(\dot c(0),n_0)\leq w(\gamma)\ .
$$
This concludes the proof.	
\end{proof}

\subsubsection{Deformation of strongly positive curves}

We now introduce the class of curves for which we prove a finite Plateau problem.

\begin{definition}\label{defi:Deformable}
A \emph{deformation} of a strongly positive curve $\gamma$ in $\Hn$ is an isotopy $\{\gamma_t\}_{t\in[0,1]}$ with $\gamma=\gamma_1$ such that
\begin{enumerate}
	\item every curve $\gamma_t$ is strongly positive,
	\item the curve $\gamma_0$ lies in a hyperbolic plane.
\end{enumerate}
A strongly positive curve admitting a deformation is called \emph{deformable}.
\end{definition}
Observe that by compactness of the isotopy, if $\{\gamma_t\}_{t\in[0,1]}$ is a deformation, the angular width $w(\gamma_t)$ is uniformly bounded.

\subsection{Maximal surfaces}

\subsubsection{Second fundamental form}

Consider a spacelike embedding $u: S \hookrightarrow M$, where $(M,\g)$ is a pseudo-Riemannian manifold of signature $(2,n)$ and $S$ a surface. The pull-back bundle $u^*\T M$ splits orthogonally as 
\[u^*\T M=\T S\oplus \No S\ ,\]
where the {\em normal bundle} $\No S$ is the orthogonal of the tangent bundle $\T S$. We denote their induced metric by $g_\I$ and $g_N$ respectively. Observe that $g_\I$ is positive definite, while $g_N$ is negative definite.

We recall that  {\em second fundamental form} $\II$,  which is a symmetric 2-tensor on $S$ with values in $N$,  and the {\em shape operator} $B$ which is a 1-form on $S$ with values in   $\Hom(\No S,\T S)$  are given by 
\[g_\I\left(Y,B(X)\xi \right)=g_N\left(\II(X,Y),\xi \right) = \g\left(\nabla_X Y,\xi\right)\ ,\] 
where  $X$ and $Y$ are  vector fields along $S$,  $\xi$  is a section of the normal bundle and $\nabla$ the Levi-Civita connection on $M$. For instance the second fundamental $\II_0$ form on $\T_x\Hn$, where we see $\Hn$ locally isometrically embedded in $E$ is given by 
$$
\II_0(u,u)=\g(u,u)\ x\ .
$$
Thus if $\gamma$ is a geodesic in $\Sigma\subset\Hn\subset E$, we have
\begin{eqnarray}
\left.\frac{\rm d^2}{{\rm d t}^2}\right\vert_{t=s}\gamma(t)=\II(u,u)+\g(u,u)\ x\ .	\label{eq:geod}
\end{eqnarray}
The {\em norm} of the second fundamental form $\II$ is 
\begin{equation}\label{e:NormOfII}
\Vert \II\Vert ^2 \defeq - \max_{\vert v\vert=1} \sum_{i=1,2} \g\left(\II(v,e_i),\II(v,e_i) \right)\ ,	
\end{equation}
where $(e_1,e_2)$ is an orthonormal basis of $\T_x S$. 

If $M_\Lambda$ denotes the manifold $M$ equipped with the metric $\g_\lambda= \frac{1}{\lambda}\g$, we have the following:

\begin{lemma}\label{l:RescalingSecondFundamentalForm}
Let $S$ be a spacelike surface in $M$, with fundamental form $\II$. The second fundamental form $\II_\lambda$ of $S$ in $M_\lambda$ satisfies $\Vert \II_\lambda\Vert^2=\lambda\Vert \II\Vert^2$.	
\end{lemma}

\begin{proof}
 Observe  that a  unit vector with respect to $\g_\lambda$ has the form $\sqrt\lambda v$, where $v$ is a unit vector for $\g$. The result then follows from tracking the effect of the subsequent substitutions in \eqref{e:NormOfII}, after observing that the Levi-Civita connection is a conformal invariant.
\end{proof}

\begin{definition}
The \emph{mean curvature}  is the normal vector field $$\mathsf H\defeq\tr_{g_\I}\left(\II\right)=\II(e_1,e_1)+\II(e_2,e_2) \ .$$
\end{definition}

\subsubsection{Variation of the area}

Given a spacelike embedding $u: S \hookrightarrow M$, one can define the \emph{area functional}, associating to any compact set $K$ in  $S$ the number
$$\mathcal{A}_K(u) \defeq  \int_K \dvol_{g_\I}\ ,$$
where $\dvol_{g_\I}$ is the volume form of $g_\I$. The following is classical, but we include the proof for the sake of notation.

\begin{lemma}
A spacelike surface $u_0: S \hookrightarrow M$ is a critical point of the area functional if and only if $\mathsf H=0$.
\end{lemma}

\begin{proof}
Let  $\xi$ be a normal vector field with compact support. Let $\{u_t\}_{t\in (-\epsilon, \epsilon)}$ be a smooth deformation of $u_0$ so that $t\to u_t(x)$ are geodesics with initial tangent vector $\xi$. For $\epsilon$ small enough, the image $u_t(S)$ is spacelike. We respectively denote by $g_t,B_t$ and $\II_t$ the induced metric, shape operator and second fundamental form of $u_t$.

Let $G$ and $\nabla$ be respectively the pull-back of the metric and connection of $M$ by $U:(t,x)\to u_t(x)$. The metric $G$ restricts to $g_t$ on $S\times \{t\}$. We have
\[\partial_t G(X,Y) = G(\nabla_\xi X,Y) + G(X,\nabla_\xi Y) = G(\nabla_X\xi,Y)+G(X,\nabla_Y \xi)\ .\]
Thus, restricting to $S\times\{0\}$, we obtain
\begin{equation}\label{e:FirstVariationMetric}
\dot{g}_0(X,Y):=\frac{d}{dt}_{\vert_{t=0}} g_t(X,Y) = -2g_0\left(B_0(X)\xi,Y\right) = -2 g_N\left(\II_0(X,Y),\xi\right).
\end{equation}
In particular, if $(e_1,e_2)$ is an orthonormal framing of $(\T S,g_0)$, and $(e^1,e^2)$ its dual, we have 
\[\dvol_{g_t}= \det\left(\Id -2t \left(g_N\left(\II_0(e_i,e_j),\xi\right) \right)_{i,j=1,2} +o(t) \right) e^1\wedge e^2~.\]
It follows that 
\begin{equation}\label{e:FirstVariationAreaForm}
\dot{\dvol}_{g_t} =  -2g_N(\mathsf H,\xi)\dvol_{g_0}\ .
\end{equation}
Thus  $\dot{\mathcal A}_K(u_0) = -2\int_K g_N(\mathsf H,\xi)\dvol_{g_0}$ and the result follows. 
\end{proof}

We now compute the second variation of the area functional.

\begin{proposition}\label{prop:SecondVariationAreaForm}
Given a spacelike surface $u_0: S \hookrightarrow M$ with $\mathsf H=0$, a non-zero normal deformation along $\xi\in\Omega^0(S,\No S)$ and $K\subset S$ a compact subset, we have
$$
\ddot{\mathcal A}_K(u_0) = \int_K W(\xi)\ \dvol_{g_0}\ ,
$$
where $W(\xi)=2\tr_{g_0}(Q_\xi)$ for the symmetric tensor $Q_\xi$ given by
\[Q_\xi(X,Y)=g_0(\left(R_0(\xi,X)\xi,Y\right)-g_0(B_0(X)\xi,B_0(Y)\xi)+g_N\left(\nabla_X^N\xi,\nabla_Y^N\xi \right) \ .\]
Here $R_0$ is the Riemann curvature tensor of $u_0^*D$ and $D$ the Levi-Civita connection of $M$.
\end{proposition}

\begin{proof}
We use the same notation as the previous proof. By equation (\ref{e:FirstVariationAreaForm}) when $\mathsf H=0$, we have
\[\ddot{\dvol}_{g_0} = -2 g_N( \nabla_\xi \mathsf H,\xi) \dvol_{g_0}=W(\xi)\dvol_{g_0}\ ,\]
where $W(\xi)=2\tr_{g_0}(Q_\xi)$ for the symmetric tensor $Q_\xi$ defined by 
\[Q_\xi(X,Y)=-g_0\left(\nabla_\xi B_0(X)\xi,Y\right)\ .\]
Our goal is to compute $Q_\xi(X,Y)$.\\

One the one hand, we have
\begin{eqnarray*}
\partial_t G(\nabla_X\xi,Y) & = & G(\nabla_\xi\nabla_X\xi,Y)+G(\nabla_X\xi,\nabla_\xi Y) \\
& = &  G(\nabla_\xi\nabla_X\xi,Y)+G(\nabla_X\xi,\nabla_Y \xi) \\
& = &  G\left(R(\xi,X)\xi,Y\right) +G\left(\nabla_X\nabla_\xi\xi,Y\right)+G\left(\nabla_{[\xi,X]}\xi,Y\right)+G(\nabla_X\xi,\nabla_Y\xi) \\
& = &  G\left(R(\xi,X)\xi,Y\right) + G\left(\nabla_X\xi,\nabla_Y\xi\right)\ ,
\end{eqnarray*}
where $R(a,b)c=\nabla_a\nabla_bc-\nabla_b\nabla_ac-\nabla_{[a,b]}c$.

Restricting to $S\times \{0\}$, we obtain
\[\frac{d}{dt}_{\vert_{t=0}} g_t(-B_t(X)\xi,Y) = R_0(\xi,X,\xi,Y)+g_0\left(B_0(X)\xi,B_0(Y)\xi\right)+g_N\left(\nabla^N_X\xi,\nabla^N_Y\xi\right)\ \ .\]
On the other hand, using equation (\ref{e:FirstVariationMetric}), we thus get
\begin{eqnarray*}
\frac{d}{dt}_{\vert_{t=0}} g_t(-B_t(X)\xi,Y) & = & \dot{g_0}(-B_0(X)\xi,Y) + g_0(-\nabla_\xi B_0(X)\xi,Y) \\
& = & 2g_0(B_0(X)\xi,B_0(Y)\xi)+Q_\xi(X,Y)\ .
\end{eqnarray*}
This gives
\begin{eqnarray}
	Q_\xi(X,Y) = R_0(\xi,X,\xi,Y)-g_0\left(B_0(X)\xi,B_0(Y)\xi\right)+g_N\left(\nabla^N_X\xi,\nabla^N_Y\xi\right)\ . \label{eq:2varA}
\end{eqnarray}

\end{proof}

\begin{corollary}\label{c:StabilityMaximalSurfaces}
If $u: S \hookrightarrow \Hn$ is a spacelike surface with $\mathsf H=0$, and let $\xi$ be a non-zero normal deformation supported on a compact $K\subset S$. Then the second variation of the area satisfies
$$\ddot{\mathcal A}_K(u)\leq 4 \int_K g_N(\xi,\xi) \dvol_{g_0} <0~.$$
\end{corollary}
\begin{proof}
Because $\Hn$ has curvature $-1$, we have
\[g_0\left(R_0(\xi,X)\xi,Y\right) = g_N(\xi,\xi)g_0(X,Y)\ .\]
Thus, we obtain
$$\tr_{g_0}(Q_\xi)=2g_N(\xi,\xi) + R~,$$
where $R=\tr_{g_0}\big(-g_0(B_0(.)\xi,B_0(.)\xi)+g_N(\nabla^N_.\xi,\nabla^N_.\xi)\big)\leq 0$. This concludes the proof.
\end{proof}

\subsubsection{Maximal surfaces}\label{ss:MaximalSurfaces} Corollary \ref{c:StabilityMaximalSurfaces} motivates the following definition

\begin{definition}
A spacelike immersion from a surface $S$ to $M$ is a \emph{maximal surface} if $\mathsf H=0$.
\end{definition}

In the sequel, we will denote by $\Sigma$ a maximal surface, and $S$ any surface.

Corollary \ref{c:StabilityMaximalSurfaces} implies the stability of maximal surfaces in $\Hn$ and $\hHn$: given a maximal surface $u$, there is no non-zero compactly supported normal deformation with $\ddot{\mathcal A}(u)=0$.

Calculating the tangential part of the curvature tensor of $\nabla$, one obtains the following

\begin{proposition}\label{p:GaussEquation}
If $\Sigma$ is a maximal surface in $M$ and $P$ is a tangent plane to $\Sigma$  equipped with an orthonormal frame  $(e_1,e_2)$, then we have
\[K_\Sigma(P) = K_M(P) - q_N(\II(e_1,e_1))- q_N(\II(e_1,e_2))\ ,\]
where $K_\Sigma$ and $K_M$ are the sectional curvatures of $\Sigma$ and $M$, respectively, and $q_N$ is the (negative definite) quadratic form on $N$. 

In particular if $M=\Hn$, since $q_N$ is negative definite, $K_\Sigma\geq -1$. 
\end{proposition}

\subsubsection{Convex hull}
Recall from paragraph \ref{sss:ConvexHull} that any semi-positive loop in $\cHn$ has a well-defined convex hull. The following was proved in \cite[Proposition 3.26]{CTT} for the asymptotic boundary case. The proof is the same in our case, but we include it for the sake of completness.

\begin{proposition}\label{p:MaximalSurfaceContainedInConvexHull}
If $\Sigma$ is a maximal surface in $\Hn$ whose total boundary $\Lambda$ (see Definition \ref{d:AsymptoticBoundary}) is a semi-positive loop, then $\Sigma$ is contained in the convex hull of $\Lambda$.
\end{proposition}

\begin{proof}
Consider a connected component $\Sigma_0$ of the preimage of $\Sigma$ in $\hHn$, and let $\varphi_0$ be a linear form on $E$ which is positive on $\Lambda$ and denote by $\varphi$ its restriction to $\Sigma$.  Now, let $\gamma$ be a geodesic on $\Sigma$ that we consider as a curve in $E$, with $\dot\gamma(0)=u$, then using equation \eqref{eq:geod} in the last equality
\begin{eqnarray*}
	\Hess_x\varphi(u,u)=\left.\frac{\rm d^2}{{\rm d t}^2}\right\vert_{t=0}\varphi(\gamma(t))
	=\varphi_0\left(\left.\frac{\rm d^2}{{\rm d t}^2}\right\vert_{t=0}\gamma(t))\right)
=\q(u)\ \varphi(x)+\varphi_0(\II(u,u))\ ,
\end{eqnarray*} 
where $\II$ is the second fundamental form of $\Sigma$ and $\Hess_x\phi$ is the Hessian of $\varphi$ at $x$. Taking the trace yields
\[\Delta \varphi = 2\varphi~.\]
The classical maximum principle thus implies that $\varphi$ is positive on $\Sigma$ and so $\Sigma$ is contained in the convex hull of $\Lambda$. 	
\end{proof}

\subsection{Gauß lift and holomorphic curves}\label{ss:GaussLift} Recall from Subsection \ref{ss:Grassmannians} that the Grassmannian $\GR{M}$ of positive definite 2-planes in $M$ is the fiber bundle over $M$ whose fiber over $x$ is the Riemannian symmetric space $\Gr{\T_xM}$. The tangent space of $\GR{M}$ at $(x,P)$ splits as
\[\T_{(x,P)}\GR{M} = \T_x M \oplus \Hom(P,P^\bot)=P\oplus P^\bot\oplus \Hom(P,P^\bot)\ .\]
Furthermore, the canonical Riemannian metric $g$ on $\GR{M}$ is given by
\[g=\left(g_{\vert P},-g_{\vert P^\bot},h_P\right)\ ,\]
where  $h_P$ is the Riemannian metric on the fiber defined in paragraph \ref{ss:RiemannianSymmetricSpace}.

Given a spacelike surface $u: S \hookrightarrow M$, we define its \emph{Gauß lift} by
\[
\Gamma:\left\{
\begin{array}{llll}
 & S & \longrightarrow & \GR{M} \\
& x & \longmapsto & \left(u(x),\T u(\T_xS)\right)\ . \end{array}\right. \]
One easily checks that in the splitting describe above, $\T \Gamma = (\T u, \II)$ where $\II\in \Omega^1(S,\Hom(\T S,{\mathsf N} S))$ is the second fundamental form. 

We denote by $g_\II$ and $d_\II$ the induced metric and distance  respectively on $S$ by $\Gamma$.

\begin{proposition}\label{pro:IetII}
For any spacelike surface we have 
\[g_\II= g_\I + Q\leq (1+\Vert \II\Vert^2)g_\I \ \ , \]
where  $Q(x,y)=-\tr_{g_\I}\left(g_N\left(\II(x,.),\II(y,.) \right) \right)$
\end{proposition}

\begin{proof}
As noted above, $\T\Gamma=(\T u,\II)\in \Omega^1\left(S,\T S\oplus \Hom(\T s,\No S) \right)$. The first term in the expression of $\Gamma^*g$ is clear, so we just have to explain the second one. We recall from subsection \ref{ss:RiemannianSymmetricSpace} the bilinear form $h_P(\varphi,\psi)=-\tr(\varphi^*\psi)$ where $\varphi,\psi\in \T_P\Gr{E}$ and $\varphi^*: P^\bot \to P$ is the adjoint of $\varphi$ using the induced scalar product. In particular, the second term is given by $h_P\left(\II(x,.)\II(y,.)\right) =-\tr_{g_\I}\left(\II^*(x,.)\II(y,.)\right)$.

Using $\II^*=B$, and taking an orthonormal framing $(e_1,e_2)$ of $(\T S,g_\I)$, the second term may then be written as
\[-\sum_{i=1,2} g_\I\left(B(x)\II(y,e_i),e_i\right) = \sum_{i=1,2} -g_N\left(\II(x,e_i),\II(y,e_i) \right) = Q(x,y)\ .\]
This proves the result.
\end{proof}

Given a point $x$ on a spacelike acausal surface $\Sigma$ in $\Hn$, we can define three Riemannian metrics on $\Sigma$: the metric $g_\I$ induced by the metric on $\Hn$, the metric $g_\II$ induced by the Gauß lift and the metric $g^x_H\defeq \pi_x^*g_H$ where $\pi_x$ is the warped projection on the pointed hyperbolic plane $(x,H)$ tangent to $\Sigma$ at $x$, and $g_H$ is the hyperbolic  Riemannian metric on $H$.  

\begin{corollary}\label{c:UpperBoundLiftMetric} 
Let $S$ be a spacelike surface whose second fundamental form is uniformly bounded by $M$,
\begin{enumerate}
	\item  the  Riemannian tensors  $g_\I$,  $g_\II$  are uniformly equivalent, with a constant only depending on $M$. 
	\item For any $R$, the metrics $g_\I$ and $g^x_H$ are uniformly equivalent on the ball of center $x$ and radius $R$ with respect to $g_\I$ with a constant only depending on $M$ and $R$. Moreover the projection $\pi_x$ from $S$ equipped with $d_\I$ and $d_\II$ is Lipschitz.
\end{enumerate}
\end{corollary}
\begin{proof} By the previous proposition, the bound on the norm of $\II_\Sigma$ implies that $g_\I$ and $g_\II$ are biLipschitz.
Moreover, the Gauß lift of a ball $B_R$ with respect to $g_\I$  in $\GR{\hHn}$ is contained in the closed ball of center $\T_xS$ and radius $R(1+M)$. In particular, the restriction of the warped projection $\pi_x$ to $B_R$ is infinitesimally biLipschitz by Proposition \ref{pro:diamwarp}. This shows that $g_\I$ and $g^x_{H}$ are uniformly equivalent. 
\end{proof}
\begin{corollary}\label{c:complete} 
If $S$ is a properly immersed  spacelike surface  without boundary whose second fundamental form is uniformly bounded,
then $(S, g_\I)$ is complete.
\end{corollary}
\begin{proof}
If $S$ is properly immersed, it is thus a global graph over each of its tangent planes. Thus by the previous corollary, for each $x$ in $S$, the ball of radius $1$ with center $x$ with respect to $g_I$ is complete, since the ball of radius $1$ with center $x$ with respect to $g^x_H$ is complete. This shows that $(S, g_\I)$ is complete. 
\end{proof}
\subsubsection{Holonomic distribution}\label{sss:HolonomicDistribution} 
Using the splitting \begin{equation}
 \T_{(x,P)}\GR{M} = \T_x M \oplus \Hom(P,P^\bot)\ , \label{eqdef:split} \end{equation}
 where $\T_x M$ is the {\em horizontal distribution}
we  define the {\em holonomic distribution} $\mathcal{D}$ on $\GR{M}$ by 
\begin{equation}\mathcal{D}_{(x,P)}:=P\oplus \Hom(P,P^\bot)\ .\label{eqdef:holon}
\end{equation}
The following is straightforward.
\begin{lemma}
The Gauß lift of a spacelike surface in $M$ is tangent to the holonomic distribution.
\end{lemma}

\subsubsection{Almost complex structure and holomorphic Gauß lift} The holonomic distribution $\mathcal{D}$ carries a natural almost-complex structure ${J}$ defined by taking the rotation $i$ of angle $\frac{\pi}{2}$ on $P$ and the pre-composition by $i$ on $\Hom(P,P^\bot)$:
\begin{equation}
	J(u,A)=(iu, A\circ i)\ . \label{eqdef:almJ}
\end{equation}
The following is classical (see \cite{LabouriePseudoHolomorphic}). 

\begin{proposition}\label{p:PseudoHolomorphicGaussLift}
A spacelike surface $u: S \hookrightarrow M$ is maximal if and only if its Gauß lift $\Gamma: S \to \GR{M}$ is $J$-holomorphic when $S$ is equipped with the complex structure $j$ induced by $g_\I$.
\end{proposition}

\begin{proof}
Considering the splitting $u^* \T M=\T S\oplus \No S$, we get $\Gamma^* \mathcal{D} = \T S\oplus \Hom(\T S,\No S)$, where $\mathcal{D}$ is the distribution on $\GR{M}$ defined in \eqref{eqdef:holon}. In particular, $\T\Gamma\in\Omega^1\left(S,\T S\oplus \Hom(\T S,\No S) \right)$ is identified with $(\T u,\II)$.  

The first factor is clearly ${J}$-holomorphic. For the second factor, it follows from the observation that a map $\varphi: \T_xS \to \Hom(\T_xS,\No_xS)$ satisfies $\varphi\circ j=J \circ \varphi$ if and only if $\varphi$ is symmetric and trace-less.
\end{proof}

\begin{definition}[\sc Boundary condition]\label{sss:BoundaryConditions}
Let $\gamma$ be a strongly positive curve (Definition \ref{d:TypeOfCurves}). \begin{enumerate}
 	\item The {\em boundary condition} associated to $\gamma$ is the immersed submanifold
\begin{equation}
W(\gamma)\defeq
\big\{(x,P)\in\GG(\Hn)\mid x\in\gamma, \ \T_x\gamma\subset P\big\}\ . \label{eqdef:bdc}
\end{equation}
\item For any positive number $K$, the {\em local boundary condition}, is the open subset of  $W(\gamma)$
\begin{equation}
	W_K(\gamma)
	\defeq
	 \big\{(x,P)\in W(\gamma) \mid   d(P,\T^{(2)}_x\gamma)<K \big\}\ . \label{eqdef:local-boundary-condition}
\end{equation}
 \end{enumerate}
 \end{definition}
We have 

\begin{proposition} \label{pro:bc}
	Let $\gamma$ be a strongly positive curve of regularity $C^{k,\alpha}$. \begin{enumerate}
 \item	$W(\gamma)$ is a submanifold of regularity $C^{k-1,\alpha}$  tangent to the holonomic distribution. 
 \item  $\T W(\gamma)$ is a totally real subspace of half the dimension of the holonomic distribution.
 \item Finally, let $\omega$ be the orthogonal projection from $\T\GG(\Hn)$ to the horizontal distribution, then if $(x, P)$ belongs to $W(\gamma)$, we have  $\omega(\T_PW(\gamma))=\T_x\gamma$.
 \end{enumerate}
\end{proposition}
\begin{proof} The first item is obvious. The last item follows from the definition of $W(\gamma)$. Let $(x,P)$ be in $W(\gamma)$. Using the splitting \eqref{eqdef:split},
$$
\T_{(x,P)}W(\gamma)=\T_x\gamma\oplus \{A\in\Hom(P,P^\perp)\mid \T_x\gamma\subset \ker(A)\}\ .
$$
It  follows from the definition \eqref{eqdef:holon}, that $\T_{(x,P)} W(\gamma)\subset \mathcal D_{(x,P)}$. Moreover the definition \eqref{eqdef:almJ} of the almost complex structure $J$ implies that
$$
\T_{(x,P)}W(\gamma)\oplus {J}\T_{(x,P)}W(\gamma)=\mathcal D_{(x,P)}\ .
$$
This completes the proof.
\end{proof}

\section{Uniqueness}\label{sec:Uniqueness}

This section is devoted to the proof of the following result.

\begin{theorem}[\sc Uniqueness]\label{t:Uniqueness}
Let $\Sigma$ be a complete maximal surface in $\H^{2,n}$ whose total boundary $\Lambda$ is either finite and positive or asymptotic (see Definition \ref{d:AsymptoticBoundary}). Then $\Sigma$ is the unique complete maximal surface bounded by $\Lambda$.
\end{theorem}

This statement was proved in \cite[Theorem 3.21]{CTT} in the case of a cocompact group action on $\Sigma$, and the proof relies on a maximum principle. We will adapt this maximum principle here to the non-compact case using a weak version of Omori's maximum principle. Such an adaptation was made in the case of $\mathbf{H}^{2,2}$ for polygonal surfaces in \cite{TamburelliWolf} and suggested to us by the first author of that work.

We work by contradiction. Suppose there exists two maximal surfaces $\Sigma_1$ and $\Sigma_2$ sharing a common boundary, denoted by $\Lambda$. By Proposition~\ref{p:MaximalSurfaceContainedInConvexHull}, both surfaces are contained in the convex hull $\CH(\Lambda)$ of $\Lambda$. Lift of $\mathcal{CH}(\Lambda)$ to $\hHn$ and recall that the scalar product of any pair of points in this lift is negative by Proposition \ref{p:PropertiesConvexHull}, item \ref{it:PropCH2}. This defines a lift of $\Sigma_1$ and $\Sigma_2$ in $\hHn$ that we denote the same way.

Consider the function
\[B: \left\{\begin{array}{lll}
 \Sigma_1 \times \Sigma_2 & \longrightarrow & \R\ , \\
 (x,y) & \longmapsto & \langle x,y\rangle\ .	
\end{array}\right.
 \]
We remark that since $B$ is negative everywhere, $B$ is bounded from above.

\subsection{Lower bound on the Hessian}

We prove the following estimate.

\begin{lemma}\label{l:LowerBoundHessB}
Let $p=(x,y)$ be a point  in $\Sigma_1\times \Sigma_2$. Then there exists two unit vectors $u_0$ and $v_0$ in $\T_x\Sigma_1$ and $\T_y\Sigma_2$ respectively such that, for $w_0=(u_0,v_0)$, we have
\begin{equation}\label{e:HessianOfB}
	\operatorname{Hess}_pB(w_0,w_0) \geq 2B(p) + 2~.
\end{equation}
\end{lemma}

We first compute the Hessian of $B$.

\begin{lemma}\label{l:II2}
The Hessian of $B$ at a point $p=(x,y)$ in $\Sigma_1\times \Sigma_2$	 in the direction $w=(u,v)\in \T_x\Sigma_1\times \T_y\Sigma_2$ is given by
\[\operatorname{Hess}_pB(w,w) = (\q(u)+\q(v))B(p)+2\langle u,v\rangle 	+\langle \II_1(u,u),y\rangle + \langle x,\II_2(v,v)\rangle. \]
Here $\II_i$ is the second fundamental form of $\Sigma_i$.
\end{lemma}

\begin{proof}
Let $\gamma_1$ be a geodesic in $\Sigma_1$ with $\dot\gamma_1(0)=u$, while $\gamma_2$ is a geodesic in $\Sigma_2$ with $\dot\gamma_2(0)=v$, then, using equation \eqref{eq:geod} in the last equality
\begin{eqnarray*}
\operatorname{Hess}_pB(w,w)&=&\left.\frac{{\rm d}^2}{{{\rm d} t}^2}\right\vert_{t=0}\braket{\gamma_1(t),\gamma_2(t)}\crcr
&=&\langle \left.\frac{{\rm d}^2}{{{\rm d} t}^2}\right\vert_{t=0} \gamma_1(t),y \rangle + 2\langle u,v\rangle + \langle x,\left.\frac{{\rm d}^2}{{{\rm d} t}^2}\right\vert_{t=0}\gamma_2(t)\rangle\crcr
&=& \braket{\q(u)\ x+\II_1(u,u),y}+ 2\langle u,v\rangle + \braket{ x, \q(v)\ y + \II_2(v,v)}\ ,	
\end{eqnarray*}
and the result follows. \end{proof}

\begin{proof}[Proof of Lemma \ref{l:LowerBoundHessB}]
Since the surfaces $\Sigma_1$ and $\Sigma_2$ are maximal, the quadratic forms $\langle \II_1(.,.),y\rangle$ and $\langle x,\II_2(.,.)\rangle$ have opposite eigenvalues $\pm \lambda_1$ and $\pm\lambda_2$ respectively. Thus at a given point $p=(x,y)\in \Sigma_1\times \Sigma_2$, up to switching $\Sigma_1$ and $\Sigma_2$, we may assume $\lambda_1\geq\lambda_2\geq 0$. Observe then that for any unit vector $v$ in $v\in \T_y\Sigma_2$ we have
$$
\braket{x,\II_2(v,v)}\geq -\lambda_2\geq-\lambda_1\ .
$$
Now let us choose a unit vector  $u_0\in \T_x\Sigma_1$ such that $\langle \II_1(u_0,u_0),y\rangle=\lambda_1$. Let $w=(u,v)$, then we obtain from  lemma \ref{l:II2} that for any unit $v$ in $\T_y\Sigma_2$, 
\[\operatorname{Hess}_pB((u_0,v),(u_0,v))=2B(p)+2\langle u_0,v\rangle + \langle \II_1(u_0,u_0),y\rangle + \langle x,\II_2(v,v)\geq  2B(p)+2\langle u_0,v\rangle .\]

It is now enough in order to conclude the proof of the lemma to find a unit vector $v_0$ such that $\braket{u_0,v_0}\geq 1$.

Let $\pi$ be the orthogonal projection from $E$ to  $\T_x\Sigma_1$. Since the kernel of $\pi$ is negative definite, we have $\q(a)\leq \q(\pi(a))$ for any $a\in E$. Because $\T_y\Sigma_2$ is positive definite, it follows that the restriction of $\pi$ to $\T_y\Sigma_2$ is a linear isomorphism. Let $v_1$ to be the unique vector such that $\pi(v_1)=u_0$ and observe that $0<\q(v_1)\leq\q(u_0)=1$. Finally, let  $v_0$ the unit vector defined by $v_0=\sqrt{\q(v_1)}^{-1}v_1$. Then
\[\langle u_0,v_0\rangle = \sqrt{\q(v_1)}^{-1}\langle u_0,v_1\rangle = \sqrt{\q(v_1)}^{-1}\langle u_0,\pi(v_1)\rangle =\sqrt{\q(v_1)}^{-1}\geq 1~.\]
The result now follows.
\end{proof}

\subsection{A maximum principle}

The following is a weaker version of Omori's maximum principle \cite{Omori}.

\begin{proposition}[\sc Maximum Principle]\label{p:MaximumPrinciple}
Let $M$ be a complete Riemannian manifold  without boundary whose sectional curvature is bounded from below. Let $f$ be a function of $M$ satisfying the following:
\begin{enumerate}
	\item The function $f$ is of class $\mathcal{C}^2$.
	\item There are positive constants $A,\Lambda$ so that, if $f(x)>A$, then  there is a non-zero vector $v$ in $\T_xM$ such that
$$\operatorname{Hess}_xf(v,v)\geq \Lambda\Vert v\Vert^2.$$
\end{enumerate}
Then either $f$ is bounded by $A$, or $f$ is unbounded.
\end{proposition}

We will denote by $B(x_0,r)$ the ball in $M$ of center $x_0$ and radius $r$, and $d_{x_0}$ the function distance to $x_0$. Recall the classical Hessian comparison theorem (see for instance \cite[Chapter 1]{Besson}).

\begin{proposition}[\sc Hessian comparison principle]\label{p:HessianComparison}
Let $M$ be a complete Riemannian manifold whose sectional curvature is bounded below. There exists positive constants $\epsilon$ and $\lambda$ such that for every point $x_0$ in $M$, and any vector $v\in \T_x M$ with $x$ in $B(x_0,\epsilon)$, we have
\[\operatorname{Hess}_xd^2_{x_0}(v,v)\leq \lambda \Vert v\Vert^2~,\]
\end{proposition}

\begin{proof}[Proof of Proposition \ref{p:MaximumPrinciple}]
Let $\epsilon$ be as in Proposition \ref{p:HessianComparison} and $A$ and $\Lambda$ as in the statement of the proposition. Let $\kappa>0$ be chosen so that for every $x$, we have the following inequality:  for every $z$ in  $B(x,\epsilon)$, 
\[\kappa\operatorname{Hess}_zd^2_{y}(v,v)\leq \frac{\Lambda}{2} \Vert v\Vert^2~.\]
 To prove the proposition, it is enough to show that if there exists $x_0$, with $f(x_0)>A$, then there exists $y$ with $f(y)\geq f(x_0)+\kappa\epsilon$. Choose $x_0$ so that $f(x_0)>A$.
Let now $y$ achieve the maximum of $g\defeq f-\kappa d_{x_0}$ on the ball $B(x_0,\epsilon)$. Observe that at this maximum,  the inequality $g(y)\geq g(x_0)$ reads
\begin{eqnarray}
f(y)\geq f(x_0)+\kappa d^2(x_0,y)\ . \label{ineq:omori}	
\end{eqnarray}
In particular, $f(y)\geq f(x_0)>A$.
Thus there exists $v$ so that $\Hess_yf(v,v)>\Lambda\Vert v\Vert^2$, hence
$$
\Hess_yg(v,v)= \Hess_yf(v,v) -\kappa \Hess_yd^2_{x_0}(v,v)\geq \frac{\Lambda}{2} \Vert v\Vert^2>0\ .
$$
Hence $y$ cannnot be a point in the interior of $B(x_0,\epsilon)$. Thus $d(x_0,y)=\epsilon$. The inequality \eqref{ineq:omori} now reads 
$$
f(y)\geq f(x_0)+\kappa\epsilon^2\ .
$$
This concludes the proof.
\end{proof}

\subsection{Proof of Theorem \ref{t:Uniqueness}}

We will combine two lemmas.

\begin{lemma}\label{l:SupB} For any $x$ in $\Sigma_1$, the supremum
$M_x$ of $B$ on $\{x\}\times \Sigma_2$ is  greater or equal to $-1$, with equality if and  only if  $x$ belongs to $\Sigma_2$.
\end{lemma}

\begin{proof}
Consider a pointed hyperbolic plane $\P=(x,H)$ tangent to $\Sigma_1$ at a point $x$ in the interior of $\Sigma_1$. In the non-compact case, both surfaces are graphs above the entire $\P$ while in the compact case, they are graphs above the compact domain bounded by the image of the warped projection of their common boundary.

In particular, the fiber above $x$ of the warped projection on $\P$ is a totally geodesic timelike sphere that intersects $\Sigma_2$ in a unique point $y$; furthermore, the geodesic passing through $x$ and $y$ is timelike. This gives $$
	B(x,y)=\frac{1}{2}\left( \langle x,x\rangle +\langle y,y\rangle -\langle x-y,x-y\rangle\right)\geq -1\ .$$
with equality if and only if $y=x$.

Assume conversely that $x$ belongs to both $\Sigma_1$ and $\Sigma_2$. . Then, for all $y$ in $\Sigma_2$, the arc $[x,y]$ is spacelike since $\Sigma_2$ is achronal. Thus 
$$
	B(x,y)=\frac{1}{2}\left( \langle x,x\rangle +\langle y,y\rangle -\langle x-y,x-y\rangle\right)\leq -1=B(x,x)\ .
$$
This concludes the proof.
\end{proof}

We now prove 

\begin{lemma}\label{l:LowB}
The supremum $M$ of $B$ on $\Sigma_1\times\Sigma_2$ is equal to $-1$.
\end{lemma}
\begin{proof}
	We first consider the compact case. Let $m$ be the point where  the function $B$ achieves its maximum $M$ on $ \Sigma_1 \times \Sigma_2$.  
	\begin{enumerate}
		\item 	Assume first that $m=(x,y)$ where $x$ and $y$ belong to the interior of  $\Sigma_1$ and   $\Sigma_2$ respectively. At such a point $m$, the Hessian of $B$  is non-positive. Then  Lemma \ref{l:LowerBoundHessB} says 
	$$M=B(m)\leq -1\ .$$
	\item 	Assume now that $m=(x,y)$ where (say) $x$ belongs to $\partial\Sigma_1$ (the case where $y$ belongs to $\partial\Sigma_2$ is treated in a symmetric fashion).   Since $x$ then also belongs to $\Sigma_2$, by  Lemma \ref{l:SupB}
	$$-1=M_x=M\ .$$
\end{enumerate}

In both situations $M\leq -1$.
 
In the non-compact case, we follow a similar argument using the weak version of the Omori maximum principle Proposition~\ref{p:MaximumPrinciple}. First, note that the Riemannian manifold $\Sigma_1\times \Sigma_2$ has sectional curvature larger than $-1$ by Proposition \ref{p:GaussEquation}. For any $x=(x_1,x_2)$ with $B(x)>-1+\delta$ with $\delta>0$,  Lemma \ref{l:LowerBoundHessB} implies that there exists $w$ tangent to $x$ with
\[\operatorname{Hess}_{x}B(w,w)\geq \left(2B(x)+2\right)\Vert w\Vert^2\geq 2\delta\  \Vert w\Vert^2~.\]
 Since we noted after its definition that $B$ is bounded from above,  Proposition~\ref{p:MaximumPrinciple} then implies that $B$ is bounded by $-1+\delta$. Since this is true for all $\delta$, the function $B$ is bounded by $-1$. Thus $M\leq -1$ in the non-compact case as well. \end{proof}

 We can now conclude the proof of the  Uniqueness Theorem \ref{t:Uniqueness}.  Combining the two Lemmas \ref{l:SupB} and  \ref{l:LowB}, we obtain that for all $x$ in $\Sigma_1$, we have $M_x=-1$. Thus, using the equality case in the same  Lemma \ref{l:SupB}, we obtain that $x$ belongs to $\Sigma_2$. Thus  $\Sigma_1$ is equal to $\Sigma_2$. This concludes the proof.

\section{Main compactness theorem}\label{sec:MainCompactness}
 
The main result of this section is the following compactness theorem concerning complete surfaces.

We start with a definition.

\begin{definition}\label{d:ConvergenceAsGraph}
	A sequence of complete acausal surfaces $\seqk{\Sigma}$ with boundary  converges as a graph over an open subset $V$ of a pointed hyperbolic plane $\Pp_0$ associated to the warped projection $\pi_0$, if,  denoting $\Sigma^V_k\defeq \Sigma_k\cap\pi_0^{-1}(V)$ 
	\begin{itemize}
	\item the sequence  $\seqk{U}$, with 	$U_k\defeq\pi_0(\Sigma^V_k)$ converges smoothly as  open sets with smooth boundary in  $V$ to an open set with smooth  boundary $U_0$, 
	\item there exists a smooth complete acausal surface $\Sigma_0$ over $\pi_0^{-1}(U_0)$, so that $\Sigma^{U_k}_k$ converges to $\Sigma_0$.
	\end{itemize}
\end{definition}
We will mainly use this definition of the convergence of complete acausal surfaces $\Sigma_k$ in settings where the projections $U_k$ are fixed, i.e. not varying with the parameter $k$.
Recall the Gauß lift $\Gamma(S) \subset \GG(M)$, described in section~\ref{ss:GaussLift}, of a surface in $S \subset M$ into its Grassmannian $\GG(M)$.

\begin{theorem}[\sc Compactness theorem]\label{t:CompactCompleteTheorem}
Let  $\seqk{\Sigma}$ be a sequence of connected complete acausal maximal surfaces in $\Hn$, and let $\gamma_k\defeq \partial \Sigma_k$ be the finite boundary of $\Sigma_k$. 
Assume that we have the following boundary conditions:  
\begin{enumerate}
\item The sequence $\seqk{\gamma}$ is strongly positive and uniformly unpinched.
\item There is a positive constant $A$ so that for all $k$, for all $x$ in $\gamma_k$, we have $d\left(\T_{x_k}\Sigma_k, \T_{x_k}^{(2)}\gamma_k\right)\leq A$.
 \item  The sequence $\seqk{\gamma}$ has $C^\infty$ bounded geometry.
\item There is a pointed hyperbolic plane $\Pp$ within a  uniformly bounded distance of the Gauß lift $\Gamma(\Sigma_k)$.
\end{enumerate}

Then, the sequence of  surfaces $\seqk{\Sigma}$ converges as a graph on every bounded ball of $\Pp$.\end{theorem}

The definition of bounded geometry and convergence for spacelike surfaces and strongly positive curves is given in Appendix \ref{app:bg}, the definition of uniformly unpinched is given in Definition \ref{d:TypeOfCurves}, and the definition of finite boundary is given in Definition \ref{d:AsymptoticBoundary}, while the definition of the Gauß lift is given in section \ref{ss:GaussLift}.

The first three hypotheses can be thought of as a $C^1$ bound along the boundary, while the fourth one is an interior $C^0$ bound.

This theorem implies readily a uniform bound on the second fundamental form of complete acausal surfaces without boundary. Such a result is also a consequence of a result by Ishihara \cite{Ishihara}. However the Ishihara bound is not optimal and we will improve upon it in our setting in a subsequent paper \cite{QS}.

We next describe a bound on the second fundamental form in the non-complete case.

If an acausal maximal surface $\Sigma$ is not complete, we define its frontier $\Fr(\Sigma)$ and the distance $\eth(x,\Fr(\Sigma))$ to the frontier as in paragraph \ref{sss:Geodesics} and Definition \ref{d:AsymptoticBoundary}. We may refer to the non-complete case as the  {\em free boundary} case. In this setting, we will have two results: we will have a both a local bound on the geometry as well as a local compactness theorem.

\begin{theorem}\label{t:CompactTheorem}
Let  $\seqk{\Sigma}$ be a sequence of connected acausal maximal surfaces in $\Hn$, and let $\gamma_k\defeq \partial \Sigma_k$ be the finite boundary of $\Sigma_k$. Let also $\seqk{x}$ be a sequence of points so that $x_k$ belongs to $\Sigma_k$.

Assume that we have the following boundary conditions:  
\begin{enumerate}
\item The sequence $\seqk{\gamma}$ is strongly positive and uniformly unpinched,
\item There is a positive constant $A$ so that for all $k$, for all $x$ in $\gamma_k$, we have $d\left(\T_{x_k}\Sigma_k, \T_{x_k}^{(2)}\gamma_k\right)\leq A$.
 \item  The sequence $\seqk{\gamma}$ has $C^\infty$ bounded geometry.
\end{enumerate}
Assume furthermore that $\eth(x_k,\Fr(\Sigma_k))$ is bounded from below by  a positive constant $R$. Then there exists a positive constant $\epsilon$ less than $R$ so that for all $k$ we have that
\begin{itemize}
 \item 	the second fundamental form of $\Sigma_k$ is uniformly bounded on the ball $\Sigma_k^\epsilon$ (with respect to $d_I$) on $\Sigma_k$ of center $x_k$ and radius $\epsilon$. 
 \item  the sequence $\seqk{x_k,\Sigma^\epsilon}$ subconverges smoothly.
 \end{itemize}
\end{theorem}

In Section \ref{sec:Specific}, we will describe three avatars 
of our compactness theorem.

\subsection{Structure of the proof}
The structure of the proof is as follows: 
\begin{enumerate}
	\item In Paragraph \ref{sec:proj-achr}, we describe how to construct "good"neighborhoods of points on acausal surfaces: see Proposition \ref{pro:control-proj}.
	\item In Paragraph \ref{sec:loc-contr}, we use this good neighborhood together with results on holomorphic curves to show local subconvergence under a uniform bound on the second fundamental form.
	\item  In Paragraph \ref{sec:glob-contr}, we extend this subconvergence globally, again under a uniform bound on the second fundamental form.
	\item In Paragraph \ref{sec:BernT}, we prove a Bernstein type theorem: {\em complete maximal surfaces without boundary in the pseudo-Euclidean space $\E^{2,n}$ are spacelike planes}; we also prove a boundary version.
	\item in Paragraph \ref{sec:bd-II} we use the subconvergence, a renormalisation and the Bernstein type theorem to prove a uniform bound on the second fundamental form.
	\item  We conclude the proof of the main compactness Theorem~\ref{t:CompactCompleteTheorem} in paragraph \ref{sec:CompT}, also proving Theorem \ref{t:CompactTheorem}.
 .
\end{enumerate}

\subsection{Preliminary : constructing \enquote{good} neighborhoods}\label{sec:proj-achr} This paragraph is devoted to the proof of Proposition \ref{pro:control-proj} below.

In order to apply the theory of holomorphic curves to prove our main compactness theorems, we want to find neighborhoods of points in an acausal surface that are homeomorphic to disks with \emph{at most} one connected arc in the boundary. We also want to control the size of these neighborhoods: they should not be too small (with respect to the warped projection), and their Gauß lift should be uniformly bounded.

Let $\Sigma$ be a spacelike surface in $\Hn_\lambda$. We denote by
\begin{enumerate}
	\item $d_\I$ the induced metric on $\Sigma$ from the metric on $\Hn_\lambda$, and $\area_I$ the corresponding area form,
	\item $\Gamma:\Sigma\to\GR{\Hn_\lambda}$ the  Gauß lift, $d_{\II}$ the induced metric and $\area_\II$ the corresponding area,
	\item $d_\mathcal{G}$ the metric on $\Grn$.
\end{enumerate} 
Given a point $x$ in $\Sigma$,  let 
\begin{enumerate}
	\item  $\Pp_x=(x,H_x)$ be the pointed hyperbolic plane tangent to $\Sigma$ at $x$ (that is such that $\T_xH_x=\T_x\Sigma$),
	\item  $d_H$ be the metric induced on $H_x$ by $\Hn_\lambda$, and $U_R$ be the disk of center $x$ and radius $R$ in $H_x$, and
   \item $\pi_x$ the warped projection from $\Sigma$ to  $H_x$ defined by $\P_x$,
\end{enumerate} 

\begin{proposition}\label{pro:control-proj} There exist constants $A$ and $\delta_0$ so that for any $\delta$ less than $\delta_0$, we have the following. Let $\Sigma$ be an acausal surface in $\Hn_\lambda$ with $\lambda\leq 1$,
 \begin{enumerate}
  \item  whose finite boundary $\partial \Sigma$ is $\delta$-unpinched, and 
  \item whose second fundamental form  has norm bounded by $1$. 
 \end{enumerate}

Then  for all positive $\kappa$ less than $\frac{1}{100}\delta$, any $x$ in $\Sigma$ with $d_\I(x,\Fr(\Sigma))\geq \delta$ admits an open neighborhood $\mathring\Sigma_x$ in $\Sigma$  homeomorphic to the disk with:
\begin{enumerate}
	\item\label{it:i-control-proj}  $\mathring\Sigma_x\cap \partial\Sigma$ has at most one connected component,
	\item\label{it:ii-control-proj} for all  $y$ in $\mathring\Sigma_x$, we have $d_\GG(\T_y\Sigma, \T_x\Sigma) \leq  A\kappa$, 
   \item\label{it:iii-control-proj} the subset $\mathring\Sigma_x$ is a graph over a subset $V_x$ of the disk $U_\kappa$ in $H_x$, and
	\item\label{it:iv-control-proj}  we have   $d(x, \pi_x(\Fr(\mathring \Sigma))\geq \frac{\kappa}{A}$.
\end{enumerate}
\end{proposition} 
 
\subsubsection{The construction: controlling projections of arcs} We assume throughout this paragraph that  
\vskip 0.2truecm
{\em  $\Sigma$ is a spacelike acausal surface with second fundamental form of norm  bounded by $1$ and non-empty $\delta$-unpinched boundary}.
\vskip 0.2truecm
Assuming that $U_R\cap \pi_x(\partial \Sigma)$ is non-empty, we choose $w$ to be a closest point in $U_R\cap \pi_x(\partial \Sigma)$ to $x$:
$$
d_H(x,w)=\inf\{(d_H(x,y)\mid y\in U_R\cap \pi_x(\partial \Sigma)\}\ .
$$
Let $\omega$ the preimage of $w$.

Let  $c_R=c_R(x,\Sigma)$ be the connected component of $U_R\cap \pi_x(\partial \Sigma)$ containing $w$, and $\gamma_R$ be the preimage of $c_R$. The point $w$ and the arc $c_R$ are not uniquely chosen, and when $U_R$ is disjoint from $\pi_x(\partial \Sigma)$, they do not exist.

We begin our approach to Proposition~\ref{pro:control-proj} by showing, under the background assumption that $x$ is relatively distant to the image of the frontier of $\Sigma$, that short geodesic arcs in $H_x$ may be lifted to arcs in $\Sigma$, assuming the geodesic arcs in $H_x$ are near $x$.

\begin{lemma}[\sc Lifting arcs]\label{lem:liftdoedarc}
For any positive constants $\delta$ and $\epsilon$, with   $\epsilon\leq \frac{1}{4}\delta$, and point $x$ in $\Sigma$ with $d_I(x,\Fr(\Sigma))\geq \delta$ the following holds. Let $c: [0,\epsilon] \to H_x$ be a geodesic arc (parametrized by arclength) not intersecting $\pi_x(\partial \Sigma)$ except possibly at its extremities and such that
\[d_I(\xi_0,x)\leq \epsilon~,~\pi_x(\xi_0)=c(0)~,\]
then $c$ is contained in $\pi_x(\Sigma)$. Moreover, if $\pi_x(\xi_1)=c(\epsilon)$, then
\begin{eqnarray}
 d_I(\xi_1,x)\leq 2\epsilon~,~d_{\mathcal G}(\T_{\xi_1}\Sigma,\T_x\Sigma )\leq 4\epsilon~. 	\label{ineq:arcx}
\end{eqnarray}

\end{lemma}

\begin{proof}
	 Let 
$$
I_0\defeq\{t\in I\mid \forall s\leq t,\ c(s)\in \pi_x(\Sigma)\} .$$ 
The set $I_0$ is open and non-empty. Let $\xi: I_0\to\Sigma$ be the lift of $c$ starting from $\xi_0$. Since $\pi_x$ is length-increasing, for all $t$ in $I_0$, 
\begin{eqnarray}
	d_\I(\xi(t),x)\leq 	d_\I(\xi(t),\xi_0)+	d_\I(\xi_0,x)\leq 2\epsilon\   .\label{ineq:arcx2}
\end{eqnarray}
Hence we get, 
 $$
d_I(\xi(t),\Fr(\Sigma))\geq d_I(x,\Fr(\Sigma))- d_I(\xi(t),x)\geq \frac{1}{2}\delta\ .  
$$
It follows that $I_0$ is closed, so $I_0=[0,\epsilon]$ and thus $c$ lies in $\pi_x(\Sigma)$. The inequality \eqref{ineq:arcx} follows directly from the inequality \eqref{ineq:arcx2}, and the bound on the second fundamental form applied to Proposition \ref{pro:IetII}.
\end{proof}

Our final ingredient for the proof of Proposition~\ref{pro:control-proj} is a statement that, still assuming that $x$ is reasonably distant from the frontier of $\Sigma$, that if the nearest component of the image of the boundary of $\Sigma$ comes very near $x$, then that component is unique. 

\begin{lemma}\label{lem:proj-arc} For any positive constant $\delta$, there exists a constant $K\geq 1$ so that  for any $R< \frac{1}{100}\delta$,  the following holds.

Choose $x$ in $\Sigma$ so that $c_R(x,\Sigma)$ is not empty. Assume that $d_I(x,\Fr(\Sigma))>\delta$, we have 
\begin{enumerate} 
\item\label{it:projarc1} If the arc $c_R$  intersects $U_{R/K}$, then $c_R$  is the unique connected component of $\pi(\partial\Sigma)\cap U_R$ intersecting $U_{R/K}$.
\item\label{it:projarc2} for all $\zeta$ in $\gamma_R$, 
\begin{eqnarray}
	d_\I(\zeta,x)\leq K \cdot R \ .
\end{eqnarray}
\end{enumerate}
\end{lemma}

\begin{figure}[!h] 
	\begin{center}
	\includegraphics[width=60mm]{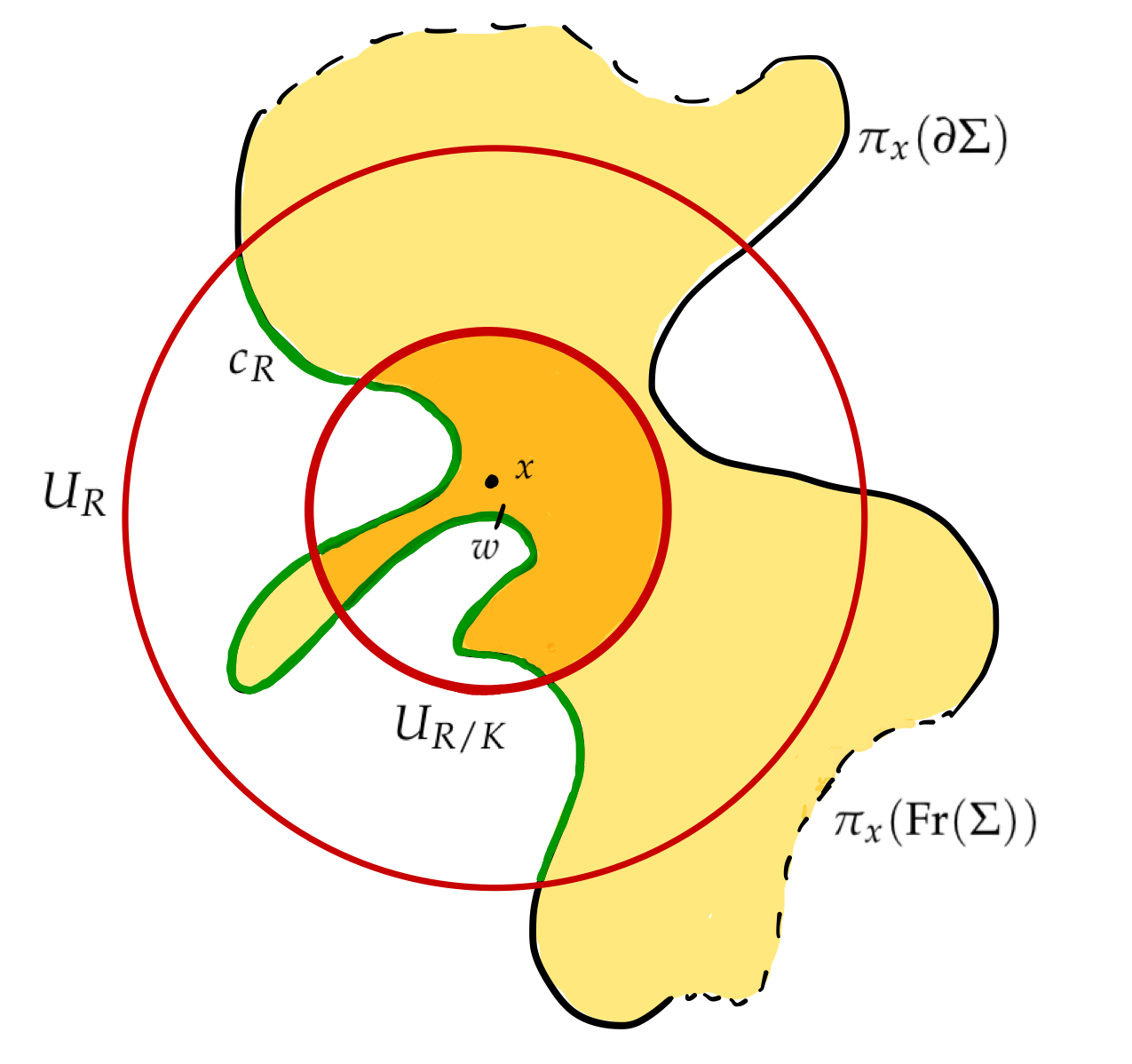}
	\caption{Lemma \ref{lem:proj-arc}}
	\label{fig:lemma56}
	\end{center}
\end{figure}

\begin{proof} 
Let us prove the first assertion. Let $R'=\frac{R}{K}$ where we choose $K$ in the sequel of the proof.

Assume that $\pi_x(\partial\Sigma)$ intersects $U_{R'}$. It follows that $c_R$ intersects $U_{R'}$. Let $w$ be a closest point to $x$ in $c_R$. Thus $w$ belongs to $U_{R'}$. Let $\omega$ be the preimage of $w$ in $\Sigma$. Consider the geodesic arc from $x$ to $w$ whose length is less that $R'$ and so less that $\frac{1}{4}\delta$. Applying the inequality \eqref{ineq:arcx}, we obtain that 
$$
d_I(\omega,x)\leq 2R' \leq \frac{\delta}{50}\ .
$$

We now prove
\vskip 0.2 truecm
 \noindent{\em Assertion A:  if $\zeta\in \partial\Sigma$ and $d_{\partial\Sigma}(\zeta,\omega)\leq  4R'$, then $\zeta\in \gamma_R$.}
 \vskip 0.2 truecm
 
Since $d_{\partial\Sigma}(\zeta,\omega)$ is finite,  $\zeta$ and $\omega$ are by definition in the same connected component of $\partial \Sigma$.  Let  $I$ be the arc along $\partial\Sigma$ from $\zeta$ to $\omega$ that we order from $\omega$ to $\zeta$.  Let $F$ be the closed subset of $I$ defined by 
$$
F=\left\{\xi \in I \mid \forall \theta \in I, \ \theta \leq \xi\implies d_H(\pi_x(\theta),x)\leq \frac{R}{2}\right\}\ .
$$
Obviously $F\subset \gamma_R$, as elements in $F$ have projections to $H$ that are within $\frac{R}{2}$ of $x$, while elements of $\gamma_R$ have projections within $R$ of $x$. The assertion will then follow from the fact  that $\zeta$ belongs to $F$. We will thus show that $F=I$ using a connectedness argument.
\begin{itemize}
	\item  $F$ is non-empty: since 
$
d_H(w,x)\leq R'<\frac{R}{4}$, 
 $F$  contains $\omega$. (Here we demand $K\geq4$.)
\item $F$ is open: assume $\xi$ belongs to  $F$ with $\xi<\zeta$ . By continuity, and because elements of $F$ project to within the smaller $\frac{R}{2}$ ball of $x$, while elements of $\gamma_R$ project to the larger $R$-ball around $x$, we can find $\eta$ in $I$, with $\zeta>\eta>\xi$ so that the interval $J$ joining $\eta$ to $\xi$ lies in $\gamma_R$.  By Lemma \ref{c:UpperBoundLiftMetric} there exists a constant $k$ only depending on $R$, so that  for  $\sigma$ in $J$, 
\begin{eqnarray*}
	d_H(\pi_x(\omega),\pi_x(\sigma))&\leq& k \cdot d_{\I}(\omega,\sigma)\\ 
	&\leq& k \cdot d_{\partial\Sigma}(\omega,\sigma)\\ &\leq& k \cdot d_{\partial\Sigma}(\omega,\zeta)\leq 4k\cdot R'\ .
\end{eqnarray*}
Thus $d_H(x,\pi_x(\sigma))\leq R'(1+4k)$.
If we now choose $K> 2(1+4 k)$, it follows that $\eta$ belongs to $F$.  This shows that $F$ is open.
\item Finally $F$ is obviously closed.
	\end{itemize}
We conclude that $F=I$ and so $\zeta \in F \subset V_{R/2}$, which implies that $\zeta \in \gamma_R$. This completes the proof of Assertion A.
\vskip 0.2 truecm
Now, let $\zeta$ be in the intersection  $\partial\Sigma\cap \pi_x^{-1}(U_{R'})$. The geodesic (in $\Hn$) between $\omega$ and $\zeta$ is spacelike, and so contained in some hyperbolic plane. Since by Lemma \ref{l:WarpedProjectionIncreaseDistances} the warped projection increases the distance, we find
$$
\eth(\omega,\zeta)\leq d_H (\pi_x(\omega),\pi_x(\zeta))\leq 2R'\leq \delta\ .
$$
By the $\delta$-unpinched condition 
\begin{eqnarray*}
	d_{\partial\Sigma}(\omega,\zeta)\leq 2\eth(\omega,\zeta)\leq 4R'\ . 
\end{eqnarray*}
Using  Assertion A, we can now conclude that $\zeta$ belongs to $\gamma_R$.  This proves item \ref{it:projarc1}.

For item \ref{it:projarc2}, consider $\zeta$ in $\gamma_R$. As above 
$$
\eth(\omega,\zeta)\leq d_H (\pi_x(\omega),\pi_x(\zeta))\leq 2R\leq \delta\ .
$$
Thus by  the unpinched condition 
\begin{eqnarray}
	d_{\partial\Sigma}(\omega,\zeta)\leq 2\eth(\omega,\zeta)\leq 4R\ . \label{ineq:liftarc2}
\end{eqnarray}
It follows that 
$$
d_\I(\zeta,x)\leq d_I(\omega,x)+d_{\partial\Sigma}(\zeta,\omega)\leq R'+4R\leq 5 R\ .
$$
This concludes the proof.
\end{proof}

\subsubsection{Proof of Proposition \ref{pro:control-proj}} If a curve is $\delta$-unpinched it is also $\kappa$-unpinched for all $\kappa\leq\delta$. Thus it is enough to prove the proposition for $\kappa=\delta$ and we will do so to avoid burdening the notation.

We choose $R\leq \frac{1}{100}\delta$ as in Lemma \ref{lem:proj-arc}.  Let $K$, $U_R$, $\gamma_R$ and $c_R$ be as in the conclusion of this lemma and the paragraph above it.

We may choose $R$  so that the projection of $\partial\Sigma$ with $U_{R/A}$ intersects transversally.  We apply Lemma \ref{lem:proj-arc} by choosing $A \geq K$ so that $c_R$ is the unique component of $\pi_x(\partial\Sigma)\cap U_R$ intersecting $U_{R/A}$. We now apply Lemma \ref{lem:plane}, to $\gamma=c_R$. We thus obtain a topological  disk $U$ in $U_R$, whose boundary is $\partial U= \alpha_0\cup\alpha_1$, where $\alpha_0$ is a subarc of $c_R$, $\alpha_1$ a subarc of $\partial U_{R/A}$,  and so that $U$ is contained in $U_R$ (see Figure \ref{fig:goodneighborhoods}).

\begin{figure}[!h] 
	\begin{center}
	\includegraphics[width=60mm]{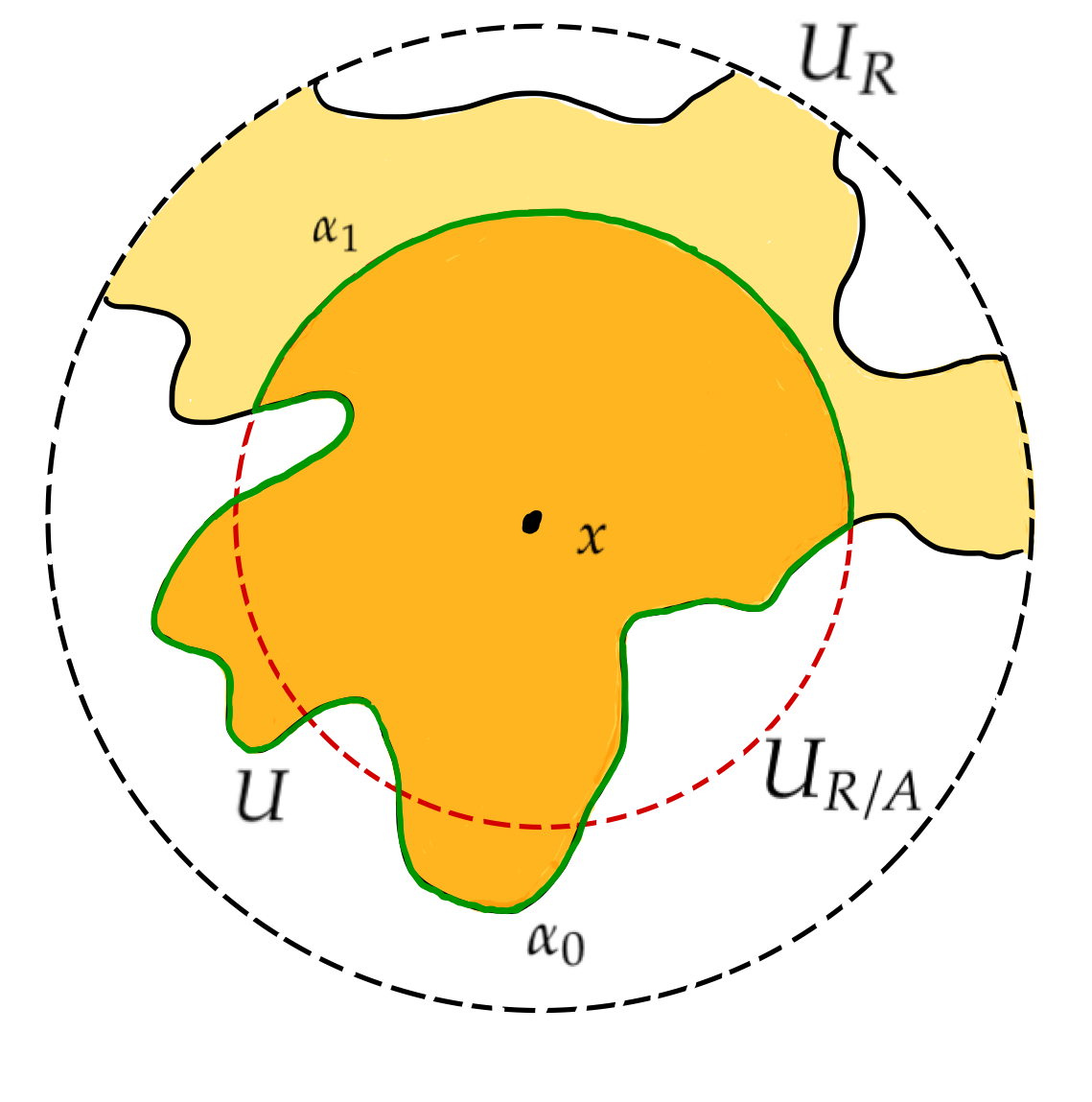}
	\caption{Good neighborhoods}
	\label{fig:goodneighborhoods}
	\end{center}
\end{figure}

Now let $\mathring\Sigma$ be the preimage of $U$.

We first  want to show that  $U$  is a subset of $\pi_x(\Sigma)$. Or in other words that any $y$ in $U$ has a preimage in $\Sigma$.  First taking the closest point $z$ in $c_R$ to $y$,  
 the hyperbolic geodesic arc  $I\defeq [y,z]$ lies in $U$. Let $\zeta$ be the preimage of $z$.

By Lemma \ref{lem:proj-arc}, $d_I(\zeta,x)\leq K \cdot R$ where $K$ is the constant of that lemma. Applying  Lemma \ref{lem:liftdoedarc} with $\epsilon=K\cdot R$, for 
$$
R\leq \frac{\delta}{8K}\,
$$
we obtain that $y=\pi_x(\xi)$ with $\xi$ in $\mathring\Sigma$ and  
\begin{eqnarray}\label{ineq:K_0K_1}
	d_\I(\xi,x)\leq 2K\cdot R\ ,\ d_{\mathcal G}(\T_\xi\Sigma,\T_x\Sigma)\leq 4 K\cdot R\ .\label{ineq:K_0K_1}
\end{eqnarray}

We have just shown that $\mathring\Sigma$ is a graph over $U$, that $U$ is a topological disk,  inequality \eqref{ineq:K_0K_1} holds and the boundary of $\mathring\Sigma$ is connected. 
We have thus proven that the first three items in the proposition.

The last item follows from the fact that $\pi_x(\Fr(\mathring\Sigma))=\alpha_1$ and thus $d(x,\pi_x(\Fr(\mathring\Sigma)))$ is equal to $R_1$.
\qed

\subsection{Local control}\label{sec:loc-contr}
Let us consider the following situation.
\begin{enumerate}
\item Let $\seqk{\lambda}$ be a  sequence of strictly positive numbers converging to $\lambda_\infty\geq 0$.
\item Let $\Sigma_k$ be an acausal maximal surface (possibly with boundary) in $\Hn_{\lambda_k}$.
\item Let $x_k$ be a point in $\Sigma_k$.
\item Let $H_k$  be the totally geodesic plane containing $x_k$, so that  $\T_{x_k}H_k=\T_{x_k}\Sigma_k$, and $\pi_k$ the corresponding warped projection.
\end{enumerate}

By convention, if $\lambda_\infty=0$, we let $\Hn_{\lambda_\infty}$ to be ${\mathbf E}^{2,n}$, the flat pseudo-Euclidean space of signature $(2,n)$.

\begin{definition}[\sc Local Control Hypothesis $(*)$]\label{sec:cond*ST}

The sequence $\{\Sigma_k,x_k, \lambda_k\}_{k\in \mathbb N}$ satisfies {\em the local control hypothesis $(*)$ }, if there exist positive constants  $B$  and $\kappa$  so that
\begin{enumerate}
	\item \label{it:1*}$\Sigma_k$ is a  maximal surface in $\Hn_{\lambda_k}$ which is a graph over $U_k$, where
	\begin{enumerate}
		\item  \label{it:1a*} the set  $U_k$ is a connected submanifold of $B(x_k,\kappa)$, the open ball of center $x_k$ and radius $\kappa$ in $H_k$ ;
		\item \label{it:a1*} the diameter of $\Gamma(\Sigma_k)$ is bounded by $B\kappa$ ;
	\item  \label{it:1b*} we have $d_H(x_k,\Fr(U_k))> {\kappa}{B}^{-1}$ ;
		\item  \label{it:1c*} the  boundary  of $U_k$ is connected.
	
	\end{enumerate}
		\item We have the bound  $
	\Vert \II_{\Sigma_k}\Vert\leq 1\ .$
\item Finally let $\gamma_k=\partial\Sigma_k$ be the finite boundary of $\Sigma_k$. Assume that $\gamma_k$ is 
 strongly positive, and that 
\begin{enumerate}
\item \label{hyp:lcbd1} The sequence of arcs $\seqk{\gamma}$ converges smoothly (in the sense of Appendix \ref{app:bg})  to a strongly positive curve $\gamma_\infty$.
\item \label{hyp:lcbd2}  For any point $y$ in $\gamma_k$, we have  $d_\GG(\T_y\Sigma_k, \T_y^{(2)}\gamma_k)\leq B$. \end{enumerate}
\end{enumerate}
\end{definition} 
Strongly positive curves in $\Hn_\lambda$ with $\lambda>0$ are defined in Definition~\ref{d:TypeOfCurves}. For $\lambda=0$, that is for the pseudo-Euclidean space $\E^{2,n}$, we apply the same definition, replacing {\em hyperbolic plane} by {\em euclidean plane} in the phrasing.

The goal of this paragraph is to show the following:

\begin{proposition}[\sc Convergence with local control]\label{pro:loc-contr}
	For $\kappa$ small enough,  assuming the local control hypothesis $(*)$ and the notation therein, then, after extracting a subsequence, the sequence 
$\{x_k,\Sigma_k, \Hn_{\lambda_k}\}_{k\in\mathbb N}$ converges in the sense of Appendix \ref{app:bg} to  $\{x_\infty,\Sigma_\infty, \Hn_{\lambda_\infty}\}$  where $\Sigma_\infty$ is a maximal surface with boundary $\gamma_\infty$.
\end{proposition}

We prove this proposition in paragraph \ref{sec:pro-loc-contr}.

\subsubsection{Distance and area estimates} Let $M_k=\T \Sigma_k$ be the Gauß lift of $\Sigma_k$ in $\GG(\Hn_{\lambda_k})$,  and let $y_k=\T_{x_k}\Sigma_k$ be the lift of $x_k$. Then set $d_\GG$ to be the Riemannian distance in $\GG(\Hn_{\lambda_k})$.

\begin{lemma}\label{lem:boundSp} We have the following estimates. For $\kappa$ small enough, and assuming the local control hypothesis $(*)$, there exist positive constants $A$, $b$ and $a$ depending only on $\kappa$ and $B$ so that 
\begin{eqnarray}
\area(M_k)&\leq& A\ ,\\
\hbox{for all $w$ in $\Sigma_k$}, \  d_\GG(\T_{x_k}\Sigma_k,\T_w\Sigma_k)&<&b \kappa\  , \\
\hbox{for all $u$ in $\operatorname{Fr}(M_k)$}, \ d_\GG(y_k,u)&>&a\ .
\end{eqnarray}
\end{lemma}

\begin{proof} The second inequality is a direct consequence of \eqref{it:a1*} with $b=B$. It then follows from  Corollary \ref{c:UpperBoundLiftMetric}, that  for $\kappa$ small enough, the projection from $M_k$ (equipped with $d_{\rm II}$) to $U_k$ is  infinitesimally biLipschitz for some constant only depending on $B$. The first inequality follows. 

The third estimate follows from the fact that $\operatorname{Fr}(M_k)\subset Z$, where $Z$ is the closed subset of $\GG(\Hn_\lambda)$ of points $(y,P)$ with $\pi_k(y)\in \operatorname{Fr}(U_k)$. So $a=\inf_{z\in Z}d(H_0,z)$ is positive and only depends on $\kappa$ and $A$.
\end{proof}

\subsubsection{The holomorphic translation}

Let us consider the two possible cases:

\begin{itemize}
\item[{\em Case 1:}] There exists some positive $\epsilon$ so that $d_\GG(\partial M_k,y_k)\geq \epsilon$ for all $k$ in $\N$.
\item[{\em Case 2}:]  There exist $w_k\in \partial M_k$, with $\lim_{k\to\infty}d_\GG(w_k,y_k)=0$.
\end{itemize}
We refer to the notation of Appendix \ref{app:phol}: $\D$ denotes the open unit disk in $\C$, while $\S=\{z\in\D,~\Re(z)\geq 0\}$ is the semi-disk.  Corresponding to these two cases, we consider the following holomophic maps
\begin{enumerate}
	\item in Case 1, we consider the uniformization $f_k: \D \to M_k\setminus \partial{M_k}$, so that $f_k(0)=y_k$.
	\item in Case 2, we consider the uniformization $f_k: \S \to M_k$ so that $f_k(0)=w_k$.
\end{enumerate}
To lighten the notation, we will write $U=\D$ or $\S$. Our hypotheses implies the following. 
\begin{lemma}\label{lem:HolBound}
The maps $f_k$ are holomorphic immersions. Moreover, for $\kappa$ small enough we have the following bounds \begin{eqnarray}
\area(f_k(U))&\leq& A\ ,\\
\hbox{for all $w$ in $U$}, \  d(f_k(0),f_k(w))&<&b\kappa \ ,\\ 
\hbox{for all $w$ in $\Fr(U)$}, \  d(f_k(0),f_k(w))&>&a\ ,
\end{eqnarray}
where $A$, $b$ and $a$ only depends on $\kappa$.
\end{lemma}
\begin{proof}
	This lemma is an immediate  consequence of the holomorphic translation described in proposition \ref{p:PseudoHolomorphicGaussLift} as well as of the bounds on $M_k$ obtained in Lemma \ref{lem:boundSp}.
\end{proof}
If $g$ is a map from $U$ to a space $X$, we define as in appendix \ref{app:phol}, 
$
g(\Fr(U))
$
to be the set of those points $x$ in $X$ so that there exists a sequence $\seqk{z}$ tending to $\Fr(U)$ with 
$$\lim_{k\to\infty}(g(z_k))=x .$$

\begin{corollary}\label{lem:limibord} For $\kappa$ small enough,  the following holds. After extracting a subsequence,the family $\seqk{f}$ converges to a 	non-constant holomorphic map $f_\infty$ so that, 
$$
f_\infty(\Fr(U))\subset \lim_{k\to\infty}f_k(\Fr(U))\ .
$$
\end{corollary}
\begin{proof}
	The Lemma \ref{lem:HolBound} guarantees that we can apply the results on pseudoholomorphic cuves obtained in appendix \ref{app:phol}.  More precisely, we split the discussion in the two cases described in the beginning of this paragraph:
	
	\noindent{\em Case 1:} We are in the free boundary case and we  apply Theorem \ref{theo:holoFB} to get the result.
	
	 \noindent{\em Case 2:} In this case, let us consider $W_k=W_A(\gamma_k)$ defined in Definition~\ref{sss:BoundaryConditions}. By the hypotheses \ref{hyp:lcbd1} and \ref{hyp:lcbd2}, the totally real submanifold $W_k$ converges smoothly to a totally real submanifold. We can now apply Theorem \ref{theo:holoBB} to get the result.
	 
\end{proof}
With the notation above, we have:
\begin{lemma}\label{lem:limigraph} The holomorphic map $f_\infty$ is an immersion at $0$. Moreover, $f_\infty(U)$ is the Gauß lift of a maximal surface.
\end{lemma}
\begin{proof} We have the orthogonal splitting
\begin{eqnarray}
\T_{(x,P)}\GG(\Hn)=P\oplus P^\perp\oplus\Hom(P,P^\perp)\ .\label{eq:split2}
 \end{eqnarray}
Let us consider the complex line subbundle  $V$ of $T\GG$ over $\GG(\Hn)$ so that, in the splitting above, we have $V_{(x,P)}=P$. The orthogonal projection from $
\T_{(x,P)}\GG(\Hn)$ to $P$ is a complex morphism, and thus gives rise to a form  $\alpha$ in $\Omega^1_{\mathbb C}(\GG,V)$.

By construction, if $M$ is the lift of a maximal surface, then $\alpha$ restricted to $\T M$ is injective. Conversely, if $\alpha$ is non-zero restricted to a holomorphic curve $M$, then $M$ is the lift of a maximal surface. 

For any $k$, we now choose a real line bundle $L_k$ in $V$ so that along $W(\gamma_k)$, we have $L_k=\omega(\T W(\gamma_k))$ (see Proposition \ref{pro:bc}).

Let $v$ be a tangent to $M_k$, and $u=\omega(v)$. 
By Proposition \ref{pro:IetII},
$$
\Vert v\Vert^2 = g_\II(u,u)\leq (1+\Vert \II\Vert^2) g_\I(u,u)\leq 2\  \Vert \alpha(v)\Vert ^2\ ,\ 
$$
where in the last inequality we have used the  assumption that the norm of the second fundamental form of $\Sigma_k$ is bounded by 1. Thus it follows that 
$$
\Vert \T f_k\Vert\leq  \sqrt{2}\  \Vert f_k^*\alpha \Vert \ .
$$
According to Proposition \ref{pro:hypboundTf}, this last inequality is enough to imply that the hypotheses of Theorem \ref{theo:imm-limi-bd} are all satisfied. Thus  $f_\infty^*\alpha$ is non-zero, and in particular $f_\infty(U)$ is  the lift of a maximal surface.
\end{proof}
\subsubsection{Proof of Proposition \ref{pro:loc-contr}}\label{sec:pro-loc-contr}  The proposition is a consequence of Lemma \ref{lem:limibord} and \ref{lem:limigraph}.

\subsection{Global control}\label{sec:glob-contr} Our goal in this subsection is to prove a global compactness result under assumptions that we make now precise:

\begin{enumerate}
\item Let $\seqk{\lambda}$ be a bounded sequence of positive numbers.
\item Let $\seqk{R}$ be a sequence of positive number so that $\lim_{k\to\infty}R_k=\infty$.
\item Let $\seqk{\Sigma}$ be a sequence of  connected complete acausal maximal  surfaces in $\Hn_{\lambda_k}$. Let $x_k\in \Sigma_k$ and let
 $$B_k(R)\defeq \{z\in \Sigma_k\mid d_\I(z,x_k)\leq R\},$$
 and let $H_k$ be the totally geodesic plane tangent to $\Sigma_k$ at $x_k$.
\end{enumerate}
\begin{definition}[\sc Global control hypothesis $(**)$]\label{sec:condSSTAR}

The sequence  $\{\Sigma_k,\lambda_k\}_{k\in \mathbb N}$ satisfies the {\em global control hypothesis $(**)$ }, if there exist positive constants $A$, $M_0$  and $\delta$ so that 		
\begin{enumerate}
\item The sequence $\{x_k,H_k,\Hn_{\lambda_k}\}_{k\in\mathbb N}$ converges (in the sense of Definition \ref{def:bdgeom2}).
\item  We have the bound $
 \Vert \II_{\Sigma_k}\Vert\leq M_0$ on $B_k(R_k)$.
	\item The 1-dimensional manifold  $\gamma_k\defeq \partial\Sigma_k$  is  strongly positive and $\delta$-unpinched, 
\begin{enumerate}
\item \label{hyp:lcbd1g} The sequence $\{\gamma_k\}_{k\in\mathbb N}$ has bounded geometry, and
\item \label{hyp:lcbd2g} For any point $y$ in $\gamma_k$, we have $d_\GG\left(\T_y\Sigma_k, \T_y^{(2)}\gamma_k\right)\leq A$. \end{enumerate}
\end{enumerate}
\end{definition}
     
Our goal is now to show the 

\begin{proposition}[\sc Convergence with global control]\label{pro:glob-contr}
	Assuming the global control hypothesis $(**)$.  
\begin{enumerate}
	\item  Then for any positive $R$, there exists $\epsilon$, only depending on $\delta$, $A$ and $\seqk{\gamma}$ so that if  $\seqk{x}$ is a sequence of points with $x_k$ in $\Sigma_k$, with the property that $\eth(x_k,\Fr(\Sigma_k))$ is bounded from below by  $R$ we have: the sequence of pointed surfaces $\{(x_k,\Sigma^\epsilon_k)\}_{k\in\mathbb N}$ subconverges smoothly on the ball  $\Sigma^\epsilon_k$ in $\Sigma_k$, where $\Sigma^\epsilon_k$ has center $x_k$ and $d_\I$-radius $\epsilon$. 
\item 	Assume now that $\Sigma_k$ is complete, so that in particular $Fr(\Sigma_k)$ is empty. Let  $\pi_0$ be the warped projection on  some pointed hyperbolic plane $\Pp_0=(H_0,x_0)$. Assume that 
	$d_{\mathcal G}(\Pp_0,\Gamma (\Sigma_k) )$
 	is uniformly bounded. Then $\{\Sigma_k,\Hn_{\lambda_k}\}_{k\in\mathbb N}$ subconverges as a graph on every compact set of $H_0$ (see Definition \ref{d:ConvergenceAsGraph}).
\end{enumerate}	
\end{proposition}
We remark that by rescaling, that is replacing  $\lambda_k$ by $\lambda_k M_0^{-2}$ in the hypothesis and reversing in the conclusion, we can always assume that $M_0=1$. We will do so in the next paragraph.

\subsubsection{Getting a  local control}
Let us first prove the following lemma.
\begin{lemma}\label{lem:1e*}
	Assume that  $\{\Sigma_k,\lambda_k\}_{k\in \mathbb N}$ satisfies  the global control hypothesis $(**)$ with $M_0=1$. Then for any positive $R$, there exists $\kappa$ and $B$ only depending on $R$, $A$, $\delta$ and $\seqk{\gamma}$ so that the following is true. Let  $\seqk{x}$ with $x_k$
	 in $\Sigma_k$ and $\eth(x_k,\Fr(\Sigma_k))$ bounded below by $R$, then there exists a sequence of maximal surfaces $\seqk{\mathring\Sigma}$ so that 
	\begin{enumerate}
		\item We have the inclusions 
		$\mathring\Sigma_k\subset \Sigma_k$, and $ \partial\mathring\Sigma_k\subset \partial\Sigma_k$.
		\item the sequence $\{\mathring\Sigma_k,  x_k, \lambda_k\}_{k\in \mathbb N}$ satisfies  the local control hypothesis $(*)$ for $\kappa$ and $B$.
	\end{enumerate}
\end{lemma}
\begin{proof} Let us use Proposition \ref{pro:control-proj} to construct, given $\kappa$ small enough, a $\mathring\Sigma_k$ and  an open set $U_k$ which will satisfy items \ref{it:1a*}, \ref{it:a1*},   \ref{it:1b*},  and \ref{it:1c*} of the local control hypothesis $(*)$ ({\it cf.} see Definition \ref{sec:cond*ST}), and thus the general  item \ref{it:1*} is satisfied. Since all the other items are consequences of  the global control  hypothesis $(**)$ ({\it cf.} see Definition \ref{sec:condSSTAR}), it follows that the local control hypothesis $(*)$ is satisfied for $\mathring\Sigma_k$.  This concludes the proof.
\end{proof}

\subsubsection{Proof of the Global Control Proposition \ref{pro:glob-contr}} 

Let  $\{\Sigma_k,\lambda_k\}_{k\in \mathbb N}$ be a sequence satisfying  the global control hypothesis $(**)$ ({\it cf.} see Definition \ref{sec:condSSTAR}).

Let  $\seqk{x}$ be a sequence of points with $x_k\in \Sigma_k$ with $\eth(x_k,\Fr(\Sigma_k))$ bounded from below by a positive constant. Let $\{\mathring\Sigma_k\}_{k\in\N}$  sequence of maximal surface obtained by Lemma \ref{lem:1e*}.

Since by Proposition \ref{pro:loc-contr}, the sequence $\{x_k,\mathring\Sigma_k, \Hn_{\lambda_k}\}_{k\in\mathbb N}$ subconverges, it follows that there exists a constant $\epsilon$, depending on the sequence $\{\Sigma_k,\lambda_k, x_k\}_{k\in \mathbb N}$ so that 
 $\{x_k,\mathring\Sigma^\epsilon_k, \Hn_{\lambda_k}\}_{k\in\mathbb N}$ subconverges smoothly. However, since we can choose our sequence $\seqk{x}$ arbitrarily, provided 
 $\eth(x_k,\Fr(S_k))$ is bounded from below, it follows that we can choose $\epsilon$ to depend only on $\delta$, $A$ and the sequence $\seqk{\gamma}$.

 This concludes the proof of the first item of the proposition.
 
 Let us show the second item.  Let $\Pp_0$ be as in the proposition and $\pi_0$ the warped projection on $\Pp_0$. Recall that the warped projection is a dilation.
 
 Consider $y$ in $\Pp_0$, so that $y=\pi_0(x_k)$ with $x_k\in \Sigma_k$ and $d_\GG(\Pp_0,\T_{x_k}\Sigma_k)$ uniformly bounded. The first item guarantees that $\Sigma_k$ converges as a graph over the ball $B_\epsilon(y)$ of center $y$ and radius $\epsilon$ in $\Pp_0$.
 
 Let now $U$ be the subset of $\Pp_0$,  consisting of those points $z$ so that $\seqk{\Sigma}$ converges as a graph over $B_\epsilon(z)$. We focus on one such particular $z_0 \in U$; in particular, if $\seqk{x}$ is a sequence of points, with $x_k$ in $\Sigma_k$ so that $\{\pi_0(x_k)\}_{k\in \mathbb N}$ converges to an element $w$ in $B_{\epsilon/2}(z_0)$, then $d_\GG(\T_{x_k}\Sigma_k,\Pp_0)$ stays uniformly bounded. Thus, by the previous argument, and using that the $\Fr(\Sigma_k)$ is empty,$w$ belongs to $U$.
 
 It follows that $U$ contains the $\frac{\epsilon}{2}$ neighborhood of itself. Thus, if $U$ is non-empty, then $U=\Pp_0$.
 
 The hypothesis guarantees the existence of a sequence $\seqk{z}$ of points so that $d_\GG(\T_{z_k}\Sigma_k,\Pp_0)$ stays bounded.  It follows that $\pi_0(z_k)$ is a bounded sequence, hence subconverges to a point $y$ which belongs to $U$. Hence $U$ is not empty.
 
 This concludes the proof of the proposition.

\subsection{Bernstein type theorem}\label{sec:BernT}

In this paragraph, we prove the following.

\begin{theorem}[\sc Bernstein for maximal surfaces]\label{theo:BernT}
\begin{enumerate}
	\item 	Let $\Sigma$ be a complete maximal surface without boundary in $\bR$. Then $\Sigma$ is a totally geodesic 2-plane.
	\item Let $\Sigma$ be a complete maximal surface in $\bR$ whose boundary $\partial \Sigma$ is a geodesic. Then $\Sigma$ is a half-plane.
\end{enumerate}
\end{theorem}
 The first part of this result was proved by Ishihara in \cite{Ishihara} for $p$-dimensional entire graphs in the signature $(p,q)$ pseudo-Euclidean space $\E^{p,q}$. Here, for completeness, we give a more direct proof for the case $\bR$. The reader may wish to compare the analogous proof of Chern \cite{Chern1969} in the classical setting.   

\begin{proof} Observe first that  $\GR{\bR}$ is isomorphic to $\bR\times \Gr{\T_x\bR}$ where $x$ is any point in $\E^{2,n}$. By the proof of Proposition \ref{p:PseudoHolomorphicGaussLift},
the projection of the Gauß map $\Gamma:\Sigma \to \GR{\bR}$ to the second factor of the decomposition above yields a holomorphic map  $\varphi: \Sigma\to \Gr{\T_x \bR}$. 

Observe now that $\Gr{\bR}=\SO_0(2,n)/(\SO(2)\times\SO(n))$ is the symmetric space of $\SO(2,n)$ and, by a theorem of Harish-Chandra (see again \cite{Clerc:2003aa}), is biholomorphic to a bounded domain in $\C^n$.

We remark that $\Sigma$ is conformal to $\C$: by Gauß' equation, the induced metric on $\Sigma$ has non-negative curvature 
(see Proposition \ref{p:GaussEquation}): by a result of Blanc and Fiala \cite{BlancFiala:1942} (see also \cite{Huber:1957}), we see that
$\Sigma$ uniformizes as the complex plane $\C$.
By Liouville theorem  $\varphi$ is constant. Hence  $u(\Sigma)$ is a spacelike plane.

Let us now consider the boundary case. Let us construct $\varphi$ as above. Let  $L$ be a spacelike line in $\bR$ and 
$$
W_L\defeq\left\{P\in \Gr{\T_x \bR}\mid  L\subset P\right\}\ .
$$
Then $W_L$ is totally real, and  totally geodesic: the geodesic between two 2-planes with a common line $L$, consists of planes containing $L$.
It follows that $\partial \Sigma$ is totally geodesic for the induced metric on $\Sigma$ by $\varphi$, and thus the same argument applied to the doubling shows that $S$ is unifomized by the half-plane.

In the corresponding symmetric domain picture we can assume that the image of $W_L$ is a contained in  real vector subspace in $\C^n$. Thus, again, the Liouville theorem allows us to conclude that $\varphi$ is constant.
\end{proof}

\subsection{Bounds on the second fundamental form}\label{sec:bd-II}

 Let $\seqk{\Sigma}$ be a sequence of acausal maximal surfaces which are complete with respect to $d_\I$. Let $\gamma_k\defeq \partial \Sigma_k$ be the finite boundary of $\Sigma_k$. Assume that we have the following boundary conditions:  
\begin{enumerate}
\item The sequence $\seqk{\gamma}$ is strongly positive and uniformly unpinched.
 \item There is a positive constant $A$ so that for all $k$, for all $x$ in $\gamma_k$, we have $d\left(\T_{x_k}\Sigma_k, \T_{x_k}^{(2)}\gamma_k\right)\leq A$.
 \item The sequence $\seqk{\gamma}$ has bounded geometry.
\end{enumerate}
Our main result is now:
\begin{proposition}\label{pro:CT1} For every $\epsilon$, there exists some constant $M$, such that if $\seqk{x}$ is a sequence with $x_k\in \Sigma_k$ so that $d_\I(x,\Fr(\Sigma_k))>\epsilon$, then  the second fundamental form of $\Sigma_k$ at $x_k$ has norm uniformly bounded by $M$.
\end{proposition}
This proposition concludes the proof of the first item of Theorem \ref{t:CompactTheorem}.
\begin{proof}

We first recall without proof the following folkloric result (see for instance Paragraph 1.D in \cite{FolHarm} for a proof).

\begin{lemma}[\sc $\Lambda$-maximum lemma]\label{lem:lml} Let $(X,d)$ be a metric space. Let $\eta$ be a positive number and assume that the ball $B_y(\eta)$ of radius $\eta$ centered at $y$ is complete. Then there exists a constant $\Lambda \geq 1$ 
only depending on $\eta$, such that every positive locally bounded function $f$ on $X$ with $f(y)\geq 1$ admits a {\em $\Lambda$-maximum} on $B_y(\eta)$, that is a point $x$ so that 
\begin{eqnarray}
	f(x)&\geq& \sup\left\{f(y)\ ,\ \frac{1}{\Lambda}f(z)\mid \forall z \hbox{ such that } d(x,z)<\frac{1}{\Lambda\sqrt{f(x)}}\right\}\ . \label{eq:lammax} 
\end{eqnarray}
	\end{lemma}

We now argue by contradiction. For each $k$, let us define the function $f_k:=\Vert\II_k\Vert$ on $\Sigma_k$.  Assume that the second fundamental form of the sequence $\Sigma_k$ is unbounded. Let then $\seqk{y}$ so that $y_k\in\Sigma_k$ and 
$$
\lim_{k\to\infty}f_k(y_k)=+\infty\ .
$$
For each $k$, we apply the $\Lambda$-Maximum Lemma \ref{lem:lml} to $f_k$ and $y_k$, taking $\eta=1$, and thus obtaining a $\Lambda$-maximum $x_k$ of $f_k$ in $B(y_k,1)$. Let $\lambda_k\defeq(\Lambda f_k(x_k))^{-2}$ so in particular we have 
$$
\lim_{k\to\infty}f_k(x_k)=\infty\ ,\ \ \lim_{k\to\infty}\lambda_k=0 \ ,  \ f_k(z) \leq \frac{1}{\sqrt{\lambda_k}}\ .$$
 for all $z$, with 
 $d_I(x_k,z)\leq \lambda_k^{\frac{1}{4}}\Lambda^{-\frac{1}{2}}$.
We renormalize the metric of $\Hn$ by $\lambda_k$ -- following  paragraph \ref{sec:renogr}. We shall denote, by $\Sigma^1_k$, the surface  $\Sigma_k$ seen as a surface in $\Hn_{\lambda_k}$ and denote all the geometric objects associated to $\Sigma^1_k$ with a superscript 1. Then by Lemma \ref{l:RescalingSecondFundamentalForm}, $\Vert \II^1_k\Vert\leq 1$ for all the points $z$ in $\Sigma^1_k$ so that 
$$
d_I^1(z,x_k)\leq R_k\defeq\lambda_k^{-\frac{1}{4}}\Lambda^{-\frac{1}{2}}.
$$
Moreover $\Vert\II^1(x_k)\Vert=1/\Lambda$.

We can now apply the first item of the  Global Control Proposition \ref{pro:glob-contr}, to obtain that, after extracting a subsequence, the sequence $\{x_k, \Sigma^1_k, \Hn_{\lambda_k}\}_{k\in\mathbb N}$ converges smoothly to 
$(x_0,\Sigma_0,{\mathbf R}^{2,n})$, and in particular the norm of the second fundamental form of $\Sigma_0$  at $x_0$ is $\Lambda^{-1}$.

Observe now that since  $\seqk{\gamma}$ has bounded geometry and is $\delta$-unpinched, then $\seqk{\gamma^1}$ converges to a geodesic: more precisely, for every sequence $\seqk{z}$ so that $z_k\in \gamma^1_k$, then $\{z_k,\gamma^1_k,\Hn_{\lambda_k}\}_{k\in\mathbb N}$ converges to $(z_0,\gamma_0,{\mathbf R}^{2,n})$ 
where $z_0 \in \gamma_0$, and $\gamma_0$ is a spacelike geodesic -- see the definition \ref{def:bdgeom2} and the observations thereafter.
Thus the boundary of $\Sigma_0$ is either empty or a geodesic. Thus by our Bernstein Theorem, the surface $\Sigma_0$ is totally geodesic. We have obtained our contradiction.

This concludes the proof of Proposition \ref{pro:CT1}.
\end{proof}

\subsection{Proof of Theorem \ref{t:CompactCompleteTheorem} and \ref{t:CompactTheorem} \label{sec:CompT}}

The theorems now follows  from the successive use of 
\begin{itemize}
	\item Proposition \ref{pro:CT1}  which guarantees that under the assumptions of the theorems the second fundamental form is uniformly bounded.
	\item Proposition \ref{pro:glob-contr} then ensures the conclusions of the theorems.
	\end{itemize}

\section{Three specific compactness theorems}
\label{sec:Specific}
Theorems  \ref{t:CompactCompleteTheorem} and  \ref{t:CompactTheorem} are quite general. They incarnate in three specific versions whose hypotheses are easier to handle.

\subsubsection{Boundary-free version}
The boundary-free version is of special interest and will be proved in Paragraph \ref{sec:bdry-free}. Int the next theorem, we explain three different compactness results that will be useful in this paper and its sequel.
\begin{theorem}[\sc Compactness theorem-boundary free]\label{t:CompactFreeTheorem}
Let $\seqk{\Sigma}$ be a sequence of connected complete acausal maximal surfaces without boundary. For each  $k$, let $x_k$ be a point in $\Sigma_k$, let  $\Lambda_k$ be the asymptotic boundary of $\Sigma_k$ (see definition \ref{d:AsymptoticBoundary}) and $\overline\Sigma_k\defeq\Sigma_K\cup\Lambda_k$.
\begin{enumerate}
	\item Assume that the sequence $\{\T_{x_k}\Sigma_k\}_{k\in\mathbb N}$ converges, then the sequence of pointed maximal surfaces   $\{(x_k, \Sigma_k)\}_{k\in\mathbb N}$ subconverges.
	\item   Assume that $\seqk{\Lambda}$ converges  to a semi-positive loop $\Lambda_0$.
Then the sequence $\seqk{\Sigma}$ subconverges smoothly as a graph on every pointed hyperbolic plane to a complete maximal surface $\Sigma_0$ whose asymptotic boundary is $\Lambda_0$. 
\item Assume that  $\{(x_k, \Sigma_k)\}_{k\in\mathbb N}$ converges to $(x_0,\Sigma_0)$  then any sequence $\seqk{y}$ so that $y_k$ belongs to $\overline\Sigma_k$ subconverges to a point in $\Sigma_0\cup\Lambda_0$ where $\Lambda_0$ is a semi-positive loop which is the asymptotic boundary of $\Sigma_0$ 
\end{enumerate}
\end{theorem}

\subsubsection{Boundary vanishing at infinity}

\begin{theorem}[\sc Vanishing boundary]\label{theo:bdry-vanish}
Let $\Lambda_0$ be a 
positive loop in $\bHn$ and $\seqk{\Sigma}$ be a sequence of connected complete acausal maximal surfaces.

Assume that  $\gamma_k=\partial \Sigma_k$ is compact. Assume furthermore that 
\begin{enumerate}
    \item  $\seqk{\gamma}$ converges to $\Lambda_0$,
    \item  the sequence $\seqk{\gamma}$ is a sequence of strongly positive curves of bounded geometry which are also uniformly unpinched,
    \item the angular width of $\gamma_k$ is uniformly bounded.
\end{enumerate}
Then $\seqk{\Sigma}$ converges as a graph on any pointed hyperbolic plane to maximal surface with asymptotic boundary $\Lambda_0$.
\end{theorem}

This will be proved in Paragraph \ref{sec:bdry-vanish}.

\subsubsection{Finite boundary version}
We also have a finite boundary version  that  will be proved in Paragraph \ref{sec:bdry-cpct}.

\begin{theorem}[\sc Compact surfaces]\label{t:CompactBoundaryTheorem}  Let 
$\seqk{\Sigma}$ be a sequence of compact acausal maximal surfaces in $\Hn$.  Let $\gamma_k$ be the 1-dimensional manifold $\gamma_k\defeq\partial \Sigma_k$. Assume that $\seqk{\gamma}$ are topological circles that converge smoothly to a topological circle $\gamma$ which is strongly  positive. Then $\seqk{\Sigma}$ subconverges smoothly to a maximal surface whose boundary is $\gamma$.
\end{theorem}

\subsection{Proximality of the action of $\G$ on $\PE$}\label{ss:Proximality}
In the first subsection, we study the dynamics on $\PE$ of a divergent sequence in $\G$. These dynamics will be helpful when proving the convergence of a sequence of maximal surfaces. 

The main result of this section is the following:

\begin{proposition}[\sc Proximality]
	\label{c:DynamicDivergingSequence}
Let $\seqj{g}$ be an unbounded sequence in the subgroup of $\G$ fixing a point $x$ in $\Hn$.
Then there exists  a hyperplane $H$ of degenerate signature $(1,n)$ through $x$, such that for any compact set $C$ in $\Proj(E)$ not intersecting  $H$, the orbit of $C$ under $\seqj{g}$ accumulates to a point in $\bHn$ orthogonal to $x$.
\end{proposition}

\subsubsection{Cartan decomposition}\label{sss:AsymptoticAction}

We call a \emph{Barbot crown} a semi-positive loop $\mathcal{C}$ in $\bHn$ made of 4 segments of photons (those objects appeared in the work of Barbot in \cite{barbot} under the name \emph{crown}). Denote the vertices of $\mathcal C$ by $\{c_+,d_+,d_-,c_-\}$, where the planes $\gamma=c_+\oplus c_-$ and $\delta=d_+\oplus d_-$ are non-degenerate, thus giving rise to space-like geodesics in $\Hn$. In particular,  $\gamma$ is a subset of $\delta^\bot$. In particular, $F\defeq\span\{\mathcal{C}\}$ is a non-degenerate subspace of signature $(2,2)$.

The group $\A$ which fixes $\{c_+,d_+,d_-,c_-\}$  and $F^\bot$ pointwise is a {\em Cartan subgroup} of $\G$. By construction every element of $\A_0 $, the component of the identity of $\A$, is characterised by its restriction on $F$ which, in the basis given by $(c_+,d_+,d_-,c_-)$  is a diagonal matrix
\[a(\lambda,\mu)\defeq \left(\begin{array}{cccc}
\lambda & 0&0 &0 \\
0& \mu &0 &0 \\
0&0 & 1/\mu &0 \\
0&0 &0 & 1/\lambda
\end{array}\right)\ , \text{ with }\lambda,\mu>0 \ .\]

Given  $\K$ a maximal compact subgroup of  $\G$, the corresponding \textit{Cartan decomposition} of $\G$ states that any element $g$ in $\G$ may be written as $g=k'a k$ where $k',k\in \K$ and $a=a(\lambda,\mu)$ is in  $\A_0 $. Note that this decomposition is not unique; however, if we impose the condition $\lambda\geq \mu \geq 1$, the pair $(\lambda,\mu)$ is uniquely defined. We call $(\log(\lambda),\log(\mu))$ the \emph{Cartan projection} of $g$.

\subsubsection{Asymptotics of the action of $\G$} Let  $\seqj{g}$ an unbounded sequence in $\G$, let $g_j=k'_ja_jk_j$ be the Cartan decomposition of $g_j$,  and  $(\log(\lambda_j),\log(\mu_j))$ the Cartan projection of $g_j$. We distinguish two cases:

\begin{enumerate}
\item The sequence $\seqj{g}$ is called {\em $\Pp_1$-divergent} when the sequence $\{\lambda_j/\mu_j\}_{j\in\N}$ is unbounded.
\item The sequence $\seqj{g}$ is called {\em non $\Pp_1$-divergent} when  the sequence $\{\lambda_j/\mu_j\}_{j\in\N}$ is bounded.
\end{enumerate}

Here $\Pp_1$ refers to the parabolic subgroup in $\G$ stabilizing a point in $\bHn$ and the terminology is inspired from \cite{GGKW}.

\begin{lemma}[\sc $\Pp_1$-divergent sequence]\label{l:DynamicOfGroupAction1}
Let $\seqj{g}$ be a $\Pp_1$-divergent sequence in $\G$. Then, up to extracting a subsequence, there exist two points $p$ and $q$ in $\bHn$ such that for any compact set $C$  in $\Proj(E)$ and not intersecting $\Proj(q^\bot)$, the sequence $\{g_j\}_{j\in\N}$ converges uniformly to $p$ on $C$.
\end{lemma}

\begin{proof}
Let $V=c_+$ and $W=c_-^\bot$. The group $\A$ preserves the splitting $E=V\oplus W$. Moreover, the element $a_j$ acts by $\lambda_j$ on $V$ and has spectral radius $\mu_j$ on $W$.

Let $y$ be in $\Proj(E)$ but not in $\Proj(W)$, so that  $y=[(v,w)]\in \Proj(V\oplus W)$ with $v\neq 0$.  Since $\lambda_j/\mu_j$ is unbounded, up to extracting a subsequence, we have that $\{a_j\cdot y\}_{j\in\N}$ converges to $[(v,0)]$.

Now since $\K$ is compact, the sequences $\seqj{k'}$ and $\seqj{k}$ subconverge to $k'_0$ and $k_0$ respectively. It follows that for any point $y$ in $\P(E)$ and not in $\P(k_0^{-1}W)$, the sequence $\{g_j\cdot y\}_{j\in\N}$ subconverges to $\P(k'_0V)$. Setting $q=(k_0^{-1}W)^\bot$ and $p=k_0'V$ yields the result.
\end{proof}

\begin{lemma}[\sc Non $\Pp_1$-divergent sequence]\label{l:DynamicOfGroupAction2}
Let $\seqj{g}$ be a Non $\Pp_1$-divergent sequence in $\G$. Then up to extracting a subsequence, there exist photons $\varphi$  and $\psi$ such that for any  compact set $C$ in $\P(E)$ and not intersecting $\P(\psi^\bot)$, the sequence $\{g_j\}_{j\in\N}$ converges uniformly  to a point in $\P(\varphi)$.
\end{lemma}

\begin{proof}
Set $V=c_+\oplus d_+$ and $W=(c_-\oplus d_-)^\bot$. Note that $V$ and $W^\bot$ are isotropic planes (that is photons) preserved by $\A$ and such that $E=V\oplus W$. The spectral radius of $a_j$ on $V$ is $\lambda_j$ and is less or equal to $1$ on $W$.

Since $\seqj{g}$ is unbounded , there is a subsequence of $\seqj{\lambda}$ that tends to infinity. 

Again, for any $y$ in $\Proj(E)$ not in $\Proj(W)$, the sequence $\{a_j\cdot y\}_{j\in\N}$ subconverges to a point in $\Proj(V)$. Setting $\psi\defeq k_0^{-1}W^\bot$ and $\varphi\defeq k'_0V$ yields the result.
\end{proof}

\subsubsection{Proof of Proposition \ref{c:DynamicDivergingSequence}}

First assume that $n\geq 2$, so the group $\Fix(x)=\Stab(x^\bot)\cong \SO(2,n)$ has rank $2$. We can thus take a Barbot crown $\mathcal{C}$ with vertices $\{c_+,d_+,d_-,c_-\}$ in $x^\bot$, which implies that the corresponding Cartan subgroup $\A$ is in $\Fix(x)$.  We consider the  Cartan decomposition (in $\Fix(x)$) $g_j=k'_ja_jk_j$, with  $k'_j$ and  $k_j$ as elements in a maximal compact subgroup $\K_F$ of $\Fix(x)$. Observe that $\K_F$ preserves the orthogonal to $x$.

In the $\Pp_1$-divergent case, we apply Lemma  \ref{l:DynamicOfGroupAction1}, observing that the points  $p$ and $q$ are in $\P(x^\bot)$, and the result follows with $H\defeq \Proj(q^\bot)$.

In the non $\Pp_1$-divergent case, we apply Lemma  \ref{l:DynamicOfGroupAction2}  and observe that  $\varphi$ and $\psi$ are contained in $x^\bot$. We take any point $q$ in $\psi$, and then $H\defeq q^\bot$ so that  $\psi^\bot$ is a subset of $H$.

Finally, if $n=1$, the stabilizer of point is isogenic $\ms{SO}(2,1)$ and has rank 1. Thus the sequence $\{g_j\}_{j\in\mathbb N}$ is $\Pp_1$-divergent and we apply Lemma  \ref{l:DynamicOfGroupAction1}.

\subsection{A priori $C^0$- estimates}

Both Theorems \ref{t:CompactFreeTheorem} and \ref{theo:bdry-vanish}
will follow from some {\it a priori} estimate that we now state:

\begin{proposition}\label{p:APrioriEstimates}
Let $\seqk{\Sigma}$ be a sequence of complete maximal surfaces and $\Lambda_0$ a semi-positive curve in $\bHn$ satisfying either the hypotheses of Theorem  \ref{t:CompactFreeTheorem} or those of Theorem  \ref{theo:bdry-vanish}.

Then for any pointed hyperbolic plane $\Pp=(q,H)$, if $x_k$ in $\Sigma_k$ is the preimage of $q$ for the warped projection, then  the sequence $\{d_\GG(\T_{x_k}\Sigma_k,\T_qH)\}_{k\in\mathbb N}$ is bounded. \end{proposition}

\begin{proof}[Proof of proposition \ref{p:APrioriEstimates}]
Suppose the result is false. Then there exists a  pointed hyperbolic plane $\Pp=(q,H)$ so that $\{d_\GG(\T_qH,\T_{x_k}\Sigma_k)\}_{k\in\mathbb N}$ is unbounded where $\pi(x_k)=q$.

Since the time-like sphere 
$$\mathsf S\eqdef \pi_\Pp^{-1}(q),$$
is compact, the sequence $\seqk{x}$ subconverges to a point $x_0$. The point $x_0$  belongs to $\CH(\Lambda_0)$ since $x_k$ is in the convex hull of the boundary of $\Sigma_k$. 

Let $\seqk{\tau}$ be a sequence of elements of $\operatorname{Stab}(\mathsf{S})$ converging to the identity and  so that $\tau_k(x_k)=x_0$, and let us write $\Sigma_k''\eqdef \tau_k(\Sigma_k)$.

Let $H_0$ be a hyperbolic plane through $x_0$ and 
let $g_k$ in $\G$ be such that $g_k(\T_{x_k}\Sigma_k)=\T_{x_0}H_0$, and let $\Sigma'_k=g_k(\Sigma_k)$. Then by construction  $d_\GG(\T_{x_0}H_0,\T\Sigma'_k)=0$. Observe now 
\begin{itemize}
\item in the context of Theorem~\ref{t:CompactFreeTheorem} the boundary of  $\seqk{\Sigma'}$ is empty and all hypothesis of Theorem~\ref{t:CompactTheorem}   are {\it de facto} satisfied,
\item in the context of Theorem~\ref{theo:bdry-vanish}, the second  hypothesis of this theorem, implies  the first and third hypothesis of  Theorem~\ref{t:CompactTheorem}, moreover  Proposition~\ref{pro:bd-ang} and the third hypothesis of Theorem~\ref{theo:bdry-vanish} gives that the second hypothesis of Theorem~\ref{t:CompactTheorem} is satisfied. \end{itemize}
Thus, in both cases, we can apply  Theorem \ref{t:CompactTheorem}. Hence  the sequence $\seqk{\Sigma'}$ subconverges  on any compact to a maximal surface $\Sigma'_0$.
Here we invoke the hypotheses of Theorem~\ref{theo:bdry-vanish} and Proposition~\ref{pro:bd-ang} if the conditions on the boundary in Theorem \ref{t:CompactTheorem} are not vacuously satisfied. By construction observe that $g_k(x_k)=x_0$, hence that $\Sigma'_0$ passes through $x_0$.

 Let finally $h_k=\tau_k\cdot g_k^{-1}$. The sequence $\seqk{h}$  belongs to  $\operatorname{Stab}(x_0)$, is unbounded by hypothesis,  and 
$$
h_k(\Sigma'_k)=\Sigma''_k\ .
$$

By Proposition \ref{c:DynamicDivergingSequence}, there exists\begin{itemize}
	\item A hyperplane $H^+$ of degenerate signature $(1,n)$ passing through $x_0$  such that for any compact set $C$ in $\cHn$ but not intersecting $H^+$, the sequence $\{h_k(C)\}_{k\in\N}$ accumulates to a point in $\bHn\cap\P(x_0^\bot)$. 
	\item  A hyperplane $H^-$ of degenerate signature $(1,n)$ passing through $x_0$  such that for any compact set $C$ in $\cHn$ but not intersecting $H^-$, the sequence $\{h_k^{-1}(C)\}_{k\in\N}$ accumulates to a point in $\bHn\cap\P(x_0^\bot)$.
\end{itemize} 
We can now prove the proposition.
\begin{enumerate}
	\item Let us first treat the case of the boundary free Theorem~\ref{t:CompactFreeTheorem}. 
	
	By definition a semi-positive curve contains a positive triple. Thus   
$\Lambda_0$ is not included in $H^-$, there exists a sequence $\{y_k\}_{k\in\N}$ in $\partial_\infty\Sigma''_k$ converging to a point $y_\infty$ in $\Lambda_0\setminus H^-$. Then $z_k\eqdef h_k^{-1}(y_k)$ converges to a point $z_\infty$ which is orthogonal to $x_0$. On the other hand, since $y_k$ belongs to $\partial_\infty\Sigma''_k$, $z_k$ belongs to $\partial_\infty\Sigma'_k$. Thus $z_\infty$ belongs to $\Lambda'_0$. This contradicts \cite[Lemma 3.7 (c)]{CTT} which asserts that if $x$ is a point in a complete maximal surface $S$ and $y$ a point in the asymptotic boundary of $S$, then $x$ and $y$ are not orthogonal.
\item Let us now  treat the case of Theorem~\ref{theo:bdry-vanish}, where we assume the limit $\Lambda_0$ is a positive curve. 

Observe that $H^+$  does not contain any positive definite $2$-plane, and thus does not contain  $\Sigma'_0$. Thus, there exists $z_k$ in  $\Sigma'_k$ converging to $z_0$ in $\Sigma'_0$ such that if $y_k\defeq h_k(z_k)$, then  $\{y_k\}_{k\in\N}$ subconverges to a point $y_\infty$ in $\bHn\cap \P(x_0^\bot)$, which is in particular orthogonal to $x_0$.

Since $y_k$ belongs to $\Sigma_k''$, it follows  that $y_k$ converges to $y_\infty$ which belongs to $\Lambda_0$. We obtain a contradiction with item \ref{it:PropCH2bis} of Proposition \ref{p:PropertiesConvexHull}. 
\end{enumerate}
\end{proof}

\subsection{Boundary-free and boundary-vanishing case}\label{sec:bdry-free} \label{sec:bdry-vanish}

The first item of Theorem \ref{sec:bdry-free} is a direct consequence of 
 the Compactness Theorem \ref{t:CompactTheorem}. Indeed in that case, the finite boundary $\gamma_k$ of $\Sigma_k$ is empty and the first and third hypotheses of Theorem \ref{t:CompactTheorem} are fulfilled. The second hypothesis is a conequence of the hypothesis of this first item.

We now move to the proof of the second item of \ref{sec:bdry-free}
In order to apply the Compactness Theorem \ref{t:CompactTheorem}, we first need to find a pointed hyperbolic plane at finite distance from
$\T\Sigma_k$. This is achieved by Proposition \ref{p:APrioriEstimates}.  This completes the proof of Theorem \ref{t:CompactFreeTheorem}. 

The proof of Theorem~\ref{theo:bdry-vanish} is also immediate: we notice from
Proposition~\ref{pro:bd-ang} that for all $x$ in $\gamma_k$
$$
d_\GG(\T_x\Sigma_k,\T^{(2)}_x\gamma_k)\leq w(\gamma_k)\ .
$$
Thus the uniform bounds on the angular width for $\gamma_k$ guarantees that the second item in the hypothesis of Theorem \ref{t:CompactTheorem} is satisfied.

Finally, let us prove the third item in Theorem \ref{sec:bdry-free} using the notation therein. Let us realize the surfaces $\Sigma_k$ as graphs over a fixed pointed hyperbolic disk $\P$ and let $\pi$ be the warped projection. We separately treat two cases.
\begin{enumerate}
	\item In the case the sequence of points $\{\pi(y_k)\}_{k\in\mathbb N}$ stays bounded, then the sequence $\seqk{y}$ subconverges since $\seqk{\Sigma}$ subconverges as a graph over $\P$
	\item Otherwise we can assume that $\{\pi(y_k)\}_{k\in\mathbb N}$ converges to a point $z_\infty$ in the boundary at infinity in $\P$. Recall that $\Sigma_k$ is included in the convex hull $\Omega_k$ of $\Lambda_k$ and that $\Lambda_k$ is the graph of a 1-Lipschitz map $f_k$ over the boundary at infinity of $\P$. We may thus extract a subsequence so that $\seqk{f}$ converges to some 1-Lipschitz map $f_\infty$. It then follows that $\seqk{\Lambda}$ converges in the Hausdorff topology  to some  semipositive loop $\Lambda_\infty$ and  thus that $\seqk{\Omega}$ converges in the Hausdorff topology to the convex hull of $\Lambda_\infty$. Then $\seqk{y}$ converges  to the  element $f_\infty(z_\infty)$ in $\Lambda_\infty$. This concludes the proof.
\end{enumerate}

\subsection{Proof of the compact boundary case: Theorem \ref{t:CompactBoundaryTheorem}}\label{sec:bdry-cpct}

Let us prove that the hypotheses of Theorem \ref{t:CompactTheorem} are satisfied.

We first remark that the first and third boundary conditions of Theorem \ref{t:CompactTheorem} are satisfied by the compactness and strong positivity of $\gamma$. Finally the second condition is satisfied due to proposition \ref{pro:bd-ang} and the fact that a maximal surface lies in the convex hull of its boundary by Proposition \ref{p:MaximalSurfaceContainedInConvexHull}.

\section{Finite Plateau problem}\label{sec:Plateau}
In this section, we prove the following finite Plateau problem.

\begin{theorem}\label{t:FinitePlateauProblem}
If $\gamma$ is a deformable strongly positive closed curve in $\hHn$, there exists a unique complete acausal maximal spacelike surface $\Sigma$ whose total boundary is $\gamma$.
\end{theorem}

Strongly positive curves are defined in Definition \ref{d:TypeOfCurves}, deformable ones in Definition \ref{defi:Deformable} and total boundary in Definition \ref{d:AsymptoticBoundary}. Observe also that by Proposition \ref{p:BoundaryacausalSurface}, such a maximal surface is a graph.

The uniqueness of $\Sigma$ has been proved in Theorem \ref{t:Uniqueness}, so it only remains to prove the existence.

\subsection{Existence by the continuity method}

Consider a deformation $\{\gamma_t\}_{t\in[0,1]}$ of $\gamma$ (see Definition \ref{defi:Deformable}). By Lemma \ref{l:DisjointBarycenter}, we can lift those curves continuously to $\Hn_+$.

Let $\mathcal M$ be the space of $C^{k,\alpha}$ complete compact spacelike surfaces whose boundary is both smooth and strongly positive, which are graphs over a disk. Define
  $$
  \mathcal{A}=
 \{t \in [0,1]\mid \hbox{ there is a complete maximal acausal surface $\Sigma \in \mathcal M$ with  } \partial \Sigma=\gamma_t\}\ .
 $$
 
 We already know that $\mathcal A$ is non empty since it contains $t=0$; we also know that  $\mathcal A$ is closed by Theorem \ref{t:CompactBoundaryTheorem}.
 
 Thus a connectedness argument shows that Theorem \ref{t:FinitePlateauProblem} follows from the following proposition.
 
 \begin{proposition}\label{p:Aopen}
 	The set $\mathcal A$ is open.
 \end{proposition}
 
 The proof is a consequence of Corollary \ref{c:StabilityMaximalSurfaces} and standard techniques of which we now give some details.
 
We need some preliminaries. Let $H_0$ be the hyperbolic plane containing $\gamma_0$, fix a point $q_0$ in $H_0$ and consider the pointed hyperbolic plane $\Pp_0=(q_0,H_0)$ (see Paragraph \ref{sss:WarpedProduct}).

Fix $t_0$ in $\mathcal A$. The image of $\gamma_{t_0}$ by the warped projection on $\Pp_0$ is a smooth simple closed curve bounding a domain $\Omega_0$. Similarly, for any $t$ close to $t_0$, we denote by $\Omega_t$ the  domain in $H_0$ bounded by the warped projection of $\gamma_t$. Let $\Sigma_0$ be a maximal surface with boundary $\gamma_{t_0}$, we want to prove that for all $t$ close to $t_0$, the curve $\gamma_t$ bounds a maximal surface.

The next two lemmas describe coordinates on the space of (some) surfaces, and trivialisation of the associated normal bundles. These technicalities are made necessary by the non linear nature of $\Hn_+$.
 \begin{lemma}[\sc Charts]\label{l:ChartSurface}
 	There exists some $n$-dimensional vector spaces $E_1$, a neighbourhood $\mathcal U$ of $(t_0,0)$    in  $I\times C^{k,\alpha}(\Omega_0,E_1)$,  a map
 	$$
 	\phi: \mathcal U\times \Omega_0\longrightarrow\Hn_+\ ,
 	$$
 	such that for any $(t,G)$ in  neighbourhood of $(t_0,0)$ in $\mathcal U$ the map  $$
 	\phi_{(t,G)}\ : \ x\mapsto \phi(t,G,x)\ ,
 	$$ is a smooth parametrisation of a surface $S_{(t,G)}$ such that
 	\begin{enumerate}
 		\item $S_{(t,G)}$ is a graph over $\Omega_t$, 
 		\item $S_{(t_0,0)}=\Sigma_0$, and  
 		\item if $G$ vanishes on the boundary, then $	\partial S_{(t,G)}=\gamma_t $.\label{it:bdSg}
 	\end{enumerate}
 	\end{lemma}
 	
\begin{proof} After possibly taking a smaller $I$, let us choose an immersion 
$$
F: I\times \Omega_0\to \Hn_+=H_0\times \S^n,
$$ so that $\pi(F(t_0,x))=x$ and for every $t$, the map from $\partial\Omega_0$ to $\Hn$ given by $$x\mapsto F(t,x)\ ,$$ is a parametrisation of $\gamma_t$. Then for every $t$ small enough, the map $x\to F(t,x)$ is a parametrisation of a spacelike surface $S_t$ whose boundary is $\gamma_t$.

Let us now consider a trivialisation of the bundle $F^*(\T \S^n)$ as  $I\times \Omega_0\times E_1$, and finally consider the map 
$$
\phi(t,G,x)=\exp_{F(t,x)}(G(x))\ ,
$$
where we have considered $G(x)$ as a vector in $\T_{F(t,x)} \S^n$ using the above trivialisation. The result follows from this construction, writing $S_{(t,G)}$ as the image of $x\mapsto \phi(t,G,x)$. \end{proof}
 	
We have a similar result for the normal bundle 	
\begin{lemma}[\sc Trivialisation of the normal bundle]\label{l:Nbundle}
 Similarly, there exists an $n$-dimensional vector space $E_2$, and a smooth map
 $$
 \Phi: \mathcal U\times C^{k-2,\alpha}(\Omega_0, E_2)\times \Omega_0\to \T\Hn_+\ , 
 $$	
 above $\phi$, 
 such that fixing $(t,G)$ in $\mathcal U$, the map 
 $$
 N_{t,G}: C^{k-2,\alpha}(E_2)\times \Omega_0\to \T\Hn_+\
 $$
 is a linear  isomorphism with the space of $C^{k-2,\alpha}$-section of the normal bundle to $S_{(t,G)}$.
 \end{lemma}

\begin{proof} Let us consider the bundle 
$$
N\to \mathcal U\times \Omega_0\ , 
$$
whose fiber at the point $\phi(t,G,x)$  is the normal bundle of the surface $S_{(t,G)}$ at the point $\phi(t,G,x)$.

We now trivialize this bundle as  $$
 \mathcal U\times \Omega_0\times E_2\ , 
$$
and the result follows 
\end{proof} 
\begin{proof}[Proof of Proposition \ref{p:Aopen}]

Let $(t,G)$ in  $\mathcal U$. 
We denote by  $\Hh(t,G)$ the mean curvature (vector) of the surface  $\Sigma_{(t,G)}$. Using the identification in Lemma \ref{l:Nbundle}, we identify $\Hh(t,G)$ -- with the same notation -- to an element of $C^{k-2,\alpha}(\Omega_0,E_2)$.  Thus  we consider the map 
$$\Hh\ :\left\{\begin{array}{lll}
  \mathcal{U}  & \longrightarrow & C^{k-2,\alpha}(\Omega_0,E_2)\ , \\
(t,G) & \longmapsto & \ \Hh(t,G)\ .
\end{array}\right. 
$$

Let $\mathcal V$ a neigbourhood of $0$ in $C_0^{k,\alpha}(\Omega_0,E_1)$ so that $\{t_0\}\times \mathcal V$ is included in  $\mathcal U$. 
Consider the restricted map
$$\Hh^0:\left\{
\begin{array}{lll}
 \mathcal V  & \longrightarrow & C^{k-2,\alpha}(\Omega_0,E_2)\ , \\
 G & \longmapsto & \ \Hh(t_0,G)\ .
\end{array}\right.
$$ 

Similarly to  the setting of Plateau problems in Euclidean spaces, as shown in  White \cite[Proposition~1.4]{WhiteIndiana87} the map ${\rm D}_0\Hh^0$  is Fredholm of index zero (to adapt the proof of this fact to the present setting, we need only observe that the map ${\rm D}_{0}\Hh^0$\ is strongly elliptic at a spacelike surface). 

By Corollary \ref{c:StabilityMaximalSurfaces}, ${\rm D}_0\Hh^0$ has zero kernel and spectrum bounded away from $0$, so it is invertible with continuous inverse.

To conclude the proof, we will use an Implicit Function Theorem. Fix an extension operator $\epsilon: C^{k,\alpha}(\partial\Omega_0,E_1) \to C^{k,\alpha}(\Omega_0,E_1)$ such that $\left.\epsilon(\delta)\right\vert_ {\partial\Omega_0}= \delta$ for any $\delta$. This gives an isomorphism
\[
\iota:\left\{ \begin{array}{llll}
&  \left(I \times C^{k,\alpha}(\partial\Omega_0,E_1)\right) \times C^{k,\alpha}_0(\Omega_0,E_1) & \longrightarrow & I \times C^{k,\alpha}(\Omega_0,E_1)\ , \\
& (t,\delta,G) & \longmapsto & (t,\epsilon(\delta)+G) \ .\end{array}\right.\]
By restricting the domain of $\iota$ to an open set $\mathcal O$ containing $((t_0,0),0)$, we can further assume $\iota$ takes value in $\mathcal U$.

Finally define
\[\Psi:= \Hh\circ \iota : \  \mathcal O \longrightarrow C^{k-2,\alpha}(\Omega_0,E_2)~.\]
The differential of $\Psi$ at $((t_0,0),0)$ in the second factor --- namely $C^{k,\alpha}_0(\Omega_0,E_1)$ --- is just the operator $\D_0\Hh^0$ which, as we noted, is an isomorphism. By the Implicit Function Theorem, we obtain an application $\theta$ from a neighbourhood $W$ of $(t_0,0)$ in $I\times C^{k,\alpha}(\partial\Omega_0,E_1)$ to a neighborhood of $0$ in $C^{k,\alpha}_0(\Omega_0,E_1)$ such that for $G$ close to $0$ in $C^{k,\alpha}_0(\Omega_0,E_1)$ and $(t,f)$ in $W$,
\begin{eqnarray}
	\Psi((t,f), G)=0\text{ if and only if } G=\theta(t,f) \ . \label{eq:PSI}
\end{eqnarray}
Observe that $\iota((t,0),f)=(t,f)$ and 
let 
 $$
 \Sigma_t=S_{(t,\theta(t,0))}\ .$$
  By equation \eqref{eq:PSI}, $\Sigma_t$ is maximal. Observe that since $\theta(t,f)$ belongs  
to $C^{k,\alpha}_0(\Omega_0,E_1)$, by the item \ref{it:bdSg} of  lemma \ref{l:ChartSurface}, we have 
$$\partial \Sigma_t= \partial S_t = \gamma_t\ .$$
Since $(t,\theta(t,0))$ is well defined and in $\mathcal U$ for $t$ in some neighborhood of $t_0$, we get that $\mathcal A$ contains a neighbourhood of $t_0$. The uniqueness part comes from equation \eqref{eq:PSI} with $f=0$.
\end{proof}

\section{Asymptotic Plateau problem: Theorem \ref{t:MainTheo1}\label{s:ProofsMainTheorems}}

In this section, we prove Theorem \ref{t:MainTheo1}, namely, that any semi-positive loop in $\bHn$ is the total boundary of a unique complete maximal surface in $\Hn$.

In particular, in light of our previous results, what remains to prove is

\begin{theorem}\label{t:ExistenceSmoothBoundary}
Any (smooth) spacelike positive loop in $\bHn$ is the total boundary of a complete maximal surface in $\Hn$.
\end{theorem}

Once this case of Theorem~\ref{t:MainTheo1} is established for smooth boundary loops $\Lambda$, the full theorem follows by approximation of continuous loops $\Lambda$ by smooth ones, using Corollary \ref{c:semi-approx} and Theorem \ref{t:CompactFreeTheorem}.

\subsection{Exhaustion} We construct here an exhaustion of a spacelike positive loop by using radial curves.

More precisely we prove

\begin{proposition}[\sc Exhaustion]\label{l:ExhaustionCurves}
Let $\Lambda$ be a spacelike positive loop in $\bHn$. Then, there exists a sequence $\seqk{\gamma}$ of closed curves in $\Hn$  converging to $\Lambda$ satisfying  the following:
\begin{enumerate}
    \item For $k$ large enough, the curve $\gamma_k$ is strongly positive (Definition \ref{d:TypeOfCurves}).\label{it:exh1}
    \item For $k$ large enough, the strongly positive curve $\gamma_k$ is deformable (Definition \ref{defi:Deformable}). \label{it:exh1bis}
    \item The sequence $\seqk{\gamma}$ is uniformly unpinched (Definition \ref{d:Unpinched}). \label{it:exh2}
    \item The sequence $\seqk{\gamma}$ has bounded geometry (Appendix \ref{app:bg}). \label{it:exh3}
    \item The angular width of $\gamma_k$ is uniformly bounded (Definition \ref{d:AngularWidth}). \label{it:exh4}
\end{enumerate}
\end{proposition} We first give the construction, then prove Proposition \ref{l:ExhaustionCurves} in paragraph \ref{sec:proof-exhaust}.
\subsubsection{Construction of an exhaustion}
Given a point $p$ in $\Hn$, recall from Subsection \ref{ss:pseudosphere} that the set
\[\T^1_p\Hn =\left\{ v\in \T_p\Hn,~\q(v)=1\right\}\]
is isometric to the signature $(1,n)$ pseudo-sphere $\S^{1,n}$. We denote by $\g_{\S^{1,n}}$ its metric. Fix a sequence $\seqk{\rho}$ of positive real numbers tending to infinity, and for any $k$, set
\[\phi_k:\left\{\begin{array}{llll}
 & \T^1_p\Hn & \longrightarrow & \Hn \ ,\\
 & v & \longmapsto & \exp(\rho_kv)\ .
\end{array}\right. 
 \]

The map $\phi_k$ is a diffeomorphism onto the pseudosphere $M_k\defeq\beta(p,\rho_k)=\{x\in\Hn,~\pd(p,x)=\rho_k\}$. We have
\begin{equation}\label{e:pullbackmetric}
\phi_k^*g_{M_k}=\sinh(\rho_k) \g_{\S^{1,n}}\ .
\end{equation}
Taking the limit as $k$ goes to infinity, we obtain a map
\[\phi_\infty: \T^1_p\Hn \longrightarrow \bHn\]
which is a  conformal diffeomorphism onto $\bHn\setminus\P(p^\bot)$.

\medskip

Let us return to our situation. If $\Lambda$ is a spacelike positive curve in $\bHn$, fix a point $p$ in the interior of the convex hull $\CH(\Lambda)$. By Proposition \ref{p:PropertiesConvexHull}, the set $\Lambda$ is disjoint from $p^\bot$. We can thus define the smooth closed curves $\gamma_0$ in $\T^1_p\Hn$  and $\gamma_k$ in $M_k$ respectively by
\[\gamma_0\defeq \phi_\infty^{-1}(\Lambda)~,~\gamma_k\defeq\phi_k(\gamma_0)\ .\]
In other words, $\gamma_k$ is the intersection of the cone over $\Lambda$ with the pseudosphere $M_k$.

\subsubsection{Proof of proposition \ref{l:ExhaustionCurves}}\label{sec:proof-exhaust} We prove that there is a subsequence of the previously constructed sequence $\seqk{\gamma}$ that satisfies the conditions of the proposition.

\

\vskip 0,2truecm

\noindent{\em Proof of \ref{it:exh1}.} Lift each $\gamma_k$ and $\Lambda$ to the hyperplane $\{ x\in E,~\langle p,x\rangle =-1\}$. We denote those lifts with the same notation.

Consider a smooth spacelike loop $\gamma$ in $E$.

We may now choose an auxiliary euclidean metric $h$ on $E$ and let us use the arc length parametrisation with respect to this metric.

Define as in paragraph \ref{sec:angwidth}, the map
\[\Phi_\gamma : \left\{ \begin{array}{lll} \gamma^{(3)} & \longrightarrow & \text{Gr}_3(E) \ , \\
 (x_1,x_2,x_3) & \longmapsto & x_1\oplus x_2\oplus x_3\ .	
 \end{array}\right.
  \]
  where $\text{Gr}_3(E)$ is the Grassmannian manifold of $3$-spaces in $E$.
  
Since $\gamma$ is smooth, using $\dot \gamma$ and $\ddot \gamma\defeq\nabla_{\dot\gamma}{\dot\gamma}$,  the map $\Phi_\gamma$ extends to a continuous map, also denoted   $\Phi_\gamma$ on $\gamma^3$. Observe that if $\q(\gamma)<0$,   then the projection of $\gamma$ in $\Hn$ is strongly positive if and only if $\Phi_\gamma$ takes values in  $\text{Gr}_{2,1}(E)$.

We also  observe that $\Phi_\gamma$ is Lipschitz  -- with respect to the induced metric by $h$ on $\text{Gr}_3(E)$ with a Lipschitz constant that depends on the third derivatives of $\gamma$.

Take an arclength parametrization  --- with respect to $h$ --- of $\Lambda$. Then $\{\Phi_{\gamma_k}\}_{k\in\mathbb N}$ converges uniformly to $\Phi_\Lambda$ by Arzela--Ascoli. 

To conclude the proof of this first item we just need to prove that $\Phi_\Lambda$ takes values in $\text{Gr}_{2,1}(E)$. Let us consider for three vectors $x_1$, $x_2$ and $x_3$  in $E$
\[\Delta (x_1,x_2,x_3)= \det(\langle x_i,x_j\rangle)_{i,j}\ .	 \]
  
Observe also that the map $\Phi_\Lambda$ does not depend on the choice of a parametrisation; we therefore choose a parametrisation of $\Lambda$ so that $\q(\Lambda(t))=1$. Observe then that
taking the derivatives of $\q(\Lambda(t))=0$ and $\q(\dot\Lambda(t))=1$ yield the following equations
\[\langle \dot\Lambda(t),\Lambda(t)\rangle=0~,~\langle \ddot\Lambda(t),\dot\Lambda(t)\rangle=0~,~\langle \ddot\Lambda(t),\Lambda(t)\rangle =-1~.\]
To simply the notation let us write 
$$
V(t,s,u)\defeq \Phi_\Lambda(\Lambda(t),\Lambda(t),\Lambda(s))\ .
$$
\begin{enumerate}
	\item Since $\Lambda$ is postive, $V(t,s,u)$ takes values in $\text{Gr}_{2,1}(E)$ for pairwise distinct $(s,t,u)$. 
	\item For distinct  $s$ and $t$, $V(t,t,s)$ is the space  $$\Lambda(t)\oplus \dot\Lambda(t)\oplus \Lambda(s)\ .$$ 
Thus,  or $s\neq t$, the above equations give, letting $a=\braket{\Lambda(t),\Lambda(s)}$ and $b=\braket{\dot\Lambda(t),\Lambda(s)}$
\[\Delta(\Lambda(s),\Lambda(t),\dot\Lambda(t)) = \det \left(\begin{array}{rrr} 0 & a & b \\ a & 0 & 0 \\ b & 0 & 1 \end{array}\right)=-a^2~<0,\]
Thus $V(t,t,s)$ is non degenerate, hence of type  $(2,1)$ by continuity.
\item Finally, $V(t,t,t)$ is the space  $$\Lambda(t)\oplus \dot\Lambda(t)\oplus \ddot\Lambda(t)\ .$$ Letting $c=\braket{\Lambda(t),\ddot\Lambda(t)}$ we have 
\[\Delta(\Lambda(t),\dot\Lambda(t),\ddot\Lambda(t)) = \det \left(\begin{array}{rrr} 0 & 0 & -1 \\ 0 & 1 & 0 \\ -1 & 0 & c \end{array}\right)=-1<0~.\] Hence $V(t,t,t)$ is non degenerate, hence of type $(2,1)$ by continuity. 
\end{enumerate}
We have proved that $\Phi_\Lambda$ takes values in $\text{Gr}_{2,1}(E)$ and so does $\Phi_{\gamma_k}$ for $k$ large enough by uniform convergence. Thus $\gamma_k$ is strongly positive for $k$ large enough. 

\vskip 0,2truecm
\noindent{\em Proof of \ref{it:exh1bis}}. Lift $p$ to $p_+$ in $\Hn_+$ and let $\Pp=(p_+,H)$ be a pointed hyperbolic plane containing $p_+$. This gives a decomposition $\bHn_+ = \S^1\times \S^n$ in which the boundary of $H$ has the form $\S^1\times \{v\}$. Since $\Lambda$ is disjoint from $p^\bot$, it admits a connected lift $\Lambda_+$ to $\bHn_+$ such that $\langle p_+,\Lambda_+\rangle<0$. This implies that in the splitting $\S^1\times \S^n$, the positive loop $\Lambda_+$ is the graph of a smooth contracting map $f: \S^1 \to \S^n$ (see Proposition \ref{p:PositiveCirclesAreGraph}) whose image is contained in the hemisphere $B=\{x\in \S^n~,~\langle x,v\rangle>0\}$.

Consider the family of maps $\{\phi_t\}_{t\in[0,1]}$ from  $B$ to itself given by 
$$
\phi_t(x)= \frac{tx +(1-t)v}{\Vert tx +(1-t)v\Vert }\ .
$$
Observe that each $\phi_t$ is smooth and contracting for $t<1$. Thus, the maps $f_t\defeq \phi_t\circ f$ are smooth and contracting and so the projection $\Lambda_t$ of their graph to $\bHn$ define an isotopy $\{\Lambda_t\}_{t\in[0,1]}$ such that $\Lambda_1=\Lambda$, the loop $\Lambda_0$ is the boundary of $H$ and for each $\Lambda_t$ is a smooth positive loop.

Because $p$ belongs to the convex hull of every $\Lambda_t$, we can apply the construction on the proof of item \ref{it:exh1}. This gives an isotopy $\{\gamma_k^t\}_{t\in [0,1]}$ for every $k$. By compactness of the isotopy $\{\Lambda_t\}_{t\in[0,1]}$, it follows that  for $k$ large enough, every $\gamma_k^t$ is strongly positive and so $\gamma_k$ is deformable.

\vskip 0,2truecm
\noindent{\em Proof of \ref{it:exh2}}. Denote respectively by $d_k$ and $d_0$ the length along $\gamma_k$ and $\gamma_0$. Given two points $x_k,y_k\in \gamma_k$, we denote by $x_k^o,y_k^o$ the points in $\gamma_0$ such that $x_k=\phi_k(x_k^o)$ and $y_k=\phi_k(y_k^o)$. We thus have by equation (\ref{e:pullbackmetric})
\begin{equation}\label{e:triangle1}
d_k(x_k,y_k) = \sinh(\rho_k) \cdot d_0(x_k^o,y_k^o)~.
\end{equation}

Since $p$ lies in the convex hull of $\Lambda$, Item \ref{it:PropCH4} of Proposition \ref{p:PropertiesConvexHull} implies that the geodesics between $p$ and $x_k$ and between $p$ and $y_k$ are spacelike: here we use that the points $x_k$ and $y_k$ are constructed to lie on geodesics connecting $p$ to $\Lambda$.  Since $\gamma_k$ is  strongly positive by the first item, the geodesic between $x_k$ and $y_k$ is also spacelike and so the triple $(p,x_k,y_k)$ is positive (unless $p$ belongs to the geodesic between $x_k$ and $y_k$ in which case the following still holds). Hence $p,x_k$ and $y_k$ are the vertices of an isosceles hyperbolic triangle $T_k$. Classical hyperbolic trigonometry implies that
\begin{equation}\label{e:triangle2}
\sinh\left(\frac{\eth(x_k,y_k)}{2}\right)=\sinh(\rho_k)\cdotp\sin\left(\frac{\alpha(x_k,y_k)}{2}\right)~,
\end{equation}
where $\alpha(x_k,y_k)$ is the angle at $p$ in the triangle $T_k$. In particular, $\alpha(x_k,y_k)$ is equal to the extrinsic distance between $x_k^o$ and $y_k^o$ in $\T^1_p\Hn$.

Consider a sequence $\{(x_k,y_k)\}_{k\in\N}$ where $x_k,y_k$ are distinct points of $\gamma_k$ such that $\eth(x_k,y_k)$ tends to 0, and let $\alpha_k\defeq \alpha(x_k,y_k)$.

By equation (\ref{e:triangle2}), $\underset{k\to \infty}{\lim}\frac{\eth(x_k,y_k)}{\alpha_k\sinh(\rho_k)}=1$. Together with equation (\ref{e:triangle1}), we obtain
\[\underset{k\to\infty}{\lim}\frac{\eth(x_k,y_k)}{d_k(x_k,y_k)} = \underset{k\to\infty}{\lim} \frac{\alpha_k}{d_0(x^o_k,y^o_k)} = 1~,\]
where for the last equality we used the fact that $\gamma_0$ is smooth, spacelike and that $\alpha_k$ is equal to the extrinsic distance between $x_k^o$ and $y_k^o$.

As a result, there exists a $\delta>0$ such that for any pair of distinct points $x_k,y_k$ on $\gamma_k$ with $\eth(x_k,y_k)\leq \delta$, then $\frac{\eth(x_k,y_k)}{d_k(x_k,y_k)}>\frac{1}{2}$. So the sequence $\seqk{\gamma}$ is uniformly unpinched.

\vskip 0,2truecm
\noindent{\em Proof of \ref{it:exh3}}. In this portion of the argument, we will use some definitions and results from Appendix~\ref{app:bg}. Now, the curve $\gamma_k\subset M_k$ is obtained from $\gamma_0\subset \T_p\Hn \cong \S^{1,n}$ by rescaling the metric by a factor $\sinh(\rho_k)$. In particular, the curvature of $M_k$ and $\gamma_k$ converges uniformly to $0$, and so for any sequence $\seqk{x}$ with $x_k\in\gamma_k$, the sequence $\{x_k,\gamma_k,M_k\}_{k\in\N}$ converges in the sense of Appendix \ref{app:bg} to $\{x,\Delta,\E^{1,n}\}_{k\in\N}$ where $\Delta$ is a spacelike line in the pseudo-Euclidean space $\E^{1,n}$ and $x\in\Delta$. In other words, the sequence $\{\gamma_k,M_k\}_{k\in\N}$ has bounded geometry.

On the other hand, the sequence $\{M_k,\H^{2,n}\}_{k\in\N}$ has bounded geometry since, up to shifting the center of the pseudosphere, $\seqk{M}$ converges to an horosphere (see Subsection \ref{ss:pseudosphere}). It follows from Lemma \ref{l:BoundedGeometry} that $\{\gamma^o_k,\Hn\}_{k\in\N}$ has bounded geometry when $\gamma_k$ is equipped (see Definition~\ref{def:framing})  with the normal framing $\gamma^o_k$ given by $P_k(x)\defeq\T_x\gamma_k \oplus \mathsf{N}_xM_k$, where $\mathsf{N}_x M_k$ is the normal to $M_k$ at $x$. The bounded geometry of $\{\gamma_k,\Hn\}_{k\in\N}$, when $\gamma_k$ is equipped with its canonical normal framing given by $\T^{(2)}\gamma_k$, will follow from a lemma.

\begin{lemma}
Using the same notation as above, the sequence $d_\GG\big(T^{(2)}\gamma_k,P_k\big)$ converges uniformly to $0$.
\end{lemma}

\begin{proof}
Parametrize $\gamma_k$ by arc length and define $\ddot\gamma_k\defeq\nabla_{\dot\gamma_k}\dot\gamma_k$. Then by definition, $\T^{(2)}\gamma_k = \span\{\dot \gamma_k,\ddot\gamma_k\}$.

For a point $x$ in $\gamma_k$, let $n_k(x)$ be the unit vector normal to $M_k$ pointing outward, so $P_k=\span\{ \dot\gamma_k, n_k\}$. Since the planes $P_k(x)$ and $\T^{(2)}_x\gamma_k$ intersect along $\T_x \gamma_k$, we have
\[d_\GG(\T^{(2)}_x\gamma_k,P_k(x)) = d_{\H^n}(n_k(x),\ddot\gamma_k(x))~,\]
where $n_k(x)$ and $\ddot\gamma_k(x)$ are considered as elements in the space of positive definite lines in the signature $(1,n)$ space $(x\oplus \T_x\gamma_k)^\bot$, which is identified with $\H^n$. In particular, 
\[\cosh\left(d_\GG(\T^{(2)}_x\gamma_k,P_k(x))\right)= \frac{\langle n_k(x),\ddot\gamma_k(x)\rangle }{\Vert\ddot\gamma_k(x)\Vert} .\]
Let us write $\ddot\gamma_k\eqdef \mu_k+\nu_k$ where $\mu_k$ belongs to 
$P_k(x)$ and $\nu_k$ to $P_k^\bot(x)$. So  
$$\cosh\left(d_\GG(\T^{(2)}_x\gamma_k,P_k(x))\right)= \frac{\mu_k }{\sqrt{\mu_k^2+\nu_k^2}} .$$

By definition, $\kappa_k\defeq g_{M_k}(\nu_k,\nu_k)$ is the geodesic curvature of $\gamma_k$ seen as a curve in  $M_k$. Similarly,  $\mu_k n_k=\II_k(\dot\gamma_k,\dot\gamma_k)$ where $\II_k$ is the second fundamental form of $M_k$. Since $M_k$ is umbilical with induced curvature $\sinh^{-2}(\rho_k)$, we have
\[\II_k(X,Y)= \coth(\rho_k) g_{M_k}(X,Y)n_k.\]

Now, the curve $\phi_k^{-1}(\gamma_k) = \gamma_0$ is fixed independently of $k$. Moreover, by equation (\ref{e:pullbackmetric}), the induced arclength parametrization of this curve grows without bound in $k$, and so the geodesic curvature $\kappa_k$ of $\gamma_k$  in $M_k$ converges uniformly to zero. The result follows from the fact that $\seqk{\mu}$ converges to 1.
\end{proof}

\vskip 0,2truecm
\noindent{\em Proof of \ref{it:exh4}.} For each $k$, let $x_k$ be a point in $\gamma_k$ such that $w(\gamma_k)=\diam(\pi_{x_k}(\gamma_k))$ where $\pi_{x_k}: E \to (x_k\oplus\T_{x_k}\gamma_k)^\bot$ is the orthogonal projection (see Definition \ref{d:AngularWidth}). We want to prove that the limit of $\diam(\pi_{x_k}(\gamma_k))$, when $k$ tends to infinity, is finite. 

Denote by $H_k$ the hyperbolic plane containing $p$ and $\T_{x_k}\gamma_k$. Since there is a subsequence of $\seqk{x}$ that converges to a point $x_\infty\in \Lambda$, the sequence of pointed hyperbolic planes $\Pp_k=(p,H_k)$ subconverges to $\Pp_\infty =(p,H_\infty)$ where $H_\infty$ is the hyperbolic plane containing $p$ and $\T_{x_\infty}\Lambda$, here using the smoothness of $\Lambda$. 

Fix a spacelike line $L\subset \T_pH_\infty$ and take $g_k\in\G$ such that $g_k(\T_{x_k} \gamma_k)=L$, and $g_k(H_k)=H_\infty$. Denote by $\gamma_k'=g_k(\gamma_k)$, $M'_k=g(M_k)$ and $p'_k=g_k(p)$. From Subsection \ref{ss:pseudosphere}, the sequence $\seqk{M'}$ converges to the horosphere $\mathcal{H}$ tangent to $p_\infty \defeq \lim p'_k$ and passing through $p$. Denote by $\overline\sigma = \sigma \cup \{p_\infty\}$ the closure in $\cHn$ of the horocycle $\sigma=\mathcal{H}\cap H_\infty$.

\begin{lemma}
Using the same notation as above, for any sequence $\seqk{y}$ with $y_k\in \gamma_k$, there is a subsequence of $\{g_k(y_k)\}_{k\in\N}$ converging to a point in $\overline\sigma$.
\end{lemma}

\begin{proof}
Let $\seqk{y}$ be such a sequence. Up to extracting a subsequence, we can assume one of the following:
\begin{enumerate}
	\item\label{it:convhorocycle1} For any $k$ we have $y_k=x_k$ ,
	\item\label{it:convhorocycle2} For any $k$, we have $y_k\neq x_k$ but $\lim y_k = x_\infty$ ,
	\item\label{it:convhorocycle3} The limit of $\seqk{y}$ is different from $x_\infty$ .	
\end{enumerate}
	We will prove that in any case, the sequence $\{g_k(y_k)\}_{k\in \N}$ subconverges to a point in $\overline\sigma$.

\medskip

Case \ref{it:convhorocycle1} is obvious since $g_k(y_k)=p\in\sigma$.

\medskip

For case \ref{it:convhorocycle2}, observe that the hyperbolic plane $V_k$ containing $p,x_k$ and $y_k$ converges to the hyperbolic plane $H_\infty$ containing $p$ and $\T_{x_\infty}\Lambda$. In fact, if $x_k^o$ and $y_k^o$ are the points in $\gamma_0\subset \T^1\Hn$ such that $x_k=\phi_k(x_k^o)$ and $y_k=\phi_k(y_k^o)$, then $V_k$ is also the hyperbolic plane such that $\T_p V_k$ contains $x_k^o$ and $y_k^o$. By the smoothness of $\gamma_0$, the sequence of lines through $x_k^o$ and $y_k^o$ converges to $\T_{x_\infty^o}\gamma_0$, where $\phi_\infty(x_\infty^0)=x_\infty$.

In particular, $g_k(y_k)$ belongs to $g_k(V_k)\cap g_k(M_k)$ and so accumulates to a point in the closure of the intersection $\mathcal{H}\cap H_\infty$, that is  in $\overline\sigma$. 
\medskip

For case \ref{it:convhorocycle3}, we will use a proximality argument in the spirit of Section \ref{ss:Proximality}. For each $k$, write $g_k=g'_k\cdot h_k$ where $h_k(p,H_k)=(p,H_\infty)$ is such that $h_k$ converges to the identity. Then the sequence $g'_k$ is a sequence in $\Stab(H_\infty)\cong \Isom(\H^2)$ which by construction satisfies:
\[  \underset{k\to \infty}{\lim} g'_k(p) = p_\infty ~,~\underset{k\to\infty}{\lim} (g'_k)^{-1}(p) = x_\infty~.\]
The sequence $\seqk{g}$, is such that it approximates the sequence $\seqk{g'}$ which have distinct fixed points on $H_\infty$, is thus $\Pp_1$-divergent and by Lemma \ref{l:DynamicOfGroupAction1} maps any point in $\P(E)$ that are not in $\P(x_\infty^\bot)$ to a sequence converging to $p_\infty$. The result then follows from the fact that $\Lambda$ is positive, so $\P(x_\infty^\bot)\cap\Lambda = x_\infty$.
\end{proof}

Now the proof of item \ref{it:exh4} follows. For each $k$, let $y_k$ and $z_k$ in $\gamma_k$ be such that $w(\gamma_k)=\diam(\pi_{x_k}(\gamma_k))=d_{\H^n}(\pi_{x_k}(y_k),\pi_{x_k}(z_k))$. Since the angular width is invariant under the action of $\G$, we have $w(\gamma_k)=\diam(\pi_{p}(\gamma'_k))$ where $\pi_p: E \to L^\bot$. By the previous lemma, the sequence $\{g_k(y_k)\}$ and $\{g_k(z_k)\}$ subconverge to a point in $\overline\sigma$. Since $\pi_p(\overline\sigma)$ is a point, the sequence $\{w(\gamma_k)\}_{k\in\N}$ subconverges to 0. \qed

\subsection{Proof of Theorem \ref{t:ExistenceSmoothBoundary}}

Consider $\seqk{\gamma}$ the sequence constructed in Proposition \ref{l:ExhaustionCurves}. By Theorem \ref{t:FinitePlateauProblem}, there exists a sequence $\seqk{\Sigma}$ of complete acausal maximal surfaces with $\partial \Sigma_k=\gamma_k$. 

Since $\seqk{\gamma}$ satisfies the hypothesis of Theorem \ref{theo:bdry-vanish}, we obtain a maximal surface whose asymptotic boundary is $\Lambda$.

\appendix
\section{Bounded geometry in the pseudo-Riemannian setting}
\label{app:bg}

\subsection{Convergence of Riemannian and pseudo-Riemannian manifolds}

\subsubsection{The Riemannian setting}  We say that a sequence $\{(x_k,M_k)\}_{k\in\N}$ of pointed Riemannian manifolds \emph{converges $C^n$} to a pointed Riemannian manifold $(x,M)$ if for every compact set $K$ in $M$,  there exists an integer $k_0$ such that, for any $k>k_0$ there is a an open set $U$ containing $K$ in $M$,  a diffeomorphism $\phi_k$ from  $U$ onto an open set in $M_k$, so that when $k$ tends to infinity, we have that $\{\phi_k^{-1}(x_k)\}_{k\in\mathbb N}$ converges to $x$ and the metric $\phi_k^* g_k$ converge $C^n$ to $g$ in $U$, where $g_k$ and $g$ are respectively the metrics on $M_k$ and $M$.

Similarly, given a sequence $\{(x_k,N_k,M_k)\}_{k\in \N}$ such that $N_k$ is a submanifold of $M_k$ and $x_k\in N_k$, we say that $\{(x_k,N_k,M_k)\}_{k\in \N}$ converges $C^n$ to $(x,N,M)$ if for every compact set $K$   there exists an integer $k_0$ and an open set $U$ containg $K$, such that for $k>k_0$,  there is  a diffeomorphism $\phi_k$ from the  $U$   to an open set $U_k$ in $M_k$ containing $x_k$, so that
\begin{enumerate}
	\item $\phi_k(N\cap U)=N_k\cap U_k$,
	\item $\{\phi_k^{-1}(x_k)\}_{k\in\mathbb N}$ converges to $x$,
	\item when $k$ tends to infinity, the sequence $\seqk{\phi_k^* g}$ of metrics converge $C^n$  to $g$ in $U$.
\end{enumerate}

We say an estimate {\em only depends on the local geometry} of a Riemannian manifold $(M,x)$ if such an estimate holds uniformly for any  sequence of pointed Riemaniann manifolds converging to $(M,x)$.

\subsubsection{The pseudo-Riemannian setting}\label{app:ConvPseudoRiem} The definitions in this setting require additional data.

Let $M$ be  a pseudo-Riemannian manifold of signature $(p,q)$, and $N$ a submanifold of non-degenerate signature $(p',q')$. 

\begin{definition}\label{def:framing}
A \emph{normal framing} of $N$ is a smooth choice, for every $x$ in $N$, of a positive definite $(p-p')$-plane of the normal space $(\T_xN)^\bot$. We will denote $N^o$ the submanifold $N$ together with a normal framing.
\end{definition}

Remark that if $N$ reduces to a point $x$, then a normal framing $x^o$ is the choice of a positive definite $p$-plane of $\T_xM$.

We denote by $\GG_p(M)$ the Grassmannian of positive definite $p$-dimensional subspaces of $M$. By definition, $\GG_p(M)$ is a bundle over $M$ whose fiber over $x$ is the Grassmannian 
$\operatorname{Gr}_{p,0}(\T_xM)$ 
of positive definite $p$-planes in $\T_xM$ (thus the fiber is identified the symmetric space of $\SO(p,q)$). The tangent space of $\GG_p(M)$ at $(x,P)$ splits as
\[\T_{(x,P)}\GG_p(M) = \T_x M \oplus \Hom(P,P^\bot) = P \oplus P^\bot \oplus \Hom(P,P^\bot)\ .\]
The  Riemannian metric $g$ on $\GG_p(M)$ is given  by
$$
g_{(x,P)} = {g_x}_{\vert P} \oplus (-{g_x}_{\vert P^\bot}) \oplus h_P\ ,
$$
where $h_P$ is the Killing metric on the symmetric space $\operatorname{Gr}_{p,0}(\T_xM)$ evaluated at $P$ and $g$ the metric on $M$.

A normal framing of a  submanifold $N$  gives an embedding of $\GG_{p'}(N)$ into $\GG_p(M)$ whose image we denote by $\GG_{p'}^o(N)$. Given a triple $(x,N,M)$ where $N$ is a submanifold of $M$ with non-degenerate induced metric and $x$ belongs to $N$, we say that the normal framing $(x^o,N^o,M)$ is \emph{compatible} if the normal framing of $x$ contains the normal framing of $N$ at $x$.

We can thus define the notion of convergence of framed pseudo-Riemannian submanifolds by using the Riemannian metric on the Grassmannian bundle:

\begin{definition}
We say that the sequence $\{(x^o_k,N^o_k,M_k)\}_{k\in \N}$, where $x^o_k$ is a compatibly normally framed point of $N_k$, and $N^o_k$ is a normally framed submanifold of (non-degenerate) signature $(p',q')$ of the pseudo-Riemanniann manifold $M_k$ of signature $(p,q)$ {\em converges $C^n$},  if the corresponding sequence of pointed Riemannian submanifolds  $\{x^o_p,\GG^o_{p'}(N_k), \GG_{p'}(M_k)\}_{k\in\mathbb N}$ converges $C^n$.
\end{definition}
Note that this notion of convergence highly depends on the choice of normal framing. This choice of normal framing of $N$ is either vacuous or natural in two important cases that are used in this paper: in the case that $N$ is a spacelike surface in a pseudo-Riemannian manifolds of signature $(2,n)$, then the framing is trivial, and it is similarly trivial for spacelike curves equipped with a (spacelike) osculating plane. Accordingly we will not explicitly describe the normal framing in our definitions below that describe the convergence in these cases.

Recall that strongly positive curves in $\Hn$ is  in particular  {\em $2$-spacelike}: that is  the osculating plane $\T^{(2)}\gamma=\span\{\dot\gamma,\nabla_{\dot\gamma}\dot\gamma\}$ is spacelike everywhere.

\begin{definition}[\sc Spacelike surfaces and strongly positive curves]\label{def:bdgeom2}
Let $\seqk{M}$ be a sequence of pseudo-Riemannian manifolds of signature $(2,n)$, 
	\begin{enumerate}
	\item Let $N_k$ be spacelike surfaces in  $M_k$. In that case, we say 
	that

\centerline{\em $\{x_k,N_k, M_k\}_{k\in\mathbb N}$ converges to $\{x_\infty,N_\infty,M_\infty\}$}
	 
	  if $\{\T_{x_k} N_k,\T  N_k, \GG(M_k)\}_{p\in\mathbb N}$ converges to $\{\T_{x_{\infty}} N_{\infty}, \T N_{\infty}, \GG(M_{\infty})\}$.
	  
    \item Let  $\gamma_k$ be a spacelike curve in $M_k$ whose osculating plane $\T^{(2)}\gamma_k$ is spacelike everywhere. In that case again we say that 

	\centerline{\em$\{x_k,\gamma_k, M_k\}_{k\in\mathbb N}$ converges to $\{x,\gamma,M\}$} 

if $\left\{\T^{(2)}_{x_k} \gamma_k,\T^{(2)} \gamma_k, \GG(M_k)\right\}_{k\in\mathbb N}$ converges to $\left\{\T^{(2)}_{x} \gamma, \T^{(2)} \gamma, \GG(M)\right\}$.   
\end{enumerate} 

\end{definition}
Observe that the limits are  always spacelike. More precisely in the first item, since by definition $\T_{x_\infty} N_\infty$ is in $\GG(M_\infty)$,  $\T_{x_\infty} N_\infty$ is spacelike. The same holds for the curves  in the second item:   by the same argument $\T^2_x\gamma$ is spacelike,  hence $\T\gamma$ is spacelike.

As for the Riemannian setting, we can use this notion of convergence to define the notion of bounded geometry.

\subsection{Bounded Geometry}

A sequence of Riemaniann manifolds  $\{M_k\}_{k\in\N}$  {\em has $C^n$ bounded geometry} if for every  sequence of points $\seqk{x}$ with $x_k\in M_k$, then every subsequence  of $\{(M_k,x_k)\}_{k\in\mathbb N}$ subconverges $C^n$ to a pointed Riemannian manifold.

A sequence $\{(N_k,M_k)\}_{k\in\N}$  where $M_k$ is a Riemannian manifold and $N_k$ a submanifold in $M_k$ {\em has $C^n$ bounded geometry} if for every  sequence of points $\seqk{x}$ with $x_k\in N_k$, every subsequence  of $\{(x_k,N_k,M_k)\}_{k\in\mathbb N}$ subconverges $C^n$.

In the pseudo-Riemannian setting, we say  a sequence $\seqk M$ of pseudo-Riemannian manifolds has $C^n$ bounded geometry if the sequence of Riemannian manifolds $\{\GG_p (M_k)\}_{k\in\N}$ has bounded geometry. 

For pseudo-Riemannian submanifolds, we use a normal framing:  a sequence $\{(N^o_k,M_k)\}_{k\in \N}$ where $N_k$ is a submanifold of $M_k$ with normal framing $N^o_k$, has $C^n$ bounded geometry if for any sequence $\seqk {x^o}$ where $x_k$ is a point of $N_k$ with compatible normal framing $x^o_k$, every subsequence of $\{x^o_k, N^o_k,M_k\}_{k\in \N}$ subconverges $C^n$. 

Accordingly, using Definition \ref{def:bdgeom2}, we extend this definition to say that $\{(N_k,M_k)\}_{k\in\mathbb N}$ has bounded geometry when $N_k$ is a spacelike surface , or a 2-spacelike curve, in a pseudo-Riemaniann submanifold $M_k$ of signature $(2,n)$.

The following is a direct consequence of the previous definitions:

\begin{lemma}\label{l:BoundedGeometry}
For each $k$, consider $N_k^o$ a normally framed submanifold of $M_k$, and let $\gamma_k^o$ be a normally framed submanfiold of $N_k$. Then if $\{\gamma_k^o,N_k\}_{k\in\N}$ and $\{N_k^o,M_k\}_{k\in\N}$ have bounded geometry, then so does $\{\gamma_k^{oo},M_k\}_{k\in\N}$, where $\gamma_k^{oo}$ is the submanifold $\gamma_k$ of $M_k$ equipped with the normal framing induced by $\gamma_k^o$ and $N_k^o$.
\end{lemma}

\section{A lemma in plane topology}

\begin{lemma}\label{lem:plane}
	Let $D_1$ and $D_0$ be two smoothly embedded closed disks in the plane so that $D_1$ is embedded in the interior of $D_0$. Let $x_1$ be a point in $D_1$. Let $\gamma:[0,1]\to D_0$ be an arc smoothly embedded in $D_0$ with $\gamma\{0,1\}\in \partial D_0$ and $\gamma$ transverse to $\partial D_1$. 
	Then there exists a disk $U$ embedded in the interior of  $D_0$, so that , 
	\begin{enumerate}
		\item $U$ contains the connected component of $x_1$ in $D_1\setminus \gamma$, 
		\item $\partial U=\gamma_1\cup \eta$, where $\gamma_1$ is a connected sub arc of $\gamma$, and $\eta$ a connected sub arc of $\partial D_1$
			\end{enumerate}
\end{lemma}
The smoothness and transversality hypothesis are quite possibly not necessary but are enough for our purpose and simplify the argument.

\begin{figure}[!h] 
	\begin{center}
	\includegraphics[width=60mm]{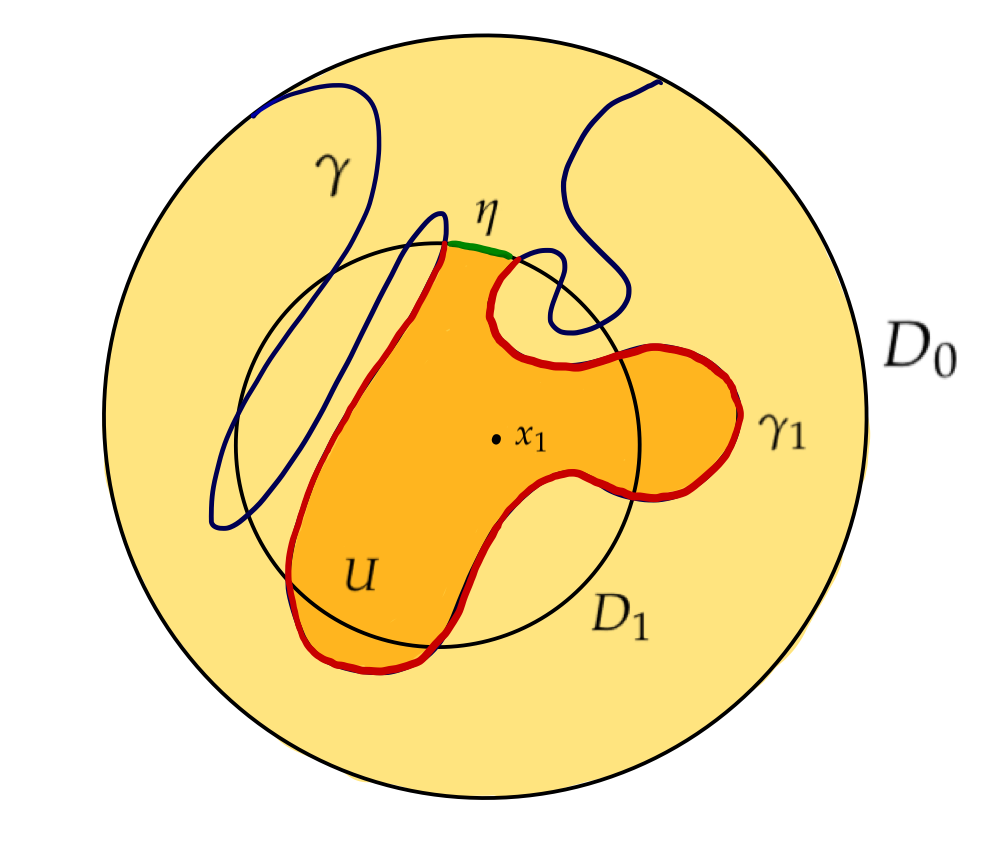}
	\caption{Lemma \ref{lem:plane}}
	\end{center}
\end{figure}

\begin{proof} Let $U_1$ be the connected component of $x_1$ in $D_1\setminus \gamma$. If $U_1=D_1$ then we take $U=U_1$ and the proof is over. Otherwise, let 
$$
t_0\defeq\inf\{t\mid\gamma(t)\in \overline U_1\}\ .
$$
 Let $y_0\defeq\gamma(t_0)$, then $D_0\setminus(\partial  D_1\cup \gamma)$ intersects a small enough disk neighbourhood of $y_0$, in four connected components. Only one of these connected  components, say $V_0$, is included in $U_1$. Let $\eta$ be the connected component of $\partial D_1\setminus\gamma$ so that a subarc of $\eta$ lies in the boundary of $V_0$. Let $y_1=\gamma(s_0)$ be the other extremity of $\eta$.

 By construction $\gamma[t_0,s_0]\cup \eta$ is an embedded arc, bounding a embedded disk $U_2$. Since $V_0\cap U_2$ is non-empty by construction, it follows that $U_1\subset U_2$. We can therefore take $U=U_2$, and this concludes the proof. \end{proof}

\section{Holomorphic curves and Gromov's Schwarz lemma}\label{app:phol}

We recall the basis of the theory of holomorphic curves and some results from \cite{Gromov} for the convergence portion and \cite{LabouriePseudoHolomorphic} for the part concerning the area. We give an improvement of these results as well as considerations of when a limit of immersions is an immersion.

\subsection{Preliminaries}

Recall that an {\em almost-complex structure} on a even dimensional manifold $M$ is a section $J$ of the bundle of endomorphisms of $\T M$ such that $J^2=-\Id$. When the real dimension of $M$ is equal to $2$, an almost complex structure is always integrable (that is, comes from an holomorphic atlas), and we call such a manifold $M$ a Riemann surface.

\begin{definition}
Given an almost complex manifold $(M,J)$, a {\em holomorphic curve} is a smooth map $f: (X,j) \to (M,J)$ where $(X,j)$ is a Riemann surface, and satisfying $\T f\circ j = J \circ \T f$.
\end{definition}

In this paper, we will be mostly interested in the case where $X=\D$ defined by
\[ \D = \{ z\in \C,~\vert z \vert < 1\}\ .\]
The {\it frontier} of $\D$ is
\[\Fr(\D) = \overline\D\setminus \D = \{ z\in \C,~\vert z\vert = 1\}\ .\]

\begin{definition}
A {\em totally real submanifold} of an almost-complex manifold $(M,J)$ (possibly with boundary) is a submanifold $W\subset M$ of half the dimension and such that for any $x$ in $W$, we have $\T_x M = \T_x W \oplus J(\T_x W)$.
\end{definition}

Our main focus is the {\em semi-disk $\S$} defined by
\[\S = \{z\in \D,~\Re(z)\geq 0\}\ .\]
We denote by
\[\partial \S = \S \setminus \text{int}(\S) = \{ z\in \D,~\Re(z)=0\}\ ,\]
\[\Fr(\S) = \overline{\S} \setminus \S = \{ z\in \Fr(\D),~\Re(z)\geq 0\}. \]
In this case, $\partial \S$ is a totally real submanifold of $\S$.

\begin{definition}
A {\em holomorphic curve with boundary} in an almost complex manifold $M$ with totally real submanifold $W$ is a holomorphic curve $f$ from $\S$ to $M$ mapping $\partial \S$ to $W$.
\end{definition}

\begin{definition}
Let $(M,J)$ be a almost complex manifold equipped with a Riemannian metric $\braket{.\mid .}$. An open set $U$ in $M$ is {\em $K$-calibrated} if there exists a 1-form $\beta$ so that 
\begin{eqnarray}
\forall u\in \T U, \ \ 
 \beta(u)^2\leq K^2\cdot \braket{u\mid u}\ , & &
	\braket{u\mid u}\leq K\cdot {\rm d}\beta(u,Ju)\ . \label{ineq:calib}
\end{eqnarray}
If $W$ is a totally real, totally geodesic submanifold, we furthermore assume that $\beta$ vanishes along $W\cap U$.
\end{definition}

Then we have the following lemma

\begin{lemma}[\sc Local calibration]\label{lem:calib}
There exists positive constants $\epsilon$ and  $K$, only depending on the geometry of $(M,x)$ so that  the ball $B$ of radius $\epsilon$ in $M$ of center $x$ is $K$-calibrated.
\end{lemma}
We sketch a proof since the extension incorporating the totally real submanifold $W$ is not in the original paper \cite{GSL}.
\begin{proof} Let  $\exp:\T_x M \to M$ be the exponential map. Observe that $\T_0\exp$ is  holomorphic. We choose $\beta=\exp_*\lambda$, where $\lambda\in\Omega^1(\T_x M)$ is defined by  $\lambda_u(v)=\braket{u\mid Jv}$. Since the preimage by $\exp$ of $W$ is  $\T_x W$ since $W$ is totally geodesic, and $\T_x W$ is a totally real  submanifold (actually linear) of $\T_x M$ since $W$ is, it follows that $\lambda$ vanishes on $\T_x W$ and thus $\beta$ vanishes on $W$. The result follows.
\end{proof}

\begin{rmk}
The notion of holomorphic curve can be extended to the following case. The (not necessarily even-dimensional) manifold $M$ carries a distribution $\mathcal{D}\subset \T M$ equipped with an almost complex structure $J: \mathcal{D} \to \mathcal{D}$. We then ask that a holomorphic curve is a map $f$ such that $\T f$ takes values in $\mathcal{D}$ and intertwines the almost complex structures.

In this framework, a totally real submanifold is a submanifold $W$ of $M$ of half the dimension of the distribution $\mathcal{D}$ and such that for any $x$ in $W$ we have $\mathcal{D}_x = \T_x W \oplus J(\T_x W)$.

All of the results described in the sequel canonically extend to this case.
\end{rmk}

\subsection{Schwarz lemmas and convergence of holomorphic curves} 
In this subsection, we state and sketch the proofs of two of our main goals for this appendix. After a few definitions, we state the results, then collect some preliminaries before concluding with a description of the arguments.

 To begin, let $f$ be a map from ${\bf D}$ or  ${\bf S}$ to $M$. If $Z$ is a subset of ${\Fr} ({\bf D})$ or respectively ${\Fr}({\bf S})$,  we denote by $f(Z)$ the accumulation set of sequences $\{f(y_k)\}_{k\in\N}$ where $\seqk{y}$ is a sequence converging to a point in $Z$.

Let $\seqk{M}$ be a sequence of complete almost complex Riemannian manifolds with bounded geometry, let $\seqk{x}$ be a sequence of points. We assume that $x_k\in M_k$ and that  $\{(x_k,M_k)\}_{k\in\mathbb N}$ converges $C^{p,\alpha}$ as almost complex Riemannian manifolds.

\begin{theorem}[\sc Free boundary]\label{theo:holoFB} Let $\seqk{U}$ be a sequence, where $U_k$ is an open set in $M_k$, uniformly calibrated and with bounded geometry. Let  $\seqk{f}$ be a sequence of holomorphic maps from ${\bf D}$ with values  in $U_k$.
Then $\seqk{f}$ subconverges  $C^{p,\alpha}$ on every compact set to $f_0$. Assume furthermore that
\begin{eqnarray}
\sup\left\{\area(f_k({\bf D})),~k\in\N\right\}< \infty\ .
\end{eqnarray} Then  for every non-empty open subset   $Z$ of ${\Fr}{(\bf D})$,
$$f_0({ Z})\subset\lim_{k\to\infty} f_k({Z})\ .$$
\end{theorem}

Assume now that $W_k$ is a totally real submanifold and totally geodesic submanifold of  $M_k$ containing $x_k$ and that $\{(x_k,W_k,M_k)\}_{k\in\N}$ converges   $C^{p,\alpha}$ to $\{(x_0,W_0,M_0)\}$, with $W_0$ totally real. In the boundary case, the following is an extension of \cite{LabourieGAFA}.

\begin{theorem}[\sc Boundary]\label{theo:holoBB} Let $\seqk{U}$ be a sequence of open sets in $M_k$, uniformly calibrated and with bounded geometry. Let  $\seqk{f}$ be a sequence of holomorphic maps from ${\bf S}$ with values  in $U_k$.
 Then $\seqk{f}$ subconverges $C^{p,\alpha}$ on every compact set to $f_0$. Assume furthermore that
\begin{eqnarray}
 \sup\left\{\area(f_k(\S)),~k\in\N\right\}< \infty\ .
\end{eqnarray} Then for every non-empty open subset  $Z$ of ${\Fr}({\bf S})$, we have
$$f_0(Z)\subset\lim_{k\to\infty} f_p(Z)\ .$$\end{theorem}

Both these theorems represent an improvement over the corresponding earlier results which only considered the case $Z={\Fr}({\bf D})$ or $Z=\partial {\bf S}$. 

\subsubsection{Quadrangles and extremal length}\label{sec:quadrangle} We prepare for the proof by recalling a classical construction.

A {\em quadrangle} $Q\defeq (U,x_1,x_2,x_2,x_3)$ in $\mathbb C$ equipped with the complex structure $J$ is a topological disk $U$  with four 
marked points $(x_1,x_2,x_3,x_4)$ in cyclic order in $\partial U$. The {\em $a$-rectangle} is the rectangle  $R_a\defeq(R,a,a+i,i,0)$ of vertices $(0,a,a+i,i)$. Two quadrangles $(U,x_1,x_2,x_2,x_3)$ and $(V,y_1,y_2,y_2,y_3)$ are conformally equivalent if we can find a conformal mapping  $\Phi$ sending $\overline U$ to $\overline V$ so that $\phi(x_i)=y_i$. 
 Every quadrangle is conformally equivalent to a unique $a$-rectangle.

Let $Q$ be a quadrangle and $\Gamma_Q$ be the set of arcs in $\overline U$ joining a point in the  interval between $x_1$ to $x_2$ on $\text{Fr}(U)$ to a point in the  interval between $x_3$ to $x_4$ on $\text{Fr}(U)$.
The {\em extremal length} of $Q$ is $\mathcal L(Q)$ where for a metric $g$
\begin{eqnarray}
	\mathcal L_g(Q)&=&\frac{\inf\left\{\operatorname{length}_g^2(\gamma)\mid \gamma\in\Gamma_Q\right\}}
	{\area_g(Q)}\\
	\mathcal L(Q)&=&\sup\left\{\mathcal L_g(Q)\mid g \hbox{ conformal to } J \right\}\ ,
\end{eqnarray}
where $\operatorname{length}_g$ and $\area_g$ denotes respectively the length and area with respect to $g$.
By construction $\mathcal L(Q)$ is a conformal invariant. A classical result asserts

\begin{proposition}
	We have $\mathcal L(R_a)=a$.
\end{proposition} 

Let $Z_0$ be the {\em standard quarter} in $S^1$, that is the subarc of $S_1$ between $1$ and $i$. Let $Q_0$ be the standard sector
$$
Q_0\defeq\{z\in \mathbf D\mid \Re(z)\geq0\ ,\  \Im (z)\geq0\ , \vert z\vert <1\}\ ,
$$
so that $\Fr(Q_0)=Z_0$.

Let $\ell(R)$ be the Euclidean distance from $0$ to a point $x$ in $\bf D$ at hyperbolic distance $R$ from $0$. (Of course, it is classical that $\ell(R) = \tanh(R)$, but the precise formula is not important for our discussion.)
We define the {\em  $R$-corner quadrangle} to be the quadrangle $Q_R=(U,i\ell(R), \ell(R),1,i)$ where $U=\{z\in  \mathbb C\mid \ell(R)<\vert z\vert<1, \Re(z)>0, \Im(z)>0\}$. 

One then checks, for example by applying the conformal map $z \mapsto \log z$ to the domain $Q_R$, that

\begin{proposition} 
	The map $K:R\to\mathcal L(Q_R)$, is a decreasing  homeomorphism from $(0,\infty)$ to $(0,\infty)$.
\end{proposition}

The following lemma is a consequence of the previous discussion and is used in the sequel.

\begin{lemma}\label{lem:bdry} For any positive $A$ and $\epsilon$, there exists a positive constant $\rho$, with the following property.
 Assume   $f$ is  a holomorphic map from  $Q_0$   to an almost complex manifold $M$. Assume that $\area(f(Q_0))\leq A$ and 
  $\Vert \T f\Vert$ is bounded by $\frac{\epsilon}{\rho}$ on the ball of radius $\rho$ (in $Q_0$)  (with respect to the hyperbolic metric on  $\bf D$).
 Then
$
d(f(0),f(Z_0))\leq 2\epsilon
$.
\end{lemma}

\begin{proof}
We choose $\rho$ so that 
$
	A\cdot {\mathcal L}(Q_{\rho})\leq \epsilon^2
$.
Let $g$ be the induced metric by $f$ from $M$. Observe that for any curve $\gamma$ in $\Gamma_{Q_\rho}$, 
\begin{eqnarray*}
	\operatorname{length}_g(\gamma)&\geq& d(f(\gamma(0),f(\gamma(1))\geq d(f(0),f(Z_0))-d(f(0),f(\gamma(0))\\ &\geq&d(f(0),f(Z_0))-\epsilon\ ,
\end{eqnarray*}
where the last inequality uses the fact that $\Vert\T f\Vert$ is bounded by $\frac{\epsilon}{\rho}$ on $B(0,\rho)$.
Thus
$$
\epsilon^2\geq {\mathcal L(Q_\rho)}\cdotp\area_g(Q_\rho)\geq (d(f(0),f(Z_0))-\epsilon)^2\ .
$$
The result follows.
\end{proof}

\subsubsection{Sketch of  the proof of  the first part of Theorem \ref{theo:holoFB}}

Without the hypothesis on the area, the subconvergence is consequence of the celebrated Gromov's Schwarz Lemma \cite{Gromov,GSL} which states that the derivatives of $f_k$ are uniformly bounded. We sketch the argument, since we are going to sketch a modification of it.
We need three preliminary lemmas. 
In the first two, $M$ is a manifold equipped with an almost complex structure $J$ and a compatible metric $\braket{.\mid.}$ (that is a metric for which $J$ is an isometry).
First we have  (see \cite{GSL})

\begin{lemma}[\sc Weingarten lemma]\label{lem:wl}
Let $\Sigma$ be a holomorphic curve in $M$, and let $x\in \Sigma$. Then there is a bound  $K$ only depending on the geometry of $(M,x)$ so that the curvature of $\Sigma$ at $x$ is less than $K$.
\end{lemma}

Our second lemma from \cite{GSL} is
 
\begin{lemma}[\sc Gromov's Schwarz lemma]\label{lem:gsl}
Let $g$ be a conformal metric on the disk. Let $g_0$ be the hyperbolic metric and $h$ the conformal factor so that $g=h g_0$ Assume that the curvature of $g$ is bounded from above by $K$, and that $g$ satisfies a  linear isoperimetric inequality, that is for any disk $A$ embedded in $\D$, we have
\begin{eqnarray}
	\area_g(A)\leq \K \operatorname{length}_g(\Fr(A))\ . \label{ineq:gsl}
\end{eqnarray}
Then there exists a bound $K_0$ only depending on $K$ so that $h\leq K_0$.
\end{lemma}

Combining these two lemmas gives the celebrated

\begin{lemma}[\sc Gromov's  Holomorphic Schwarz lemma]\label{lem:ghsl} Let $K_1$ be a positive constant.
There exists a  positive constant $K_0$, only depending on the local geometry 
of $(M,x)$ and $K_1$ so that if $\varphi$ is a holomorphic map from the hyperbolic disk  ${\bf D}$ to a $K_1$-calibrated open set with bounded curvature	then 
$$\Vert \T\varphi\Vert\leq K_0\ .$$
\end{lemma}

\begin{proof} By  replacing  $\varphi$ by the graph map $\varphi'=(\varphi,\Id):~{\bf D}\to M\times {\bf D}$ we may assume that $\varphi$ is an immersion.

We consider the induced metric $g$ by $\varphi$. By the Weingarten Lemma \ref{lem:wl} the curvature of $g$ is bounded from above. From the definition of calibration, the metric $g$ satisfies a linear isoperimetric inequality:
$$
\area(A)\leq K_1 \int_{A}{\rm d}\beta= K_1\int_{\partial A}\beta  \leq K_1^2 \ \ \operatorname{length}(\Fr(A))\ .
$$
Thus the result follows by Gromov's Schwarz lemma \ref{lem:gsl}.
\end{proof}

 The strengthening of the first conclusion of Theorem~\ref{theo:holoFB}  with the hypothesis on the area is an extension of \cite[Lemme 6.8]{LabouriePseudoHolomorphic}. 
 This will be proved in the last paragraph of this section.

\subsubsection{Sketch of  the proof of the first part of  Theorem \ref{theo:holoBB} }

We proceed as in \cite[Lemme 9.1]{LabourieGAFA}. 
 We have
\begin{proposition}
	Let $\varphi$ be a holomorphic immersion from ${\bf S}$ to a Riemannian almost complex manifold $M$  equipped with a compatible metric, so that $f(\partial {\bf S})$ lies in a totally real totally geodesic submanifold $W$. Then $\partial {\bf S}$ is totally geodesic for the metric induced by $\varphi$.
\end{proposition}
\begin{proof}  Let $J$ be the complex structure of $M$ and  $\braket{.\mid.}$ the compatible metric. Let $\gamma$ be an arc length parametrisation of $f(\partial {\bf S})$.  The geodesic curvature of  $\gamma$  is $\braket{\nabla_{\dot\gamma}\dot\gamma\mid J\dot\gamma}$. Since $\gamma=f(\partial {\bf S})$ is embedded in a totally geodesic submanifold $W$, $\nabla_{\dot\gamma}{\dot\gamma}$ and ${\dot\gamma}$ lie in $\T W$. Since  $W$ is totally real, for all $u$ and $v$ in $\T W$, we have $\braket{u\mid Jv}=0$, and the result now follows.
\end{proof}

Combining this lemma with the previous arguments, we obtain

\begin{lemma}[\sc Holomorphic Schwarz lemma with boundary]\label{lem:bsl}
There exists positive constants $\epsilon$ and  $K_0$, only depending on the geometry of $(M,x)$ so that if $\varphi$ is a holomorphic map from the  semi-disk  ${\bf S}$ into the ball of radius $\epsilon$ centred at $x$,	then 
$$\Vert \T\varphi\Vert\leq K_0\ .$$
\end{lemma}
\begin{proof}
After replacing $\varphi$ by  the graph map $\varphi'=(\varphi,\Id)$ from ${\bf S}$ to $M\times {\bf D}$ and $W$ by $W'=W\times\partial \S$, we may assume that $\varphi$ is an immersion.

Let $g_0$ be the hyperbolic metric on ${\bf S}$ and $g=hg_0$ the metric induced by $\varphi$. By the Weingarten Lemma, the curvature of $g$ is bounded from above. Since $\partial \S$ is totally geodesic for $g$, we can double $g$ to obtain a $C^0$ metric $g_2$ on ${\bf D}$. By the doubling argument and since $\partial \S$ is totally geodesic, the curvature of $g_2$ is also bounded from below.

To conclude the proof using Gromov's Schwarz Lemma one needs to show that $g_2$ satisfies a linear isoperimetric inequality. Let us use the form $\beta$ obtained from the Local Calibration Lemma \ref{lem:calib}. Let $A$ be a disk in ${\bf D}$, write $A=A_0\cup A_1$, where $A_1=A\cap {\bf S}$. Then by Stokes formula
$$
\area(A_1)\leq K \int_{A_1}{\rm d}\beta=K\int_{\Fr(A_1)}\beta=K\int_{\Fr(A)\cap \overline{A_1}}\beta \leq K^2 \operatorname{length}(\Fr(A))\ ,
$$
where the first equality follows from the fact that $\beta=0$ on $W$, hence on $\partial \S$.  Repeating the argument for $A_0$ leads the desired linear isoperimetric inequality for $g_2$.
\end{proof}

\subsubsection{Improving regularity}

Gromov's Schwarz Lemma gives uniform $C^1$-bounds on the sequence $\seqk{f}$. We need to improve this and proves the $C^{p,\alpha}$ convergence to obtain the first part of Theorems \ref{theo:holoFB} and Theorem \ref{theo:holoBB}.

 This is done in two steps. As a preliminary,  we choose  $C^{p,\alpha}$ local coordinates on $M_k$ so that $M_k$ is identified with $\mathbb C^n$ and $W_k$ with $\mathbb R^n$. Thanks to our $C^1$ bounds, we reduce to the case (by possibly shrinking the source) to  bounded maps $f_k$ with values in $\mathbb C^n$. The holomorphic curve condition then reduces to the equation
\begin{eqnarray}
	\partial_y f_k= J_k(f_k) \partial_x f_k\ , \label{eq:delbar}
\end{eqnarray}
where $x$ and $y$ are the coordinates on $\bf D$ or $\bf S$ and  $\seqk{J}$ converges in $C^{p-1,\alpha}$. When present, the boundary condition is 
$$
f_k(\partial S)\subset \mathbb  R^n\ .
$$ 
The two steps of our regularity improvement are as follows.
\vskip 0.2truecm 
\noindent{\em Uniform $C^2$-bounds:} we consider the 1-jet map $g_k=(f_k,  \partial_y f_k)$ with values in $\mathbb C^{2n}$ satisfying the boundary condition 
$$
g_k(\partial S)\subset \mathbb  R^{2n}\ .
$$ 
A derivation of \eqref{eq:delbar} gives that $g_k$ is holomorphic for a certain complex structure $J'_k$ on $\mathbb C^{2n}$. 
Indeed, erasing for a moment the index $k$ to have readable formulas, we claim:
\begin{eqnarray*}
	\partial_y g= J'(g) \partial_x g\ , 
\end{eqnarray*}
where for $(s,t)$ in $\mathbb C^n\times\mathbb C^n$,
$$
J'(s,t)(u,v)=\left(J(s)u,J(s)v + A(s,t)(u)\right)\ , 
$$
where $A(s,t)\defeq({\rm D}_{s}J)(t)$ is the derivative of $J$ at $s$ in the direction of $t$. Since $A(s,t)$ anticommutes with $J(s)$, we see that 
$$
J'^2(u,v)=(-u, -v + JA (u) + AJ(u))=-(u,v)\ .
$$
Gromov's holomorphic Schwarz Lemma then gives a uniform $C^1$-bound on $g_k$, hence the desired $C^2$-bound on $f_k$.

\vskip 0.2truecm 
\noindent{\em  $C^{p,\alpha}$-convergence:} Now that we have $C^2$ bounds we can return to the equation \eqref{eq:delbar} with the information that $f_k$ is in $C^2$ and in particular in $C^{1,\alpha}$, knowing that $J_k$ converges in $C^{p-1,\alpha}$. 

The proposition is then 
\begin{proposition}\label{pro:1->2}
Assume that $u$ is in $C^{1,\alpha}$ and satisfies the  equation 
\begin{eqnarray}
	\partial_y u= J(u) \partial_x u\ , \label{eq:delbar00}
\end{eqnarray}
	with possibly the boundary condition $u(\partial \bf S)\subset \mathbb R^n$, where $J$ is in $C^{1,\alpha}$. Then $u$ is in $C^{2,\alpha}$. More precisely, for every positive constant $A$, there exist positive constants $C$ and $\epsilon$ so that if the $C^{1,\alpha}$ norm of $u$ is less than $A$, then the $C^{2,\alpha}$ norm of $u$ is less than $C$ on a ball of radius $\epsilon$.
\end{proposition}
\begin{proof} 
	We reproduce and  adapt the proof of Theorem A.2.1 in \cite{Abbas:2019} in two ways, first by using $C^{1,\alpha}$ bounds rather than $W^{k,p}$ Sobolev norms, and second in assuming lower regularity of $J$.

 In this proof, the quantities  $C_i$ will be positive constants.
We may as well assume that $u(0)=0$ and $J(0)=i$ and restate equation \eqref{eq:delbar00} as
\begin{eqnarray}
	2i\bar\partial u= -i(i-J(u)) \partial_x u\ , \label{eq:delbar1}
\end{eqnarray}
	
We use the difference quotient technique and introduce for small $h$
\begin{eqnarray}
	u^h(x,y)\defeq \frac{1}{h}(u(x,y+h)-u(x,y))=\int_0^1\partial_yu(x,y+th)\ {\rm d}t\ .\label{eq:uh}
\end{eqnarray}
Moreover
\begin{eqnarray}
	0&=&\partial_y u^h - (J(u) \partial_x u)^h\cr
&=&\partial_y u^h - J(u)\partial_x u^h -J(u)^h \partial_x u(x,y+h)\cr
&=&2i\overline\partial u^h + i(J(u)-i)\partial_y u^h -J(u)^h \partial_x u(x,y+h)\ .
\label{eq:delbar3}
\end{eqnarray}
Let also $\beta$ be a bell function in $\mathbb R$
with $\beta[0,1/2]=1$,  $\beta[1,+\infty[=0$, $\beta'(s)\leq 0$. We define
$$
\beta_\epsilon(x,y)=\beta\left(\frac{x^2+y^2}{\epsilon}\right)\ .
$$
We now obtain, from \eqref{eq:delbar3} and the Leibniz rule $\beta_\epsilon \partial_y(u^h) =  \partial_y(\beta_\epsilon u^h)- \partial_y(\beta_\epsilon) u^h$,
\begin{eqnarray}
2i\overline\partial (\beta_\epsilon u^h) &=&  -i(J(u)-i)\partial_y (\beta_\epsilon u^h) +i(J(u)-i)(\partial_y \beta_\epsilon) u^h \cr &+& \beta_\epsilon J(u)^h \partial_y u(x,y+h) + (\overline\partial \beta_\epsilon) u^h)\ .\label{eq:beta-u}
\end{eqnarray}
Let us denote by $\Vert v\Vert_{p,\alpha,\epsilon}$ the $C^{p,\alpha}$ norm on the ball of radius $\epsilon$ while using the shorthand $\Vert v\Vert_{p,\alpha}=\Vert v\Vert_{p,\alpha,1}$. 

Let us makes a series of estimates
\begin{enumerate}
	\item In the equation above,  let us consider the term 
$$
B= 2i(\overline\partial \beta_\epsilon) u^h +i(J(u)-i)(\partial_y \beta_\epsilon) u^h.
$$
We have a constant $C_\epsilon$ depending only on $\epsilon$ and $\beta$,  so that 
\begin{eqnarray}
\Vert B\Vert_{0,\alpha}&\leq& C_\epsilon  \Vert u^h\Vert _{0,\alpha,\epsilon}\ . \label{eq:delbar9}
\end{eqnarray}
\item Restricting $\epsilon$ so that $\epsilon+h<1$, we have
\begin{eqnarray}
		\Vert (i-J(u))( \partial_y \beta_\epsilon u^h)\Vert _{0,\alpha,\epsilon}  &\leq & C_3 \Vert u\Vert_{0,\alpha,\epsilon} \cdot \Vert\beta_\epsilon u^h\Vert_{1,\alpha}\ .\label{eq:delbar5}
\end{eqnarray}
\item Since $J$ is in $C^{1,\alpha}$ and in particular uniformly Lipschitz, we have a uniform constant $C_0$ so that
$$
\vert J(u)^h\vert \leq C_0\vert u^h\vert\ .
$$
Thus 
\begin{eqnarray}
\Vert\beta_\epsilon J(u)^h \partial_yu(x+h,y) \Vert_{0,\alpha,\epsilon}&\leq& C_1 \Vert u\Vert _{1,\alpha,\epsilon} \cdot  \Vert  u^h\Vert _{0,\alpha,\epsilon}\ . \label{eq:delbar6}
\end{eqnarray}

\end{enumerate}

 Combining the estimates \eqref{eq:delbar9}, \eqref{eq:delbar5} and  \eqref{eq:delbar6} with our original equation \eqref{eq:beta-u}, we obtain that 
\begin{eqnarray}
	\Vert\overline\partial (\beta_\epsilon u^h)\Vert_{0,\alpha,\epsilon}&\leq& C_\epsilon  \Vert u^h\Vert _{0,\alpha} +C_1 \Vert u\Vert _{1,\alpha,\epsilon}  \Vert u^h\Vert _{0,\alpha,\epsilon}\cr & &+ 
	C_3 \Vert u\Vert_{0,\alpha,\epsilon} \cdot \Vert u^h\Vert_{1,\alpha}+ C_1 \Vert u\Vert _{1,\alpha,\epsilon} \cdot  \Vert \beta_\epsilon u^h\Vert _{0,\alpha}\ . \label{eq:delbar7}
\end{eqnarray}

Recall also that from the Cauchy--Pompeiu formula  (see \cite[Theorem 4.7.1]{AIM09} for a model) we have the estimates
\begin{eqnarray*}
	 \Vert \beta_\epsilon u^h \Vert_{1,\alpha}  &\leq &  C_4 \left\Vert  \bar\partial  (\beta_\epsilon u^h)\right\Vert_{0,\alpha}\ . \label{eq:delbar8}
\end{eqnarray*}
We now fix $\epsilon$, so that $C_3 C_4 \Vert u\Vert_{0,\alpha,\epsilon}\leq \frac{1}{2}$. Then for some constant $D=D_{\epsilon}$ depending on $u$  and independent of $h$, we find
\begin{eqnarray*}
	 \frac{1}{2}\Vert u^h\Vert_{1,\alpha,\frac{\epsilon }{2}}\leq  \frac{1}{2}\Vert \beta_\epsilon u^h \Vert_{1,\alpha}  &\leq &  D   \Vert   u^h\Vert_{0,\alpha}\ . \label{eq:delbar11}
\end{eqnarray*}
Observe that the same methods also yield that for all $\eta<\alpha$ we have 
\begin{eqnarray*}
	 \frac{1}{2}\Vert u^h\Vert_{1,\eta,\frac{\epsilon }{2}} &\leq &  D   \Vert   u^h\Vert_{0,\eta}\ . \label{eq:delbar111}
\end{eqnarray*}
Let us write $\psi^h\defeq u^h-\partial_y u$ and $M=\Vert u\Vert_{1,\alpha}$.  Using the integral form for $u^h$ in equation \eqref{eq:uh}, we observe that $\psi^h$ satisfies $\Vert \psi^h \Vert_{0,\alpha}\leq 2M$ and  
 $\vert \psi^h(x)\vert\leq  2M h^\alpha$. 

 This implies that for all $\eta<\alpha$, we have 
 $$
 \vert\psi^h(z)-\psi^h(w)\vert\leq2M\cdotp\min\{ h^\alpha, \vert z-w\vert^\alpha\}\leq 4M\cdot h^{\alpha-\eta}\vert z-w\vert^\eta\ .
 $$
 Thus
 $\Vert\psi^h\Vert_{0,\eta}\leq 4Mh^{\alpha-\eta}$. Thus 
$u^h$ converges to $\partial_y u$ in $C^{0,\eta}$ for all $\eta<\alpha$, and
 then
\begin{eqnarray*}
	 \frac{1}{2}\Vert \partial_y u\Vert_{1,\eta,\frac{\epsilon }{2}} &\leq &  D   \Vert   \partial_y u\Vert_{0,\eta}\ . \label{eq:delbar111}
\end{eqnarray*}
Then taking the limit when $\eta$ goes to $\alpha$, and using that $\partial_y u \in C^{0,\alpha}$, we find 
\begin{eqnarray*}
	 \frac{1}{2}\Vert \partial_y u\Vert_{1,\alpha,\frac{\epsilon }{2}} &\leq &  D   \Vert   \partial_y u\Vert_{0,\alpha}\ . \label{eq:delbar112}
\end{eqnarray*} 
Thus $u$ is in $C^{2,\alpha}$ in the ball of radius $\epsilon/2$.\end{proof}
\vskip 0.1 truecm 
\noindent{\em Bootstrap and regularity:} We can now conclude the argument by showing  that if $\seqk{u}$ is a sequence of solutions in $C^{p+1,\alpha}$ satisfying 
$$
\partial_y u_k=J_k(u_k)\partial_x(u_k)\ ,
$$
where 
\begin{itemize}
	\item $\seqk{J}$ converges in $C^{p,\alpha}$ to $J_0$ for which $\mathbb R^n$ is totally real,
	\item $\seqk{u}$ has uniform $C^1$-bounds and converges $C^0$ to $u_0$,
\end{itemize}
then $\seqk{u}$ converges in $C^{p+1,\alpha}$.

This is obtained via the bootstrap described in the first step, or equivalently as in  \cite[Theorem A.2.1]{Abbas:2019} which immediately leads, using proposition \ref{pro:1->2} to 
\begin{proposition}
Assume that $u$ is in $C^{k,\alpha}$ satisfies the equation \eqref{eq:delbar}
\begin{eqnarray*}
	\partial_y u= J(u) \partial_x u\ , 
\end{eqnarray*}
	with possibly the boundary condition $u(\partial \bf S)\subset \mathbb R^n$, where $J$ is in $C^{p,\alpha}$, then $u$ is in $C^{p+1,\alpha}$. More precisely, for every positive constant $A$, there exist positive constants $C$ and $\epsilon$ so that if the $C^{p,\alpha}$ norm of $u$ is less than $A$, then the $C^{k+1,\alpha}$ norm of $u$ is less than $C$ on a ball of radius $\epsilon$.
\end{proposition}

\subsubsection{Using the hypothesis on the area}
We now show the second part of Theorem \ref{theo:holoFB} and Theorem \ref{theo:holoBB}.

\begin{proof} Using the Schwarz Lemma, we can extract in both cases  a  subsequence so that $\seqk{f}$ subconverges to $f_0$. 

Observe that we have a constant $A$ so that for all subsets $U$ of $\bf D$, or $\bf S$, then
\begin{eqnarray}
\area(f_0(U))\leq A\ .\label{def:Aaire}	
\end{eqnarray} 

Let $\seqk{y}$ be a sequence in $\bf D$ or $\bf S$ converging to a point $y_0$ in an interval $Z$ in ${\Fr}(\bf D)$.

We want to show that there exists a sequence  $\{f_{n_k}\}_{k\in\mathbb N}$ for which we have   $$\lim_{k\to\infty}(d(f_0(y_k), f_{n_k}(Z))=0\ .$$ 

From the bound on the area of $f_0({\bf D})$, we have that for all $R$
 $$
 \lim_{k\to\infty}(\area(f_{0}(B(y_k,R)))=0\ .
 $$
 where $B(y_k,R)$ is the ball of radius $R$ in the hyperbolic metric.

 Using the fact that $\seqk{f}$ converges on every compact to $f_0$, we can choose a subsequence $\{f_{n_k}\}_{k\in\mathbb N}$, so that
  \begin{eqnarray}
  	\area (f_{n_k}(B(y_k,1))\leq  \frac{1}{k}\label{ineq:areaphl}\ &,&
  	d(f_0(y_k),f_{n_k}(y_k))\leq \frac{1}{k} \ .
  \end{eqnarray}  
  
  Let $u_k$ be a conformal mapping of $\bf D$ that sends $0$ to $y_k$.
  
  Since $\{y_k\}_{k\in\mathbb N}$ converges to an interior point of $Z$, the sequence $\{u_k^{-1}(Z)\}_{k\in\mathbb N}$ converges to the full boundary of ${\bf D}$. We can thus  choose for each $k$,  a subinterval  $Z_k$ in $Z$ so that the preimage of $Z_k$ by $u_k$ is a quarter of circle $Z_0$.
  
  After precomposing $u_k$ with a rotation, we may furthermore assume that the preimage of $Z_k$ is the 
  standard quarter of circle.
  
  Let then  $Q_{0}$ (as defined in the beginning of paragraph \ref{sec:quadrangle}) be the standard sector. We furthermore choose $Z_k$ so that  $Q_0$ is a subset of $u_k^{-1}({\bf S})$ in the boundary case.   

To conclude the theorem it will be enough to prove
\begin{eqnarray}
	\lim_{k\to\infty}d(f_{n_k}(y_k),f_{n_k}(Z_k))=0. \label{ass:111}
\end{eqnarray}
Let then $g_k=f_{n_k}\circ u_k$.  
Assertion \eqref{ass:111} is now restated as\begin{eqnarray}
	\lim_{k\to\infty}d(g_k(0),g_k(Z_0))=0\ .\label{asser:phol2}
\end{eqnarray}
Applying Schwarz Lemma again,  $\seqk{g}$ subconverges to some $g_0$.  By inequality \eqref{ineq:areaphl}, the area of $g_0(B(0,1))$ is equal to $0$, thus $g_0$ is constant.

Let us choose a positive $\epsilon$. Let $A$ be  the bound of the area defined in inequality \eqref{def:Aaire} and $\rho$ as in Lemma \ref{lem:bdry}. Since $\seqk{g}$ converges uniformly on every compact to the constant map $g_0$, it follows that for $k$ large enough $\Vert\T g_k\Vert$ is bounded above by $\frac{\epsilon}{\rho}$ on $B(0,\rho)$. We can thus conclude from  Lemma \ref{lem:bdry} that for $k$ large enough
$$
d(g_k(0),g_k(Z_0))\leq 2\epsilon\ .
$$
The assertion \eqref{asser:phol2} follows, hence the theorem. \end{proof}

\subsection{Immersions}
A non-constant limit of holomorphic immersions may not be immersed. We describe here certain situations in which a limit of immersions is an immersion. This result is a generalization of the case when the target is $\mathbb C$: roughly speaking the role of the 2-dimensional target is played by a complex line bundle $L_\mathbb C$, together with a never vanishing 1-form $\alpha$ with values in $L$. Our "immersion in $\mathbb C$" is now replaced by a holomorphic map $f$ so that $f^*\alpha$ is uniformly non-vanishing.

Let us  be more precise  about our hypothesis:
let $M$ be an almost complex manifold, $L$ be a real line bundle over $M$, $L_\mathbb C$ the complexification of $L$ and $\alpha\in \Omega^1(M,L_\mathbb C)$ a never vanishing 1-form with values in $L_\mathbb C$. 

We also choose a Hermitian metric $h$ on $L_\C$ as well as a unitary connection $\nabla^0$ on $L_\mathbb C$ for which $L$ is parallel.

When we have a boundary problem defined by a totally real submanifold $W$ we furthermore assume that $\alpha(\T W)= L$.
The result is the following version of a claim that a sequence of maps that are strongly immersed, in terms of the existence of one-form $\alpha$ which they all pull back in a non-degenerate way, limit on a map with the same immersivity property. 

\begin{theorem}\label{theo:imm-limi}\label{theo:imm-limi-bd}
Let $\seqk{f}$ be either ({\em free boundary case}) a  sequence of holomorphic maps from $\bf D$, or ({\em boundary case})  a sequence of holomorphic maps from $\S$ to $M$ so that $f_k(\partial \S)$ is included in $W$. Equip $(M,W)$ with a real line bundle $L$ and a one-form $\alpha$ with values in its complexification $L_\mathbb C$ as above, together with the chosen Hermitian metric $h$ and parallel unitary connection $\nabla^0$.

Assume that
\begin{enumerate}
	\item the sequence $\seqk{f}$ converges to $f_0$, 
	\item  for all $k$, $f_k^*\alpha$ never vanishes,
	\item we have a constant $K_0$ so that for all $k$ and 
\begin{eqnarray}
		\Vert f_k^*\de^{\nabla^0}\alpha\Vert&\leq& K_0 \Vert f_k^*\alpha\Vert^2\ ,\label{hyp:boundTf}
\end{eqnarray}
using the metric $h$ and some auxiliary metric on $M$.
\item assume  finally that $f_0^*\alpha$ is not identically zero.
\end{enumerate} Then $f_0^*\alpha$ never vanishes.
\end{theorem}
We remark

\begin{proposition}\label{pro:hypboundTf} Condition \eqref{hyp:boundTf} is satisfied  in the following two cases
\begin{eqnarray*}
	\de^{\nabla^0}\alpha=0\ ,  &\text{  or}&
	\Vert\T f_k\Vert\leq K_1 \Vert f_k^*\alpha\Vert\ .
	\end{eqnarray*}
\end{proposition} 

\begin{proof} The first case in Proposition~\ref{pro:hypboundTf} is obvious. The second case follows form the remark that if $\beta\in \Omega^k(V)$, then
 $
 \Vert g^*\beta\Vert\leq \Vert \T g\Vert^k\cdotp\Vert \beta\Vert$\ .
\end{proof}
The first case is satisfied when $\alpha=\de\pi$, where $\pi$ is a submersion in $\mathbb C$ which maps $W$ to a line, in which case the setting reduces to simply maps from $\bf D$ (or $\S)$ to $\mathbb C$.

\subsubsection{Preliminary controls} 
Let $f$ be a holomorphic map from $\bf D$ to $M$ satisfying inequality (\ref{hyp:boundTf}), and $g$ the quadratic form defined by
$$
g(X,Y)\defeq h(\alpha(\T f(X)),\alpha(\T f(Y)))\ ,
$$
Let $O$ be the open set in $\bf D$ on which $g$ is a metric, and let $u$ be  the   continuous vector fields  of norm 1 (defined up to sign)  on $O$ so that $\alpha(\T f(u))\in L$.

\begin{lemma}\label{lem:bdcon}
	Let  $\beta$ be the connection form for $u$ defined by 
	$$
	\beta(X)=g(\nabla_Xu,Ju)\ ,
	$$
	where $\nabla$ is the Levi-Civita connection of $g$.
	Then
	$$
	\Vert\beta(X)\Vert^2\leq K_1 g(X,X).
	$$
	where $K_1$ only depends on $K_0$, $\alpha$ and $\nabla^0$.
\end{lemma}
\begin{proof} Let us consider the induced bundle $L_0\defeq f^* L$, as well as the pull-back metric $h_0$, induced forms $\zeta=f^*\alpha\in\Omega^1({\bf D},L_0)$, and induced connection $D=f^*\nabla^0$. In the proof $k_i$ will be constants only depending on $\alpha$ and $\nabla^0$.
	The classical formula for the Levi Civita connection tells us that
	\begin{eqnarray*}
	& & 2 g(\nabla_Xu,Ju)\\&=&	u\cdot g(X,Ju)-Ju\cdot g(X,u)
	+X\cdot g(u,Ju)\\ & & -g(u,[X,Ju])-g(X,[u,Ju])+g(Ju,[X,u])\\
&=&u\cdot h_0\left(\zeta (X)),\zeta(Ju)\right)
	-Ju\cdot h_0\left(\zeta(X)),\zeta (u)\right)	+X\cdot h_0(\zeta(u),\zeta(Ju))
	\\ 
	& &-h_0\left(\zeta(u)),\zeta([X,Ju])\right)
	-h_0\left(\zeta(X),\zeta([u,Ju])\right)+
h_0\left(\zeta( Ju)),\zeta([X,u])\right)\\
&=&
2h_0(D_X(\zeta(u)), \zeta(Ju))\\
& &+h_0(\zeta(u),{\rm d}^{D}\zeta(X,Ju))-h_0(\zeta(Ju),{\rm d}^{D}\zeta(X,u))
+h_0(\zeta(X),{\rm d}^{D}\zeta(u,Ju)) \ .
\end{eqnarray*}

Here of course ${\rm d}^{D}\zeta(X,Y) = (D_X\zeta)(Y)- (D_Y\zeta)(X) - \zeta([X,Y])$.

Since $L$ is parallel for $\nabla^0$, it follows that $h_0(D_X(\zeta(u)), \zeta(Ju))=0$.
Observe now that the hypothesis \eqref{hyp:boundTf} and the Cauchy--Schwarz inequality  implies that 
$$
\vert h_0(\zeta(X),\de^D\zeta(Y,Z))\vert^2 \leq K_0^2 g(X,X)\cdot g(Y,Y) \cdot g(Z,Z)\ . 
$$ 
Thus, from the above inequality applied to the final three terms of the computation above of $2\beta(X) = 2g(\nabla_Xu,Ju)$ and using that $g(u,u) = g(Ju, Ju) = 1$, we see that 
$$
\Vert\beta(X)\Vert \leq \frac{3}{2} K_0 \sqrt{g(X,X)}\ .
$$
This concludes the proof. \end{proof}

\begin{lemma}\label{lem:dlam}
Let $g=\lambda^2 g_0$, and let $\beta_0$ be the connection form of the vector $u_0$ proportional to $u$ and of norm 1 for $g_0$. Then
$$
-({\rm d} \log \lambda) \circ J = \beta-\beta_0,
$$
\end{lemma}
\begin{proof}
The connection of $g=\lambda^2 g_0$ is given by 
$$
\nabla= D + ({\rm d}\log \lambda) \otimes \Id -({\rm d}\log \lambda\circ J) \otimes J \ ,
$$
where $D$ is the connection of $g_0$.  Thus, if $u$ is the vector field of norm 1 for $g$, then
$$
\beta(X)=g(\nabla_Xu,Ju) = g(D_X u, Ju) -{\rm d}\log \lambda (J X)
$$
Observe now that $\lambda v=v_0$ where $v_0$ has norm 1 for $g_0$, thus
$$
g(D_X v, Jv)=\frac{1}{\lambda^2} g(D_X v_0,Jv_0)=g_0(D_X v_0, Jv_0)=\beta_0(X)\ .
$$
The result follows.
	
\end{proof}

\begin{corollary}\label{coro:lambimm} 	Assuming $f$ is an immersion,
let $\gamma$ be  either
\begin{enumerate}
\item (Free boundary case) an embedded circle $\gamma$ in $\bf D$, or

	\item (Boundary case) or an embedded half circle  so that $\partial \gamma\subset \partial {\bf S}$. 
	\end{enumerate}
	Then
		$$
	\left\vert\int_\gamma ({\rm d} \log \lambda) \circ J\ \right\vert \leq K_2\int_\gamma \lambda {\rm d}s\ ,	$$ 
	where ${\rm d}s$ is the arc length of $\gamma$ with respect to $g_0$.
\end{corollary}
	
\begin{proof}
Let us consider first the free boundary case: from Lemma \ref{lem:dlam},
$$
	\int_\gamma ({\rm d} \log \lambda) \circ J=\int_\gamma\beta -\int_\gamma\beta_0=\int_\gamma\beta -\int_U {\rm d}\beta_0=\int_\gamma\beta\ ,$$
	where $U$ is the disk of boundary $\gamma$ in the boundary free case, and boundary $\gamma\sqcup I$, where $I\subset \partial {\bf S}$ in the boundary case. Observe that we have used here that 
	$$
	\int_I\beta_0=0\ ,
	$$
	which follows from the fact that $u_0$ is tangent to $\partial \S$, so that its covariant derivative in the tangential direction is also tangential and hence orthogonal to $Ju_0$. 
	Thus the inequality follows from the bounds in Lemma \ref{lem:bdcon}.
	For the  boundary case, we first have to remark that if $X\in \T\partial \S$, then obviously $\beta_0(X)=0$. Moreover, $f(\partial S)$ is a curve in a totally real and totally geodesic manifold $W$. Thus $\nabla X u$ belongs to $\T W$ and $Ju$ is orthogonal to $\T W$. Thus $$
	\beta(X)=g(\nabla_X u, Ju)=0\ .$$
	The fact that $\beta$ and $\beta_0$ are zero when restricted to $\partial S$ allows us to conclude the argument.
\end{proof}

\subsubsection{Proof of Theorem \ref{theo:imm-limi}}
The proof in both cases follow the same scheme. We will point out where the difference occurs.
It is enough to prove that $f_0^*\alpha$ does not vanish at $0$.
Let $g_k$ be the conformal metric on $\bf D$ given by
$$
g_k(u,v)=h\left(\alpha(\T f_k(u)),\alpha(\T f_k(u))\right)\ ,
$$
and $\lambda_k$ be the function on $\bf D$ so that $g_k=\lambda_k^2 g_0$ where $g_0$ is the Euclidean metric on $\bf D$. 

To prove the theorem, it is enough  is to find a positive $\rho$ so that, for all $k\in \N$
$$
\lambda_k(0)\geq \rho \ .
$$ 

Let $D(R)$ be the disk of radius $R$ centered at $0$ with respect to $g_0$ and $\gamma(R)$ its boundary. 
In the boundary case, we let $D(R)$ be the half disk centered at $0$ and denote by $\gamma(R)$ the half circle which is part of $\Fr(D(R))$. We denote by $(r,\theta)$ the polar coordinates on $\mathbb C\setminus{0}$. Let $\omega_\theta$ is the closed form on $\mathbb C\setminus\{0\}$ given by 
$$
\omega_\theta=\frac{x{\rm d} y-y{\rm d}x}{x^2+y^2}\ .
$$
(Of course, $\omega_{\theta}$ is usually denoted by ${\rm d\theta}$ but the notation ${\rm d\theta}$ --disliked by the first author -- suggests that $\omega_\theta$ is exact)  
Observe that for any 1-form $\alpha$

\begin{eqnarray*}
\alpha\wedge\omega&=&\alpha(\partial_r)\ {\rm d}r\wedge \omega_\theta ,\\
\int_{S(r)} \alpha\circ J &=& \int_{S(r)} \alpha (J\partial_\theta)\ \omega_\theta= -r\int_{S(r)}\alpha(\partial_r)\  \omega_\theta\ .
\end{eqnarray*}

For $k\in\N$, let $G_k$ be the function $\mathbb R_{>0}$ given by 
$$
G_k(R)\defeq\int_{\partial D(R)}\log(\lambda_k)\cdot \omega_\theta.$$
Observe that, because $\lambda_k$ does not vanish, we map apply Stokes theorem to the annulus $\{\eta < r , R\}$, and let $\eta \to 0$ to obtain
\begin{eqnarray*}
G_k(R)=\int_{D(R)} {\rm d}\log(\lambda_k)\wedge\omega_\theta + G_k(0)=
\int_0^R\left(\int_{\gamma(r)} {\rm d}\log(\lambda_k)(\partial_r)\cdot \omega_\theta\right){\rm d}r + G_k(0) .
\end{eqnarray*}
After taking the derivatives with respect to $R$, we get
\begin{eqnarray*}
\dot{G}_k(R)=	\int_{\gamma(R)} {\rm d}\log(\lambda_k)(\partial_r)\cdot \omega_\theta=-\frac{1}{R}\int_{\gamma(R)} {\rm d}\log(\lambda_k)\circ J \ .
\end{eqnarray*}
Since $f_k$ is an immersion, it follows from Corollary \ref{coro:lambimm} that
\begin{eqnarray*}
\vert\dot{G}_k(R)\vert \leq \frac{K_0}{R} \int_{\gamma(R)}\lambda_k {\de}s \ ,\end{eqnarray*}
where $\de s$ is the length with respect to $g_0$. By the Schwarz lemma, we see that $\lambda_k$ is uniformly bounded from above and thus
$ \vert\dot{G}_k(R)\vert \leq  C_2$.
It follows that for all $R_0$
$$
\epsilon\pi\vert \log \lambda_k(0)\vert =\vert G_k(0)\vert \leq \vert G_k(R_0)\vert +C_2 R_0 \ ,
$$
where $\epsilon=2$ in the free boundary case, and $\epsilon=1$ in the boundary case.
Thus
$$
\lambda_k(0)\geq \exp\left(-\frac{1}{\epsilon\pi}\left( \vert G_k(R_0)\vert +C_0 R_0\right)\right).
$$
Then, since $\lambda_\infty$ has isolated zeroes (See the similarity principle Theorem A.5.2 and Proposition A.5.3 in \cite{Abbas:2019}),  there exists some $R_0$ so that $\gamma(R_0)$ does not contain any zeroes of $\lambda_\infty$, thus there exists a positive $\alpha$ so that

$$
\int_{\partial D(R_0)}\log(\lambda_{\infty})\cdot \omega_\theta > -\alpha>-\infty\ .
$$
Thus for $k$ large enough,
$
G_k(R_0)\geq -2\alpha$.
In particular, for $k$ large enough
$$
\lambda_k(0)\geq \mu\defeq\exp\left(-\frac{1}{\epsilon\pi}(2\alpha+C_0R_0)\right)>0\ .
$$
The result follows.

\bibliographystyle{amsplain}
\providecommand{\bysame}{\leavevmode\hbox to3em{\hrulefill}\thinspace}
\providecommand{\MR}{\relax\ifhmode\unskip\space\fi MR }
\providecommand{\MRhref}[2]{%
  \href{http://www.ams.org/mathscinet-getitem?mr=#1}{#2}
}
\providecommand{\href}[2]{#2}

\end{document}